\documentclass[a4paper,reqno]{amsart}
\usepackage{amsmath}
\usepackage{amssymb}
\usepackage{hyperref}
\usepackage[margin=30mm]{geometry}
\usepackage{zref-clever}
\usepackage{mathrsfs}
\usepackage{tikz}
\usepackage{graphicx}
\graphicspath{{./anc/}}
\usepackage{a4wide}
\usepackage{enumerate}
\usetikzlibrary{calc}
\usepackage{tikz-cd}
\usepackage{adjustbox}
\usepackage{verbatim}
\hypersetup{
	colorlinks=true,
	linkcolor=blue,
	citecolor=magenta,      
	urlcolor=cyan,
	pdftitle={The Casson-Sullivan invariant for homeomorphisms of 4-manifolds},
}

\zcsetup{cap,countertype={lem=lemma},countertype={prop=proposition},countertype={thm=theorem},countertype={cor=corollary}, countertype={dfn=definition}}

\AddToHook{env/thm/begin}{%
	\zcsetup{countertype={thm=theorem}}}
\AddToHook{env/lem/begin}{%
	\zcsetup{countertype={thm=lemma}}}
\AddToHook{env/cor/begin}{%
	\zcsetup{countertype={thm=corollary}}}
\AddToHook{env/prop/begin}{%
	\zcsetup{countertype={thm=proposition}}}
\AddToHook{env/dfn/begin}{%
	\zcsetup{countertype={thm=definition}}}
\AddToHook{env/rmk/begin}{%
	\zcsetup{countertype={thm=remark}}}
\AddToHook{env/que/begin}{%
	\zcsetup{countertype={thm=question}}}
\AddToHook{env/stp/begin}{%
	\zcsetup{countertype={thm=setup}}}

\newtheorem{thm}{Theorem}[section]
\newtheorem{prop}[thm]{Proposition}
\newtheorem{lem}[thm]{Lemma}
\newtheorem{cor}[thm]{Corollary}

\zcRefTypeSetup{question}{
	Name-sg = Question ,
	name-sg = questions ,
	Name-pl = Questions ,
	name-pl = questions ,
}

\zcRefTypeSetup{remark}{
	Name-sg = Remark ,
	name-sg = remark ,
	Name-pl = Remarks ,
	name-pl = remarks ,
}

\zcRefTypeSetup{setup}{
	Name-sg = Setup ,
	name-sg = setup ,
	Name-pl = Setups ,
	name-pl = setups ,
}

\theoremstyle{definition}
\newtheorem{dfn}[thm]{Definition}
\newtheorem{exa}[thm]{Example}

\theoremstyle{remark}
\newtheorem{rmk}[thm]{Remark}
\newtheorem{stp}[thm]{Setup}

\numberwithin{equation}{section}

\newcounter{counterd}

\newcommand{\R}{\mathbb{R}}
\newcommand{\Q}{\mathbb{Q}}
\newcommand{\Z}{\mathbb{Z}}

\newcommand{\Lg}{\mathbb{L}}

\newcommand{\sm}{\setminus}

\newcommand{\ol}{\overline}

\DeclareMathOperator\Orb{Orb}
\DeclareMathOperator\Arf{Arf}
\DeclareMathOperator\Id{Id}
\DeclareMathOperator\cs{cs}
\DeclareMathOperator\pt{pt}

\DeclareMathOperator\Stab{Stab}
\DeclareMathOperator\Homeo{Homeo}

\DeclareMathOperator\CAT{CAT}
\DeclareMathOperator\DIFF{DIFF}
\DeclareMathOperator\Diff{Diff}
\DeclareMathOperator\TOP{TOP}
\DeclareMathOperator\PL{PL}

\DeclareMathOperator\tO{O}
\DeclareMathOperator\ks{ks}

\DeclareMathOperator\pr{pr}
\DeclareMathOperator\red{red}
\DeclareMathOperator\Image{im}
\DeclareMathOperator\coker{coker}
\DeclareMathOperator\Sq{Sq}
\DeclareMathOperator\ab{ab}
\DeclareMathOperator\Mic{Mic}
\DeclareMathOperator\sig{sig}
\DeclareMathOperator\PD{PD}
\DeclareMathOperator\RS{RS}
\DeclareMathOperator{\BTOP}{\B\!\TOP}
\DeclareMathOperator{\BCAT}{\B\!\CAT}
\DeclareMathOperator{\BO}{\B\!\tO}
\DeclareMathOperator{\BG}{\B G}

\newcommand{\B}{\mathcal{B}}

\title[The Casson-Sullivan invariant for homeomorphisms of $4$-manifolds]{The Casson-Sullivan invariant for homeomorphisms of $4$-manifolds}
\author{Daniel A.P. Galvin}
\address{University of Texas at Austin, US}
\email[Daniel A.P. Galvin]{daniel.galvin@austin.utexas.edu}

\def\subjclassname{\textup{2020} Mathematics Subject Classification}
\expandafter\let\csname subjclassname@1991\endcsname=\subjclassname

\subjclass{
	57K40, 
	57R10
}
\keywords{4-manifolds, surgery theory, pseudo-isotopies, smoothing theory}

\begin{document}
	
\begin{abstract}
	We investigate the realisability of the Casson-Sullivan invariant for homeomorphisms of smooth $4$-manifolds, which is the obstruction to a homeomorphism being stably pseudo-isotopic to a diffeomorphism, valued in the third cohomology of the source manifold with $\Z/2$-coefficients.  We prove that for all pairs of orientable, homeomorphic, smooth $4$-manifolds this invariant can be realised fully after stabilising with a single $S^2\times S^2$.  As an application, we obtain that topologically isotopic surfaces in a smooth, simply-connected $4$-manifold become smoothly isotopic after sufficient external stabilisations.  We further demonstrate cases where this invariant can be realised fully without stabilisation for self-homeomorphisms, which includes for manifolds with finite cyclic fundamental group.  This method allows us to produce many examples of homeomorphisms which are not stably pseudo-isotopic to any diffeomorphism but are homotopic to the identity.  Finally, we reinterpret these results in terms of finding examples of smooth structures on $4$-manifolds which are diffeomorphic but not stably pseudo-isotopic.
\end{abstract}
	
\maketitle

\section{Introduction}

\subsection{Results}

The Casson-Sullivan invariant of a homeomorphism $f\colon X \to X'$ between compact, smooth $4$-manifolds $X$ and $X'$, written $\cs(f)\in H^3(X,\partial X;\Z/2)$, is an obstruction to $f$ being isotopic to a diffeomorphism relative to the boundary.  It is defined as the Kirby-Siebenmann invariant of the mapping cylinder of~$f$.  In \zcref{sbs:casson_sullivan_basic_properties} we will establish the properties of this invariant, but we give a few briefly here before stating the results.  We have that $\cs(f)=0$ if $f$ is pseudo-isotopic to a diffeomorphism of $X$ (i.e.\ if $f$ is pseudo-smoothable) and further that $\cs(f)=\cs(g)$ if $f$ is pseudo-isotopic to $g$ (\zcref{prop:cs_pseudo-isotopy_inv}).  In fact, $\cs$ defines a crossed homomorphism $\widetilde{\pi}_0\Homeo(X,\partial X)\to H^3(X,\partial X;\Z/2)$, where $\widetilde{\pi}_0\Homeo(X,\partial X)$ denotes the pseudo-mapping class group of $X$ (\zcref{prop:cs_homomorphism}).

We now state the results of the paper.  We start with the fact that, for orientable manifolds, the Casson-Sullivan invariant is fully realisable stably.

\begin{thm}\label{thm:stable_realisation_theorem}
	Let $X$ and $X'$ be compact, connected, smooth, orientable $4$-manifolds such that $X\cong X_0\#(S^2\times S^2)$ and $X'\cong X'_0\# (S^2\times S^2)$ where $X_0\approx X'_0$, and let $\eta\in H^3(X,\partial X;\Z/2)$.  Then there exists a homeomorphism $f\colon X\to X'$ with $\cs(f)=\eta$.
\end{thm}

\zcref{thm:stable_realisation_theorem} will follow immediately from the more technical \zcref{thm:stable_realisation_techincal}.  As an application of \zcref{thm:stable_realisation_techincal} we obtain the following result concerning stable isotopy of surfaces in $4$-manifolds.

\begin{thm}\label{thm:stable_isotopy}
	Let $X$ be a connected, compact, simply-connected, smooth $4$-manifold and let $\Sigma_1,\Sigma_2\subset X$ be a pair of smoothly, properly embedded surfaces which are topologically isotopic relative to their boundaries.  Then there exists $n\geq 0$ such that $\Sigma_1$ and $\Sigma_2$ are smoothly isotopic relative to their boundaries in $X\#(\#^n S^2\times S^2)$, where the connected-sums are taken away from $\Sigma_1\cup \Sigma_2$.
\end{thm}

Results similar to \zcref{thm:stable_isotopy} have been referred to previously (see the introductions in \cite{auckly_kim_melvin_ruberman_2014,auckly_kim_melvin_ruberman_schwartz_2019,hayden_kang_mukherjee}), but, to the best knowledge of the author, a proof of a result like this has never appeared in the literature.  In the above references, it seems to be implicit that the complement of the surfaces be simply-connected, but we will need no such condition.

We make some remarks on these two theorems.  First, two remarks on \zcref{thm:stable_realisation_theorem}.

\begin{rmk}
	\zcref{thm:stable_realisation_theorem}, combined with \zcref{prop:cs_stable_2} (due to Freedman-Quinn), recovers Gompf's result \cite{gompf_1984} that homeomorphic compact, connected, smooth, orientable $4$-manifolds are stably diffeomorphic.  The same reasoning also shows that \zcref{thm:stable_realisation_theorem} cannot be extended to the non-orientable case, as there exist compact, connected, smooth, non-orientable $4$-manifolds which are homeomorphic but are not stably diffeomorphic (see \cite{kreck_1984}, \cite{cappell_shaneson_1976}).  If one considers only self-homeomorphisms of non-orientable smooth manifolds, then a result like \zcref{thm:stable_realisation_theorem} may still hold.
\end{rmk}

\begin{rmk}
	For the class of $4$-manifolds considered in \zcref{thm:stable_realisation_theorem}, we cannot remove the stabilisation assumption in \zcref{thm:stable_realisation_theorem}.  For example, the Casson-Sullivan invariant is not realisable for $X=S^1\times S^3=X'$ (see \zcref{lem:sw_s1xs3}).  
\end{rmk}

And a remark on \zcref{thm:stable_isotopy}.

\begin{rmk}\label{rmk:stable_isotopy}
	If the surface exteriors have trivial or cyclic fundamental groups, we can pair \zcref{thm:stable_isotopy} with results concerning when homologous surfaces of the same genus whose exteriors have isomorphic fundamental groups are (stably) topologically isotopic (see \cite{lee_willczynski_1990,hambleton_kreck_1993,lee_wilczynski_1993, boyer_1993}). This gives, in some cases, that smoothly embedded homologous surfaces with the same boundaries and of the same genus and whose exteriors have isomorphic fundamental groups are stably smoothly isotopic.  See \zcref{cor:stable_isotopy} for the precise statement.  This (in some sense) generalises the result in \cite{auckly_kim_melvin_ruberman_schwartz_2019}, although the result there is stronger in the sense that only one stabilisation is required.
\end{rmk}

The methods we use to prove \zcref{thm:stable_realisation_theorem} lead us the following interesting example, which demonstrates that the smoothability of a self-homeomorphism of a smooth manifold depends on the isotopy class of the smooth structure (see \zcref{sec:isotopy_smooth_structures}).

\begin{cor}\label{cor:swap_not_smoothable}
	Let $X=(S^1\times S^3)\#(S^1\times S^3)\#(S^2\times S^2)$ with the standard smooth structure and let $g\colon X\to X$ be the diffeomorphism which swaps the two $S^1\times S^3$ connected-summands and is the identity on the $S^2\times S^2$ connected-summand.  Then there exists a smooth structure $\mathscr{S}'$ on $X$, which is diffeomorphic to the standard smooth structure, but is such that $g$ is not stably pseudo-smoothable with respect to $\mathscr{S}'$.  In particular, $g$ is not smoothable with respect to $\mathscr{S}'$ ($g$ is not isotopic to a diffeomorphism).
\end{cor}

We now restrict to considering self-homeomorphisms.  The next result concerns a class of $4$-manifolds where it is possible to remove the stabilisation assumption in \zcref{thm:stable_realisation_theorem}.  This theorem will be stated in terms of a certain ``realisability condition", which is defined fully in \zcref{dfn:cs_realisation_condition}.  Instead of giving an informal definition of this condition, we will state some classes of $4$-manifolds for which it applies.

The Casson-Sullivan realisability condition is satisfied for $4$-manifolds whose fundamental groups are in the following classes.
\begin{enumerate}[(i)]
	\item Finite cyclic groups $\Z/n$.
	\item Groups of the form $\Z/(2^n)\times \Z/2$.
	\item Groups of the form $(\Z/2)^n$.
	\item Dihedral groups $D_n$.
\end{enumerate}

The above list is not exhaustive (for more information see \zcref{sbs:realisation_condition_examples}).  It should also be noted that the condition holds for all manifolds with odd order fundamental groups, however, this class is not interesting as in these cases $H^3(X,\partial X;\Z/2)=0$ and hence the Casson-Sullivan invariant is trivially realisable.

\begin{thm}\label{thm:unstable_realisation_theorem}
	Let $X$ be a compact, connected, smooth, orientable $4$-manifold with $\pi_1(X)$ a good group such that $X$ satisfies the Casson-Sullivan realisability condition $($\zcref{dfn:cs_realisation_condition}$)$.  Then for every class $\eta\in H^3(X,\partial X;\Z/2)$ there exists a homeomorphism $f\colon X\to X$ with $\cs(f)=\eta$.
\end{thm}

\begin{rmk}\label{rmk:good}
	By `good' in \zcref{thm:unstable_realisation_theorem} we mean in the sense of Freedman-Quinn (see \cite[Chapter 2.9]{freedman_quinn_1990} or \cite[\S19]{behrens_kalmar_kim_powell_ray_2021} for a definition).  It is known that the set of good groups includes elementary amenable groups, as well as groups of sub-exponential growth.  In particular, all finite groups are good, so the groups listed above as satisfying the Casson-Sullivan realisability condition are also good.  However, it should be noted that we do not show that the realisability condition implies that the group is good.
\end{rmk}

As a corollary to the proof of this theorem we obtain the following.

\begin{thm}\label{thm:non_pseudo_smoothable_but_hom_smoothable}
	Let $X$ be a compact, connected, smooth, orientable $4$-manifold with $\pi_1(X)$ a good group such that $X$ satisfies the Casson-Sullivan realisability condition $($\zcref{dfn:cs_realisation_condition}$)$.  Then there exists a family of homeomorphisms $\{f_\eta\colon X\to X\mid 0\neq\eta \in H^3(X,\partial X;\Z/2)\}$ all distinct up to pseudo-isotopy (relative to the boundary) such that each element $f_\eta$ is not stably pseudo-smoothable but each~$f_\eta$ is homotopic to the identity map.
\end{thm}

\begin{rmk}
	We quickly demonstrate that the class of manifolds for which \zcref{thm:non_pseudo_smoothable_but_hom_smoothable} applies non-trivially is non-empty.  Recall that for any finitely generated group $G$ there exists a closed, connected, smooth, oriented $4$-manifold $X$ with $\pi_1(X)\cong G$. For $n\geq 2$ let $X$ be such a $4$-manifold for $G=\Z/n$ (which is in case (i) above).  For $n$ even we have that $H^3(X;\Z/2)\cong \Z/2$ and hence \zcref{thm:non_pseudo_smoothable_but_hom_smoothable} applies non-trivially to $X$.
\end{rmk}

We will also discuss the consequences of \zcref{thm:stable_realisation_theorem} and \zcref{thm:unstable_realisation_theorem} in terms of non-isotopic but diffeomorphic smooth structures on $4$-manifolds.  All of the concepts will be defined in \zcref{sec:isotopy_smooth_structures}.  We have the following theorems.

\begin{thm}\label{thm:_stable_non_stably_pseudo_iso}
	Let $X$ be a $4$-manifold such that $X\cong X'\#(S^2\times S^2)$ for some compact, connected, smooth, orientable $4$-manifold $X'$ and let $\mathscr{S}$ denote the smooth structure on $X$.  Then for every non-zero $\eta \in H^3(X,\partial X;\Z/2)$ there exists a smooth structure $\mathscr{S}_\eta$ on $X$ which is not stably pseudo-isotopic (relative to the boundary) to $\mathscr{S}$ but $X_\mathscr{S}$ is diffeomorphic to $X_{\mathscr{S}_\eta}$.  Furthermore, the elements of this family of smooth structures $\{\mathscr{S}_\eta\}$ are pairwise distinct up to stable pseudo-isotopy.
\end{thm}

\begin{thm}\label{thm:_unstable_non_stably_pseudo_iso}
	Let $X$ be a compact, connected, smooth, orientable $4$-manifold with $\pi_1(X)$ a good group which satisfies the Casson-Sullivan realisation condition $($\zcref{dfn:cs_realisation_condition}$)$.  Let $\mathscr{S}$ denote the smooth structure on $X$.  Then for every non-zero $\eta \in H^3(X,\partial X;\Z/2)$ there exists a smooth structure $\mathscr{S}_\eta$ on $X$ which is not stably pseudo-isotopic to $\mathscr{S}$ but $X_\mathscr{S}$ is diffeomorphic to $X_{\mathscr{S}_\eta}$.  Furthermore, the elements of this family of smooth structures $\{\mathscr{S}_\eta\}$ are pairwise distinct up to stable pseudo-isotopy.
\end{thm}

\subsection{Background}
\subsubsection{Non-smoothable homeomorphisms}
Non-smoothable homeomorphisms have been studied extensively for simply-connected $4$-manifolds.  The first examples were due to Friedman-Morgan \cite{friedman_morgan_1988} on Dolgachev surfaces, detected using Donaldson invariants.  Since then, many examples have been found, all detected using gauge-theoretic invariants.  For a brief survey of these results, see \cite[Section 1.3]{galvin_ladu_2025}.

Outside of the simply-connected case, much less was known.  Cappell-Shaneson-Lee \cite{cappell_shaneson_1971,lee_1970} and Scharlemann-Akbulut \cite{scharlemann_1976,akbulut_1999} produced examples of homotopy equivalences $f\colon (S^1\times S^3)\# (S^2\times S^2)\to (S^1\times S^3)\#(S^2\times S^2)$ which are not homotopic to diffeomorphisms.  It was shown by Wang \cite[Chapter 6.2]{wang_93} that the Cappell-Shaneson-Lee construction could be improved to finding a non-smoothable homeomorphism $f\colon (S^1\times S^3)\# (S^2\times S^2)\to (S^1\times S^3)\#(S^2\times S^2)$ (the reader should note that Wang states that this implies the existence of an exotic smooth structure on $(S^1\times S^3)\#(S^2\times S^2)$, when it actually gives a non-isotopic but diffeomorphic smooth structure; see \zcref{sec:isotopy_smooth_structures} for more information).  In \zcref{sec:stable_realisation} we will show that this homeomorphism has non-trivial Casson-Sullivan invariant.
\subsubsection{History of the Casson-Sullivan invariant}
The Manifold Hauptvermutung is the following conjecture: `every homeomorphism $f\colon M \to N$ between two $\PL$ manifolds is homotopic to a $\PL$-homeomorphism'.  There is also the related conjecture, called the Isotopy Manifold Hauptvermutung: `every homeomorphism $f\colon M \to N$ between two $\PL$ manifolds is isotopic to a $\PL$-homeomorphism'.  The work of Casson and Sullivan \cite{casson_1967,sullivan_1967} showed that the Isotopy Manifold Hauptvermutung is true for simply-connected $n$-manifolds of dimension $n\geq 5$ provided that $H^3(N;\Z/2)=0$, by considering a certain obstruction class~$\omega\in H^3(N;\Z/2)$.  In fact, they also showed that the (homotopy) Manifold Hauptvermutung is true for simply-connected $n$-manifolds of dimension $n\geq 5$ provided that $H^4(N;\Z)$ contains no $2$-torsion.  Later work of Kirby-Siebenmann \cite{kirby_siebenmann_1977}, crucially using the classification of $\PL$ homotopy tori by Hsiang-Shaneson and Wall \cite{hsiang_shaneson_1971,wall_1970}, showed that the Manifold Hauptvermutung is false in general in all dimensions $n\geq 5$, and, in particular, showed that the Casson-Sullivan class is precisely the obstruction to the Isotopy Manifold Hauptvermutung.  This also gave a very slick definition of the Casson-Sullivan invariant of a homeomorphism as the Kirby-Siebenmann invariant of the mapping cylinder of the given homeomorphism.

The Casson-Sullivan invariant can similarly be defined for $4$-manifolds, where, since there is no difference between the $\PL$ and smooth categories (for our purposes), we can consider it as an obstruction to smoothing homeomorphisms.  Freedman-Quinn \cite{freedman_quinn_1990} showed that the Casson-Sullivan invariant is the unique obstruction to stably pseudo-smoothing a homeomorphism (see \zcref{prop:cs_stable_2}).

\subsection{Outline}
The contents of this paper are organised as follows.  In \zcref{sec:invariant_properties} we define the Casson-Sullivan invariant and establish its fundamental properties, including proving a certain ``connected-sum along a circle" formula for it (\zcref{thm:connect_sum}).  In \zcref{sec:stable_realisation} we prove the technical stable realisation result \zcref{thm:stable_realisation_techincal}, which will immediately imply \zcref{thm:stable_realisation_theorem}.  We then prove \zcref{cor:swap_not_smoothable} in \zcref{sbs:interesting_example}.  In \zcref{sec:unstable_realisation} we prove the unstable realisation result, \zcref{thm:unstable_realisation_theorem}, and discuss the consequence, \zcref{thm:non_pseudo_smoothable_but_hom_smoothable}.  We will also discuss some partial realisation results in \zcref{sbs:partial_realisation}.  \zcref{sec:stable_isotopy} concerns stable smooth isotopy of surfaces, and is where we will prove \zcref{thm:stable_isotopy}, as an application of \zcref{thm:stable_realisation_techincal}.  In \zcref{sec:isotopy_smooth_structures} we discuss how to interpret these results in terms of finding non-isotopic but diffeomorphic smooth structures on $4$-manifolds.

\subsection{Acknowledgements}
I would like to thank Mark Powell for suggesting this project, and for his generous assistance throughout the process of writing this paper.  I would also like to thank Stefan Friedl, Matthias Kreck, Alexander Kupers, Ana Lecuona, Csaba Nagy, Patrick Orson, Arunima Ray, Daniel Ruberman, Peter Teichner and Simona Vesel\'{a} for various helpful conversations and comments.  I am also very grateful to the anonymous referee for his/her thorough report and pertinent comments.  I am grateful to the Max Planck Institute for Mathematics in Bonn, where I was staying during a portion of my writing of this paper, and also to the University of Glasgow, whom I was financially supported by.

\subsection{Notation}
Some notation before we begin.  For $X$ and $Y$ (smooth) manifolds we write $X\cong Y$ when $X$ is diffeomorphic to $Y$, $X\approx Y$ when $X$ is homeomorphic to $Y$, and $X\simeq Y$ when $X$ is homotopy equivalent to $Y$.  We will denote the connected-sum operation on manifolds by $\#$, and the $k$-fold connected-sum of a manifold $X$ as $\#^kX$.  For a submanifold $\Sigma\subset X$ we will denote its open tubular neighbourhood by $\nu\Sigma$, and its closed tubular neighbourhood by $\ol{\nu}\Sigma$.  For a pair of spaces $(K_1,K_2)$ and a based space $S$ with basepoint $\ast$ we will write $[(K_1,K_2),(S,\ast)]$ for the set of (relative) homotopy classes of maps of pairs.  We will often omit the basepoint and simply write $[(K_1,K_2),S]$.

\section{The Casson-Sullivan invariant}\label{sec:invariant_properties}
In this section we define the Casson-Sullivan invariant and prove some of its fundamental properties.  The Casson-Sullivan invariant is named for Andrew Casson and Dennis Sullivan, and the canonical reference is the collection of papers edited by Andrew Ranicki \cite{ranicki_1996}.

We begin in \zcref{sbs:microbundles_classifying_spaces} by stating the relevant theory about microbundles and classifying spaces which we will need to define the invariant in \zcref{sbs:casson_sullivan_definition}.  We then establish the Casson-Sullivan invariant's fundamental properties in \zcref{sbs:casson_sullivan_basic_properties}, before finishing this section by proving a ``connected-sum along a circle" formula for the invariant in \zcref{sbs:casson_sullivan_connect_sum}.

\subsection{Microbundles and classifying spaces}\label{sbs:microbundles_classifying_spaces}

Let $\TOP(k)=\{g\colon\R^k\xrightarrow{\approx}\R^k\vert f(0)=0\}$ and let $\tO(k)$ be the group of orthogonal $k$-dimensional matrices.  Then there are obvious inclusions $\TOP(k)\hookrightarrow \TOP(k+1)$ and $\tO(k)\hookrightarrow \tO(k+1)$, and we denote the corresponding direct limits as $\TOP$ and $\tO$, respectively.  We will use the notation $\CAT$ to stand in for $\TOP$ and $\tO$.  The classifying spaces $\BTOP$ and $\BO$ then classify stable $\R^n$ fibre bundles and stable vector bundles, respectively.  The universal stable vector bundle has an underlying stable topological $\R^n$ fibre bundle and its classifying map will be denoted as $\xi\colon \BO\to\BTOP$.  Similarly, let $\mathcal{B}_{\DIFF}$ and $\mathcal{B}_{\TOP}$ denote the classifying spaces of stable $\DIFF$-microbundles and stable $\TOP$-microbundles, respectively (see {\cite[Essay IV, \S10]{kirby_siebenmann_1977}}), and let $\xi'$ denote the classifying map of the universal $\DIFF$-microbundle's underlying $\TOP$-microbundle.  
\begin{lem}\label{lem:bundle_to_microbundle}
	Let $\CAT$ stand in for $\TOP$ or $\tO$.  Let $u_{\CAT}\colon \BCAT\to \B_{\CAT}$ denote the classifying map of the universal stable $\CAT$-bundle's underlying stable $\CAT$-microbundle.  Then $u_{\CAT}$ is a homotopy equivalence.
\end{lem}

\begin{proof}
	The Kister-Mazur theorem \cite{kister_1964, kuiper_lashof_1966, siebenmann_guillou_hahl_1973} gives that isomorphism classes of stable $\CAT$ bundles over a CW-complex $X$ are in one-to-one correspondence with isomorphism classes of stable $\CAT$-microbundles over $X$.  This means that there is a natural bijection $\kappa_{\CAT}\colon[X,\BCAT]\leftrightarrow [X,\mathcal{B}_{\CAT}]$ defined as the composition of the natural bijections
	\[
	[X,\BCAT]\to \CAT(X) \to \Mic_{\CAT}(X) \to [X,\B_{\CAT}]
	\] where $\CAT(X)$ denotes the set of isomorphism classes of stable $\CAT$-bundles over $X$ (i.e.\ stable $\R^n$ fibre bundles if $\CAT=\TOP$ and stable vector bundles if $\CAT=\DIFF$) and $\Mic_{\CAT}(X)$ denotes the set of isomorphism classes of stable $\CAT$-microbundles over $X$.  In fact, by the definition we can conclude that $\kappa_{\CAT}=(u_{\CAT})_*$.  Then naturality gives the following commutative diagram.
	\[
	\begin{tikzcd}
		{[\BCAT,\BCAT]} \arrow[r,"(u_{\CAT})_*"] & {[\BCAT,\mathcal{B}_{\CAT}]} \\
		{[\mathcal{B}_{\CAT},\BCAT]} \arrow[r,"(u_{\CAT})_*"] \arrow[u,"(u_{\CAT})^*"] & {[\mathcal{B}_{\CAT},\mathcal{B}_{\CAT}]} \arrow[u,"(u_{\CAT})^*"]
	\end{tikzcd}
	\]
	By a simple diagram chase one can see that there exists an element $f\in[\B_{\CAT},\BCAT]$ such that $(u_{\CAT})^*(f)=\Id_{\BCAT}$ and such that $(u_{\CAT})_*(f)=\Id_{\B_{\CAT}}$.  Hence $f$ is the homotopy inverse of $u_{\CAT}$, completing the proof.
\end{proof}

We have the following square which is a homotopy pullback\[
\begin{tikzcd}
	\BO \arrow[r,"u_{\DIFF}"] \arrow[d,"\xi"] & \B_{\DIFF} \arrow[d,"\xi'"] \\
	\BTOP \arrow[r,"u_{\TOP}"] & \B_{\TOP}
\end{tikzcd}
\]and it follows that the homotopy fibres of $\xi$ and $\xi'$ are homotopy equivalent.  We will denote this space as $\TOP\!/\!\tO$.  Boardman-Vogt \cite{boardman_vogt_1968} showed that we can `deloop' this fibre to obtain a space $\mathcal{B}(\TOP\!/\!\tO)$ and that we can extend $\xi$ to the right to obtain the fibration sequence
\[
\TOP\!/\!\tO \to \mathcal{B}\!\tO\xrightarrow{\xi} \mathcal{B}\!\TOP \xrightarrow{p} \mathcal{B}(\TOP\!/\!\tO).
\]  
By the above exposition, we can dispense with considering $\B_{\TOP}$ and $\B_{\DIFF}$ and we will consider the classifying map of a $\TOP$-microbundle to be a map to $\BTOP$ and the classifying map of a $\DIFF$-microbundle to be a map to $\BO$.  This is helpful as much is known about the homotopy types of $\BO$, $\BTOP$ and $\TOP\!/\!\tO$, whereas it is useful to work solely with $\CAT$-microbundles, rather than passing between microbundles and vector bundles depending on the category.

\subsection{Definition of the Casson-Sullivan invariant}\label{sbs:casson_sullivan_definition}

Let $X$ be an $n$-dimensional topological manifold with (potentially empty) boundary $\partial X$.  Further assume that we already have a smooth structure on the boundary $\partial X$.  The manifold $X$ admits a stable topological tangent microbundle $t_X\colon X\to\BTOP$.  The first obstruction for $X$ to be smoothable is to be able to lift $t_X$ to a stable $\DIFF$-microbundle $\tau_X\colon X\to \mathcal{B}\!\tO$ which extends the lift $\tau_{\partial X}$ which is already defined on $\partial X$, which we have since we already supposed the existence of a smooth structure on the boundary.  We have the following diagram
\begin{equation*}\label{dia:ks_def}
\begin{tikzcd}
	& & \mathcal{B}\!\tO \arrow[d,"\xi"] \\
	\partial X \arrow[urr,"\tau_{\partial X}", bend left] \arrow[r] & X \arrow[ur, dotted, "\tau_X"] \arrow[r,"t_X"'] & \mathcal{B}\!\TOP \arrow[d,"p"] \\
	& & \mathcal{B}(\TOP\!/\!\tO) 
\end{tikzcd}
\end{equation*}
It is known \cite[Essay IV, \S 10.12]{kirby_siebenmann_1977} that there is a $6$-connected map $\TOP\!/\!\tO\to K(\Z/2,3)$, and this means we can consider the composite as $p\circ t_X$ as a map $X\to K(\Z/2,4)$ since $X$ is $4$-dimensional.  The lift $\tau_{\partial X}$ corresponds to a null-homotopy $h_t(\tau_{\partial X})$ of $p\circ (t_{X}\!\vert_{\partial X})$ and, since the inclusion $\partial X\to X$ is a cofibration, this homotopy extends to a homotopy \[\widetilde{h}_t(\tau_{\partial X})\colon X\times I \to \B (\TOP\!/\!\tO)\] such that $\widetilde{h}_0(\tau_{\partial X})=p\circ t_X$ and such that $\widetilde{h}_1(\tau_{\partial X})$ defines an element \[\widetilde{h}_1(\tau_{\partial X})\in[(X,\partial X),(K(\Z/2,4),\ast)]\cong H^4(X,\partial X;\Z/2).\]

\begin{dfn}\label{dfn:kirby_siebenmann_invariant}
	Let $(X,\partial X)$ be as above.   We define the \emph{Kirby-Siebenmann invariant} \[\ks(X,\partial X):=\widetilde{h}_1(\tau_{\partial X})\in H^4(X,\partial X;\Z/2).\]
\end{dfn}

\begin{thm}[{\cite[Essay IV, \S10]{kirby_siebenmann_1977}, \cite[Corollary 8.3D]{freedman_quinn_1990}}]\label{thm:ks}
	The stable tangent microbundle of $X^n$ for $4\leq n\leq 7$, written as $t_X\colon X\to \mathcal{B}\TOP$, lifts to a stable tangent bundle $\tau_X$ extending the already specified lift on $\partial X$ if and only if $\ks(X)=0\in H^4(X,\partial X;\Z/2)$. 
\end{thm}

\begin{rmk}
	In the high-dimensional case ($n\geq 5$) this is the first in a sequence of obstructions (and for $n=5,6,7$ it is the only one), and the vanishing of all of these obstructions implies the existence of a smooth structure on $X$ extending the given one on $\partial X$.  However, for $n=4$ we do not have the corresponding geometric outcome if $\ks(X)=0$.  For example, the manifold $E_8\# E_8$ has vanishing Kirby-Siebenmann invariant but does not admit a smooth structure.  Instead, one gets that $\ks(X)=0$ implies that there exists some $k\geq 0$ such that $X \#^k (S^2\times S^2)$ admits a smooth structure \cite[Section 8.6]{freedman_quinn_1990}.
\end{rmk}

\begin{rmk}\label{rmk:PL}
	The Kirby-Siebenmann invariant is usually defined as the obstruction to lifting the stable $\TOP$-microbundle to a stable $\PL$-microbundle, where $\PL$ denotes the piecewise linear category.  However, in the dimensions that we will examine there is no difference between the $\DIFF$ and $\PL$ categories, so we will ignore this.  Indeed, there is a $6$-connected map from $\TOP\!/\!\tO$ to $\TOP/\PL$, and $\TOP/\PL$ is homotopy equivalent to $K(\Z/2,3)$.  We will say no more about the $\PL$ category in this paper.
\end{rmk}

\begin{dfn}
	Let $X$ be a $4$-manifold with (potentially empty) boundary $\partial X$.  We say that $X$ is \emph{formally smoothable} if there exists a lift of the stable tangent microbundle $t_X\colon X\to \BTOP$ to a stable $\DIFF$-microbundle $\tau_X\colon X\to \BO$ extending the already specified lift on $\partial X$.  We call any such lift a \emph{formal smooth structure}.  Equivalently (by \zcref{thm:ks}), $X$ is formally smoothable if the Kirby-Siebenmann invariant $\ks(X)=0\in H^4(X,\partial X;\Z/2)$.  We say that $X$ is \emph{formally smooth} if it is equipped with a choice of lift $\tau_X$.
\end{dfn}

\begin{rmk}\label{rmk:formal_vs_informal}
	A smooth structure on a topological manifold determines a canonical formal smooth structure after choosing a Riemannian metric on the manifold.  This is proved by first showing that after choosing a Riemannian metric we have a canonical bundle isomorphism between a smooth manifold's tangent bundle and its normal bundle in the diagonal embedding (see \cite[Lemma 11.5]{milnor_stasheff_1974}), and then observing that the normal bundle of the diagonal embedding's underlying topological microbundle has a canonical microbundle isomorphism to the tangent microbundle of the manifold. Since the space of Riemannian metrics on a smooth manifold is contractible, this means that a smooth structure on a topological manifold determines an essentially unique formal smooth structure.
\end{rmk}

We wish to define the Casson-Sullivan invariant as the Kirby-Siebenmann invariant of the mapping cylinder of a homeomorphism, but there is a subtlety that must be addressed first.  Let $f\colon X\to X'$ be a homeomorphism of (formally) smooth manifolds.  Let $t_X$, $t_{X'}$  denote the classifying maps of the $\TOP$-microbundles of $X$ and ${X'}$ and let $\tau_X$ and $\tau_{X'}$ denote their corresponding lifts to $\BO$ given by their (formal) smooth structures.  The pullback $f^*(\tau_{X'})=\tau_{X'}\circ f$ is not a lift of $t_X$ immediately (rather, it is a lift of $t_{X'}\circ f$), but there is a homotopy which is unique up to homotopy to make it a lift of $t_X$.

\begin{lem}\label{lem:microbundle_differential}
	There is a homotopy $h(f)_t\colon X\times I\to \BTOP$ such that $h(f)_0=t_{X'}\circ f$ and $h(f)_1=t_X$, which is unique up to homotopy.
\end{lem}

\begin{proof}
	The homeomorphism $f$ induces a canonical microbundle isomorphism between the microbundles $t_X$ and the pullback bundle $f^{*}(t_X)$.  Use this isomorphism to form the mapping cylinder of microbundles on $X\times I$ and denote this by $\mathfrak{X}(f)$.  By Kirby-Siebenmann \cite[Essay IV, Proposition 8.1]{kirby_siebenmann_1977}, $\TOP$-microbundles over $X\times I$ which restrict to $t_X$ and $f^*(t_X)$ on either end are in one-to-one correspondence with homotopy classes of maps $X\times I\to \BTOP$ restricting to the classifying maps $t_X$ and $t_{X'}\circ f$.  Since $\mathfrak{X}$ is such a $\TOP$-microbundle, we get a well-defined up to (relative) homotopy map $h'(f)\colon X\times I \to \BTOP$ such that $h'(f)_0=t_X$ and $h'(f)_1=t_{X'}\circ f$, i.e.\ a homotopy between these classifying maps that is well-defined up to (relative) homotopy.  Taking the reverse homotopy gives the desired homotopy from $t_{X'}\circ f$ to~$t_X$.
	
	To get uniqueness of this homotopy up to homotopy, one can again use Kirby-Siebenman \cite[Essay IV, Proposition 8.1]{kirby_siebenmann_1977} on $(X\times I) \times I$, and then use the same argument as above.
\end{proof}

Since the map $\xi\colon \BO\to \BTOP$ is a fibration, we can lift the homotopy $h(f)_t$ from \zcref{lem:microbundle_differential} to a homotopy $\widetilde{h}(f)_t$ such that $\widetilde{h}(f)_0=\tau_{X'}\circ f=f^*(\tau_{X'})$ and such that $\widetilde{h}(f)_1$ is a lift of $t_X$.  Since this homotopy is unique up to homotopy, we will cease to mention this and instead whenever we write $f^*(\tau_{X'})$ it should be taken to mean that we have homotoped this such that it is a lift of $t_X$ (in the unique way).  In fact, one can extend the argument in the proof of \zcref{lem:microbundle_differential} to see that all higher homotopies are unique up to homotopy, etc., and hence there is an essentially unique classifying map for the tangent microbundle.  We will not need this stronger statement.

\begin{rmk}\label{rmk:ks_naturality}
	\zcref{lem:microbundle_differential} gives us immediately that the Kirby-Siebenmann invariant is natural with respect to homeomorphisms.  More precisely, it implies that for a homeomorphism of topological manifolds $f\colon X\to {X'}$  we have that $\ks(X)=f^*\ks({X'})$, since the homotopy classes of maps representing these cohomology classes are clearly homotopic by \zcref{lem:microbundle_differential}.  By an analogous argument, the Kirby-Siebenmann invariant is also natural with respect to inclusion of open submanifolds.
\end{rmk}

\begin{rmk}
	In what follows we will assume that our homeomorphisms $f\colon X\to X'$ between (formally) smooth manifolds already restrict to (formal) diffeomorphisms on the boundary.  This we can do freely since every homeomorphism of 3-manifolds is isotopic (in an essentially unique way) to a diffeomorphism by Cerf (relying on Hatcher's proof of the Smale conjecture) \cite{cerf_1959,hatcher_1983}.  We will not comment on this further.
\end{rmk}

We now define the Casson-Sullivan invariant.

\begin{dfn}\label{dfn:casson_sullivan_invariant}
	Let $X$ and $X'$ be $n$-dimensional (formally) smooth manifolds with (potentially empty) boundaries $\partial X\cong\partial X'$ and let $f\colon X\to X'$ be a homeomorphism restricting to a diffeomorphism on the boundary.  Let $M_f$ be the mapping cylinder 
	\[
	M_f:= \frac{(X\times I)\sqcup X'}{(\{x\}\times\{1\})\sim f(x)}
	\]
	and note that $\tau_X\cup_{\partial} f^*(\tau_X)'$ defines a lift of $t_{M_f}$ on $\partial M_f$.  The reduced suspension (of pairs) construction gives an isomorphism
	\begin{equation}\label{eq:varpi}
		\varpi\colon H^3(X,\partial X;\Z/2) \xrightarrow{\cong} H^4(M_f,\partial M_f;\Z/2).
	\end{equation}
	We then define the \emph{Casson-Sullivan invariant} $\cs(f)$ as
	\[
	\cs(f):= \varpi^{-1}(\ks (M_f,\partial M_f))\in H^3(X,\partial X;\Z/2).
	\]
\end{dfn}

We establish an equivalent characterisation of $\varpi$ which will sometimes be useful.

\begin{lem}
	The following diagram concerning $\varpi$ commutes
	\[
	\begin{tikzcd}
		H^3(X,\partial X;\Z/2) \arrow[r,"\varpi"] \arrow[d,"\PD"] & H^4(X\times I,\partial(X\times I);\Z/2) \arrow[d,"\PD"] \\
		H_1(X;\Z/2) \arrow[r,"i_*"] & H_1(X\times I;\Z/2)
	\end{tikzcd}
	\]
	and hence $\varpi = \PD^{-1}\circ i_* \circ \PD$, where $i\colon X\times\{0\}\to X\times I$ denotes the inclusion map.
\end{lem}

\begin{proof}
	We claim the following diagram commutes (with $\Z/2$ coefficients suppressed)
	\[
	\begin{tikzcd}
		H^3(\partial(X\times I),(\partial X \times I)\cup X\times\{1\}) \arrow[r,"\delta","\cong"'] \arrow[d,"j_*","\cong"'] & H^4(X\times I,\partial(X\times I)) \arrow[dd,"\PD","\cong"'] \\
		H^3(X,\partial X) \arrow[d,"\PD","\cong"'] & \\
		H_1(X\times \{0\}) \arrow[r,"i_*","\cong"'] & H_1(X\times I)
	\end{tikzcd}
	\]
	where $\delta$ denotes the connecting homomorphism in the long exact sequence in cohomology of the triple $(X\times I, \partial(X\times I),(\partial X\times I)\cup X\times\{1\})$ and $j$ denotes the inclusion map of pairs \[(X\times\{0\},\partial X)\to (\partial (X\times I),(\partial X\times I)\cup X\times\{1\}).\]
	One may readily verify that the maps $\delta$ and $j_*$ are isomorphisms by considering the other two adjacent terms in the long exact sequence and excision, respectively.  Commutativity can be obtained using the correct version of naturality of relative cap product.  The proof is finished by noting that $\varpi=\delta\circ (j_*)^{-1}$, since this is precisely the map that induces the (reduced) suspension isomorphism of pairs.
\end{proof}

\begin{rmk}
	It will often be useful for us to think of $\cs(f)$ instead as the element \[\varpi(\cs(f))\in H^4(M_f,\partial M_f;\Z/2).\]  In the rest of this paper, we will reserve the symbol $\varpi$ to always mean the isomorphism given in (\ref{eq:varpi}) so that is clear when we are going between these two elements.
\end{rmk}

\begin{rmk}\label{rmk:cs_description}
	We now explicitly describe the homotopy class of the map corresponding to the Casson-Sullivan invariant in the following way which will be useful for proofs.  Take the map~$p\circ \tau_{X\times I}$ and then use the null-homotopies defined by $\tau_X$ and $f^*\tau_X'$ (glued over the null-homotopy for $\partial X \times I$) to construct a relative map $(X\times I,\partial(X\times I))\to (\B(\TOP\!/\!\tO),\ast)$, which is relatively homotopic to $\varpi\cs(f)$ by construction.  We will refer to this relative class as the homotopy class of $p\circ \tau_{X\times I}$ extended by the null-homotopies $\tau_X$ and $f^*\tau_X$.
\end{rmk}

In this paper we will consider only homeomorphisms of $4$-manifolds.  In this case it is clear by Poincar\'{e} duality that the Casson-Sullivan invariant vanishes for all homeomorphisms if $X$ is simply-connected. 

\subsection{Dependence of the Casson-Sullivan invariant on smooth structures}\label{sbs:cs_dependence_smooth_structures}

This subsection is devoted to describing the extent to which the Casson-Sullivan invariant depends on the choice of a (formal) smooth structure.  We will begin by giving a different definition for the Casson-Sullivan invariant which more naturally exhibits it as an element of $H^3(X,\partial X;\Z/2)$.

Let $f\colon X\to X'$ be a homeomorphism and let $X$ be a (formally) smoothable $4$-manifold with smooth boundary $\partial X$ (endow $\partial X'$ with a smooth structure using $f$).  For any choice of lifts $\tau_X$, $\tau_{X'}$ of the topological tangent bundles $t_X$, $t_{X'}$ we have the following (very non-commutative) diagram, augmented from \zcref{dia:ks_def}.  
\begin{equation*}\label{dia:cs_def}
	\begin{tikzcd}[column sep=large,row sep=large]
		& & \TOP\!/\!\tO \arrow[d,"\iota"] \\
		& & \mathcal{B}\!\tO \arrow[d,"\xi"] \\
		\partial X \arrow[urr,"\tau_{\partial X}", bend left, near start, end anchor={[]north west}] \arrow[r] & X \arrow[uur,"{\delta(f,\tau_X,\tau_{X'})}", dashrightarrow, bend left, crossing over, near end, start anchor={[xshift=6]north west}, end anchor={[yshift=1]west}] \arrow[ur, "\tau_X", bend left] \arrow[ur, "f^*(\tau_{X'})", bend right, inner sep=0.25ex] \arrow[r,"t_X"', bend right] & \mathcal{B}\!\TOP \arrow[d,"p"] \\
		& & \mathcal{B}(\TOP\!/\!\tO) 
	\end{tikzcd}
\end{equation*}
We will now work to define $\delta(f,\tau_X,\tau_{X'})$ in the diagram, but roughly one should think of it as the `difference' between the lifts $\tau_X$ and $f^*(\tau_{X'})$.  First, note that we have an action of $[X,\partial X;\TOP\!/\!\tO]$ on homotopy classes of lifts of the stable tangent microbundle to $\BO$, defined using the $H$-space structure on $\BO$ which we will denote by the symbol $\oplus$.

\begin{lem}\label{lem:lifts_action}
	Let $(X,\partial X)$ be a topological $4$-manifold with boundary and stable tangent microbundle $t_X$.  There is a well-defined action of $[(X,\partial X), \TOP\!/\!\tO]$ on lifts of the stable tangent microbundle to $\BO$ rel.\ $\tau_{\partial X}$ defined by \[ \delta \cdot \tau_X:=\tau_X\oplus \iota\delta\]
	and furthermore this action is free and transitive.
\end{lem}

\begin{proof}
	Let $\tau_X\colon X\to \BO$ be a lift of $t_X$ extending $\tau_{\partial X}$ and let $[X,\partial X;\BO]_{t_X}$ denote homotopy classes of lifts of $t_X$ extending $\tau_{\partial X}$ rel.\ boundary (i.e.\ lifts of $t_X$ up to homotopy through lifts of $t_X$ restricting to $\tau_{\partial X}$ on $\partial X$).  By \cite[Theorem 1.3.8]{baues_1977}, and the surrounding discussion, there is a bijection 
	\[
	[(X,\partial X), \TOP\!/\!\tO] \to [X,\partial X;\BO]_{t_X}
	\]
	given by $\delta \mapsto \tau_X\oplus\iota\delta$.  This induces an action via the $H$-space structure on $\TOP\!/\!\tO$ and it being free and transitive follows by the above map being a bijection.
\end{proof}

We can now define $\delta(f,\tau_X,\tau_{X'})$.  

\begin{dfn}
	By \zcref{lem:lifts_action} the lifts $\tau_X$ and $f^*(\tau_{X'})$ determine a unique element which we denote as $\delta(f,\tau_X,\tau_{X'})$ in $[(X,\partial X), \TOP\!/\!\tO]$ such that $\delta(f,\tau_X,\tau_{X'})\cdot\tau_X=f^*(\tau_{X'})$.
\end{dfn}

\begin{prop}\label{prop:cs_dependence_ss}
	Let $X$ be a formally smoothable topological $4$-manifold and $f\colon X\to X'$ a homeomorphism.  Let $\tau_X$ and $\widetilde{\tau}_X$ be two lifts of the stable tangent microbundle of $X$, let $\tau_{X'}$ and $\widetilde{\tau}_{X'}$ be two lifts of the stable tangent microbundle of $X'$, let $a\in [(X,\partial X), \TOP\!/\!\tO]$ be the unique element given by \zcref{lem:lifts_action} such that $a\cdot \tau_X=\widetilde{\tau}_X$ and similarly let $b$ be such that $b\cdot \tau_{X'}=\widetilde{\tau}_{X'}$.  Then
	\[
	a+\delta(f,\widetilde{\tau}_X,\widetilde{\tau}_{X'}) = f^*(b) + \delta(f,\tau_X,\tau_{X'}).
	\]
	Hence it follows that if $X=X'$, $\tau_X=\tau_X'$, $\widetilde{\tau_{X}}=\widetilde{\tau}_{X'}$, and $f$ acts trivially on $H^3(X,\partial X;\Z/2)$ then $\delta(f)$ is defined and does not depend on the choice of lift of the stable tangent microbundle of $X$.    
\end{prop}

\begin{proof}
	We have that 
	\[
	(a+\delta(f,\widetilde{\tau}_X,\widetilde{\tau}_{X'}))\cdot\tau_X=\delta(f,\widetilde{\tau}_X,\widetilde{\tau}_{X'})\cdot(a\cdot\tau_X)=\delta(f,\widetilde{\tau}_X,\widetilde{\tau}_{X'})\cdot\widetilde{\tau}_X=f^*\widetilde{\tau}_{X'},
	\] where the first equality is given by using the action, the second from the definition of $a$, and the third from the definition of $\delta(f,\widetilde{\tau}_X,\widetilde{\tau}_{X'})$.  We also have that 
	\[
	(f^*(b) + \delta(f,\tau_X,\tau_{X'}))\cdot \tau_X = f^*(b)\cdot (\delta(f,\tau_X,\tau_{X'})\cdot\tau_X)= f^*(b)\cdot f^*\tau_{X'}=f^*(b\cdot \tau_{X'})=f^*\widetilde{\tau}_{X'},
	\] where the first equality is again given by using the action, the second from the definition of $\delta(f,\tau_X,\tau_{X'})\cdot\tau_X)$, the third by the fact that the $H$-space structure is compatible with taking pre-composition by the homeomorphism $f$, and the final equality comes from the definition of $b$.
	Since the action of $[(X,\partial X);\TOP\!/\!\tO]$ is free by \zcref{lem:lifts_action}, this completes the proof.
\end{proof}

It remains to be seen that this definition for $\delta(f)$ matches up with the definition of the Casson-Sullivan invariant, given as the Kirby-Siebenmann invariant of the mapping cylinder.

\begin{prop}\label{prop:delta=cs}
	Let $X$ and $X'$ be formally smooth topological $4$-manifolds with formal smooth structures $\tau_X$ and $\tau_{X'}$, respectively. Furthermore, let $f\colon X\to X'$ be a homeomorphism.  Then $\delta(f,\tau_X,\tau_{X'})=\cs(f)$.
\end{prop}

\begin{proof}
	By \zcref{rmk:cs_description} we can consider the Casson-Sullivan invariant as a relative homotopy class $\varpi\cs(f)\in [\Sigma (X_+,\partial X\sqcup\{\pt\}),\mathcal{B}(\TOP\!/\!\tO)]$, where $\Sigma$ denotes the reduced suspension (of a pair) and $X_+$ denotes $X$ with a disjoint basepoint $\pt$ added (addeding the extra point is only necessary if $\partial X=\emptyset$).  By the loopspace-suspension adjoint (and that $\B(\TOP/\tO)$ is the delooping of $\TOP/\tO$), this gives an element $\cs(f)\in [(X_+,\partial X\sqcup\{\pt\}),\TOP\!/\!\tO]$.  Since the action in \zcref{lem:lifts_action} is free and transitive, to finish the proof it suffices to show that $\cs(f)\cdot \tau_X=f^*(\tau_{X'})$.  Since this action was defined using the $H$-space structure on $\TOP/\tO$, under the loopspace-suspension adjoint this action can be computed using the map on the suspension $\varpi\cs(f)\in [\Sigma (X_+,\partial X\sqcup\{\pt\}),\mathcal{B}(\TOP\!/\!\tO)]$, and, by this element's construction (see \zcref{rmk:cs_description}), the action must send $\tau_{X}$ to $f^*(\tau_{x'})$ as these describe the null-homotopies on the two ends of the suspension.
\end{proof}
If we put together the previous two propositions, we immediately receive the following corollary.

\begin{cor}\label{cor:cs_dependence_ss}
	Let $X$ and $X'$ be topological manifolds each with two $($formal$)$ smooth structures $\mathscr{S}^1_X$ and $\mathscr{S}^2_X$, and $\mathscr{S}_{X'}^1$ and $\mathscr{S}_{X'}^2$, respectively.  Denote the corresponding lifts of the stable tangent microbundle by $\tau_X$ and $\widetilde{\tau}_X$, and $\tau_{X'}$ and $\widetilde{\tau}_{X'}$, respectively.  Let $f\colon X\to X'$ be a homeomorphism, and let $\cs(f,\mathscr{S}^1_X,\mathscr{S}^1_{X'})$ and $\cs(f,\mathscr{S}^2_X,\mathscr{S}^2_{X'})$ denote the Casson-Sullivan invariants of $f$ with respect to the corresponding $($formal$)$ smooth structures.  Let $a\in [(X,\partial X), \TOP\!/\!\tO]$ be the unique element given by \zcref{lem:lifts_action} such that $a\cdot \tau_X =\widetilde{\tau}_X$, and similarly let $b$ be the unique element such that $b\cdot \tau_{X'}=\widetilde{\tau}_{X'}$.  Then we have that
	\[
	a + \cs(f,\mathscr{S}^1_X,\mathscr{S}^1_{X'})=f^*(b)+\cs(f,\mathscr{S}^2_X,\mathscr{S}^2_{X'}).
	\]
\end{cor}

\begin{proof}
	The proof is immediate from \zcref{prop:cs_dependence_ss} and \zcref{prop:delta=cs}.
\end{proof}

If we only consider self-homeomorphisms, then $\tau_X=\tau_{X'}$, $\widetilde{\tau}_X=\widetilde{\tau}_{X'}$ and $a=b$, giving the formula \[
a+\cs(f,\mathscr{S}^1_X)=f^*(a)+\cs(f,\mathscr{S}^2_X).\]  This means, for $f$ a self-homeomorphism, we can define $\cs(f)$ not only for smooth $4$-manifolds, but also for \emph{smoothable} manifolds in the case where all self-homeomorphisms must act trivially on $H^3(-;\Z/2)$.  For example, this occurs if the fundamental group is cyclic.  In fact, in these cases one could also define $\cs(f)$ for non-smoothable manifolds by first removing a point, and then using the fact that all topological $4$-manifolds admit a smooth structure away from a point \cite[Section 8.2]{freedman_quinn_1990}.  We will not pursue this in the rest of the paper, and instead simplify matters by only considering our manifolds to be smooth.

\subsection{Properties}\label{sbs:casson_sullivan_basic_properties}

We now establish properties of the Casson-Sullivan invariant.  We begin by showing that it is a pseudo-isotopy invariant.

\begin{dfn}
	Let $X$ and $X'$ be a pair of topological manifolds and let $f,g\colon X\to X'$ be a pair of homeomorphisms.  If $X$ has boundary then $f$ and $g$ restrict to a fixed homeomorphism $f_0\colon \partial X\to \partial X'$.  We say that $f$ is \emph{pseudo-isotopic} to $g$ if there exists a homeomorphism \[
	F\colon X\times I\to X'\times I
	\] such that
	\begin{enumerate}[(i)]
		\item $F\vert_{\partial X \times I}=f_0\times \Id$,
		\item $F\vert_{X\times \{0\}}=f\colon X\times\{0\}\to X\times\{0\}$,
		\item $F\vert_{X\times\{1\}}=g\colon X\times\{1\}\to X\times\{1\}$.
	\end{enumerate}  Such an $F$ is called a \emph{pseudo-isotopy}. We say that $f$ and $g$ are \emph{isotopic} if they are pseudo-isotopic via a level-preserving pseudo-isotopy.
	
	Now let $X$ and $X'$ be smooth manifolds and $f\colon X\to X'$ be a homeomorphism such that $f$ restricts to a diffeomorphism $f_0\colon \partial X\to \partial X'$.  We say that $f$ is \emph{smoothable} if $f$ is isotopic to a diffeomorphism.  We say that $f$ is \emph{pseudo-smoothable} if $f$ is pseudo-isotopic to a diffeomorphism.
\end{dfn}

\begin{rmk}\label{rmk:relative_vs_absolute}
	The above definition gives that a homeomorphism is smoothable if it is isotopic \emph{relative to the boundary} to a diffeomorphism.  One could ask whether a homeomorphism is smoothable in the weaker sense if it is (pseudo-)isotopic to a diffeomorphism through a (pseudo-)isotopy that does not fix the boundary.  In the case of isotopies of $4$-manifolds these are equivalent (see \cite[Lemma 5.3]{galvin_ladu_2025}). In the case of pseudo-isotopies these are not the same, which is a consequence of \zcref{thm:unstable_realisation_theorem} that will not be explored in this paper, but is explored in the author's thesis \cite[Chapter 9]{galvin_thesis}.
\end{rmk}

By definition isotopy implies pseudo-isotopy, so if we can obstruct a homeomorphism from being pseudo-smoothable, it cannot be smoothable either.

\begin{prop}\label{prop:cs_pseudo-isotopy_inv}
	Let $X$ and $X'$ be smooth manifolds and let $f,g\colon X\to X'$ be a pair of homeomorphisms.
	\begin{enumerate}
		\item If $f$ is pseudo-smoothable then $\cs(f)=0$.
		\item If $f$ is pseudo-isotopic to $g$ then $\cs(f)=\cs(g)$.
	\end{enumerate}
\end{prop}

\begin{proof}
	We begin with the proof of (1).  Let $\mathscr{S}_X$ and $\mathscr{S}_{X'}$ denote the smooth structures on $X$ and~$X'$, respectively.  Note that $M_f$ is homeomorphic to $X\times I$ but with the additional specified smooth structures on the boundary, so $\cs(f)=0$ if we can find a smooth structure on $X\times I$ which restricts to the given smooth structures $\mathscr{S}_X$ and $f^*(\mathscr{S}_{X'})$ on the two boundary pieces (glued along the product smooth structure $\mathscr{S}_{\partial X\times I}$).  Let $F$ be the hypothesised pseudo-isotopy between $f$ and the diffeomorphism and denote this diffeomorphism by $\widetilde{f}$.  Use the pseudo-isotopy $F$ to pull back the product structure $\mathscr{S}_{X'}\times I$ to $F^*(\mathscr{S}_X\times I)$ on $X\times I$.  We then have that $F^*(\mathscr{S}_X\times I)\vert_{X\times \{0\}}=(\widetilde{f})^*(\mathscr{S}_{X'})=\mathscr{S}_X$ since $\widetilde{f}$ is a diffeomorphism and $F^*(\mathscr{S}_X\times I)\vert_{X\times\{1\}} = f^*(\mathscr{S}_{X'})$ and together these mean that $\cs(f)=0$.
	
	Now we prove (2).  As in \zcref{rmk:cs_description}, $\varpi\cs(f)$ is the homotopy class of $p\circ \tau_{X\times I}$ extended by the null-homotopies given by $\mathscr{S}_X$ and $f^*(\mathscr{S}_{X'})$.  Similarly, $\varpi\cs(g)$ is the homotopy class of $p\circ \tau_{X\times I}$ extended by the null-homotopies given by $\mathscr{S}_X$ and $g^*(\mathscr{S}_{X'})$.  Let $F$ be the pseudo-isotopy from $g$ to $f$.  Then $F^*(\mathscr{S}_{X'}\times I)$ describes a null-homotopy of $p\circ \tau_{X\times I}$ extended by the null-homotopies given by $f^*(\mathscr{S}_{X'})$ and $g^*(\mathscr{S}_{X'})$, i.e.\ a homotopy between the null-homotopies given by $f^*(\mathscr{S}_X)$ and $g^*(\mathscr{S}_X)$.  This gives us a homotopy relative to the boundary between the relative homotopy classes defining $\varpi\cs(f)$ and $\varpi\cs(g)$, and hence $\cs(f)=\cs(g)$.
\end{proof}

We can say more in the case of self-homeomorphisms.  Let $\widetilde{\pi}_0\Homeo(X,\partial X)$ denote the pseudo-mapping class group of $X$ relative to $\partial X$, i.e.\ the quotient of $\Homeo(X,\partial X)$ by those homeomorphisms which are pseudo-isotopic to the identity (see \zcref{dfn:mapping_class_groups} and \zcref{rmk:blockeo}).  Then we have the following result.

\begin{prop}\label{prop:cs_homomorphism}
	Let $X$ be a smooth $4$-manifold.  Then the map \[\cs\colon \widetilde{\pi}_0\Homeo(X,\partial X)\to H^3(X,\partial X;\Z/2)\] sending a representative of a pseudo-isotopy class of self-homeomorphisms to its Casson-Sullivan invariant is a crossed homomorphism.  In other words, if $f,g\colon X\to X$ are representatives of pseudo-isotopy classes of self-homeomorphisms then
	\[
	\cs(g\circ f)= \cs(f) + f^*\cs(g).
	\]
\end{prop}

\begin{proof}
We first note that the well-definedness of the above map follows directly from \zcref{prop:cs_pseudo-isotopy_inv}.  It suffices to show that for two self-homeomorphisms $f,g\colon X\to X$, we have that $\cs(g\circ f)=\cs(f)+f^*\cs(g)$.

We begin by giving a method for describing the group operation on $[X\times I,\partial; K(\Z/2,4)]$.  The normal way to define this is by using the highly-connected map $\mathcal{B}(\TOP\!/\!\tO)\to \Omega K(\Z/2,5)$ and then using composition of loops to define the group operation (see \cite[Section 4.3]{hatcher_2002}).  However, we also have that $[X\times I,\partial; K(\Z/2,4),\ast]= [\Sigma X_+,\{\pt\};K(\Z/2,4),\ast]$, where $X_+:=X\sqcup\{\pt\}$ and $\Sigma X_+$ denotes the reduced suspension of $X_+$.  Hence, $[X\times I,\partial; K(\Z/2,4)]$ has a natural operation by `stacking' cylinders given by the natural group operation coming from the reduced suspension.  By the suspension-loopspace adjoint relation,
\[
[\Sigma X_+,\{\pt\};K(\Z/2,4),\ast] = [X_+,\{\pt\};\Omega K(\Z/2,4),\ast] = [X_+,\{\pt\};\Omega \Omega K(\Z/2,5),\ast]
\]
and the group operation given by `stacking' corresponds to the group operation given by the outermost loopspace on the right.  The standard group operation comes from the inner loopspace structure, but these two are equivalent (this is entirely analogous to the fact that $\pi_2$ of any space is abelian, see e.g.\ \cite[Section 4.3, pages 395-396]{hatcher_2002}).

Let $\mathscr{S}_X$ denote the smooth structure on $X$.  Consider the homotopy class of maps corresponding to $\varpi(\cs(f)+f^*\cs(g))=\varpi\cs(f) + (f\times \Id)^*\varpi\cs(g)$.  (Here we used the definition of $\varpi$ from \zcref{dfn:casson_sullivan_invariant}).  We again use the terminology in \zcref{rmk:cs_description}.  This corresponds to stacking the homotopy class of $p\circ \tau_{X\times I}$ extended by the null-homotopies given by $\mathscr{S}_X$ and~$f^*(\mathscr{S}_X)$ with the homotopy class of $p\circ \tau_{X\times I}$ extended by the null-homotopies given by~$f^*(\mathscr{S}_X)$ and $(g\circ f)^*(\mathscr{S}_X)$.  But this is relatively homotopic to the homotopy class of~$p\circ \tau_{X\times I}$ extended by the null-homotopies given by $\mathscr{S}_X$ and $(g\circ f)^*(\mathscr{S}_X)$.  The relative homotopy is given by the null-homotopy corresponding to $(f\times \Id)^*(\mathscr{S}_X\times I)$.  This completes the proof.
\end{proof}

\begin{rmk}\label{rmk:uncross_cs}
	If $X$ is such that all self-homeomorphisms must act trivially on $H^3(X,\partial X;\Z/2)$, for example if $\pi_1(X)$ is cyclic, then \zcref{prop:cs_homomorphism} actually gives that the Casson-Sullivan invariant defines a group homomorphism from the pseudo-mapping class group.  In such cases, one can combine this result with \zcref{cor:cs_dependence_ss} to obtain that the Casson-Sullivan invariant defines a group homomorphism even if we start with a smoothable $X$, i.e.\ without picking a smooth structure.
\end{rmk}

Let $X$ and $X'$ be $4$-dimensional smooth manifolds and let $f\colon X\to X'$ be an orientation-preserving homeomorphism restricting to a diffeomorphism on $\partial X$.  If we use a disc $D\subset X$ to perform a connected-sum with $S^2\times S^2$, and use the same disc in $S^2\times S^2$ and the disc $f(D)\subset X'$ to form a connected-sum $X'\# (S^2\times S^2)$, this produces a well-defined homeomorphism \[f_\#\colon X\# (S^2\times S^2)\to X'\# (S^2\times S^2)\] by extending $f$ onto the~$S^2\times S^2$ summand via the identity map.


\begin{rmk}\label{rmk:connected-sum}
	Note that there are some additional subtleties concerning the above if we consider self-homeomorphisms.  In general, one would like to be able to take two self-homeomorphisms $f_1\colon X_1\to X_1$ and $f_2\colon X_2\to X_2$ and form a connected-sum self-homeomorphism $f_1\# f_2\colon X_1\# X_2\to X_1\# X_2$, but this is, in general, not well-defined.  To define a self-homeomorphism of the connected-sum, we first need to isotope the images of the two discs (that are used to perform the connected-sum) back to their original positions.  The first step in doing this is choosing an isotopy of their centres, and then using isotopy extension \cite{edwards_kirby_1971}, but there are many choices for such an isotopy, and different choices do not necessarily lead to isotopic self-homeomorphisms.  After we have isotoped the centres of the discs back to their original positions, we can proceed by using uniqueness of normal bundles \cite[Chapter 9.3]{freedman_quinn_1990}, and the calculation of the mapping class group of $S^3$ \cite{cerf_1968} to build the connected-sum homeomorphism.  Hence, the only issue concerning well-definedness is due to the choice of the initial isotopies, and this corresponds to so-called `point pushing homeomorphisms'.
	
	For certain manifolds, however, the connected-sum homeomorphism is well-defined.  For example, if one of the connected-summands is $S^2\times S^2$ and the other is simply-connected, then the connected-sum homeomorphism is well-defined.  For more information, see \cite{auckly_kim_melvin_ruberman_2014}.  This is, unfortunately, not very useful to our purposes since the simply-connected assumption renders the Casson-Sullivan invariant trivial.
\end{rmk}

Pinching off a 4-sphere yields a standard degree one map $S^2\times S^2\to S^4$ and by extending this map via the identity onto a 4-manifold $X$ we obtain a degree one collapse map $\ell\colon X\# (S^2\times S^2)\to X\# S^4$.

\begin{prop}\label{prop:cs_stable_1}
	  The standard degree one map that collapses the $S^2\times S^2$ connected-summand $\ell\colon X\#(S^2\times S^2)\to X$ (as above) induces an isomorphism \[\ell^* \colon H^3(X,\partial X;\Z/2)\to H^3(X \# (S^2\times S^2),\partial X;\Z/2)\] such that $\ell^*(\cs(f))=\cs(f_\#)$.
\end{prop}

\begin{proof}
	Let $\mathscr{S}_X$ denote the smooth structure on $X$ and let $\mathscr{S}_{X\# S^2\times S^2}$ denote the induced smooth structure on $X\# (S^2\times S^2)$.  Consider the following diagram (where we have used the fact that the stable tangent bundle of $S^2\times S^2$ is trivial to already factor the stable tangent bundle of $X\#(S^2\times S^2)$ through $\ell$ as in the top horizontal row)
	\[
	\begin{tikzcd}[row sep=large, column sep=large]
	\partial((X \# (S^2\times S^2))\times I) \arrow[r, "\ell\cup \ell"] \arrow[d] & \partial(X\times I) \arrow[r,"\tau_X \cup \tau_X\circ f"] & \BO \arrow[d,"\xi"] \\
	(X \# (S^2\times S^2))\times I \arrow[r,"\ell\times \Id"'] \arrow[rr,bend right, "t_{(X\# (S^2\times S^2))\times I}"] \arrow[urr, dotted, "\tau_{(X\# (S^2\times S^2))\times I}", near start, end anchor={[xshift=-5, yshift=2]south west}] & X \times I \arrow[from=u, crossing over] \arrow[ur, dotted, "\tau_{X\times I}"', start anchor= {[yshift=5]east}, end anchor={[xshift=-8]south}] \arrow[r,"t_{X\times I}"'] & \BTOP \arrow[d,"p"] \\
	& & \mathcal{B}(\TOP\!/\!\tO) 
	\end{tikzcd}
	\]
	Since $f_\#$ restricts to the identity map on the $S^2\times S^2$ summand and the tangent bundle of~$S^2\times S^2$ is stably trivial, the stable tangent microbundle $t_{(X\# (S^2\times S^2))\times I}$ is homotopic to the stable tangent microbundle $t_{X\times I}$ precomposed with the map $\ell\times \Id$.  This means that $\varpi\cs(f_\#)$, the homotopy class of $p\circ t_{(X\# (S^2\times S^2))\times I}$ extended by the null-homotopies given by $\mathscr{S}_{X\# (S^2\times S^2)}$ and $f_{\#}^*(\mathscr{S}_{X'\# (S^2\times S^2)})$ is relatively homotopic to the homotopy class of $p\circ t_{X\times I}\circ (\ell\times \Id)$ extended by the null-homotopies given by precomposing the null-homotopies corresponding to $\mathscr{S}_X$ and~$f^* \mathscr{S}_{X'}$ with the map $\ell$.  Since~$(\ell\times \Id)^*(\varpi\cs(f))=\varpi(\ell^*\cs(f))$, it follows that $\ell^*(\cs(f))=\cs(f_\#)$.
\end{proof}

\begin{prop}\label{prop:cs_stable_2}
	Let $X$ and $X'$ be a pair of smooth $4$-manifolds and let $f\colon X\to X'$ be a homeomorphism restricting to a diffeomorphism on $\partial X$.  If $\cs(f)=0$ then there exists a non-negative integer $k$ such that $f_\#\colon X \#( \#^k S^2\times S^2) \to X'\# (\#^k S^2\times S^2)$ is pseudo-isotopic to a diffeomorphism.
\end{prop}

This result is stated and proved in \cite[Section 8.6]{freedman_quinn_1990}.  However, the proof is somewhat dispersed in the book and many of the details are not given.  For this reason, we give the full proof below.

\begin{proof}[Proof of \zcref{prop:cs_stable_2}]
	Let $W:=M_f$ be the mapping cylinder of $f$.  Since $\cs(f)=0$, then by definition we can smooth the stable topological normal bundle of $W$ relative to the given smoothings of the stable tangent bundles of the boundary.  By Kirby-Siebenmann \cite[Essay IV, Theorem 10.1]{kirby_siebenmann_1977}, we can realise this by a smooth structure on $W$ extending the given structures on the boundary.  Note that here we have crucially used that the dimension of $W$ is at least five and at most seven.  This allows us to view $W$ as a (relative) smooth $h$-cobordism $(W,X,X')$.  Note that $W$ is topologically a product (by definition) but that it is not necessarily a product smoothly, and that $W$ already restricts to a smooth product on $\partial X$.  We aim to turn $W$ into a smooth product cobordism via stabilisations.
	
	We may assume in the standard way that $W$ has only $2$- and $3$-handles in its relative (to $X$) handle decomposition (for a reference, see \cite[\S 20.1]{behrens_kalmar_kim_powell_ray_2021}).  Since $W$ is topologically a product we know that the $2$- and $3$-handles algebraically cancel.  Let $\mathcal{W}$ be a collection of immersed Whitney discs in $X\times\{1/2\}$ for the pairs of cancelling intersections of the descending manifolds for the $3$-handles and the ascending manifolds for the $2$-handles.  If these discs were embedded disjointly then we could use Whitney moves on the discs to force the $2$- and $3$-handles to geometrically cancel and hence $W$ could be made into a smooth product cobordism. Let $p\in D_1\cap D_2$ be an intersection point for two Whitney discs $D_1,D_2\in\mathcal{W}$ (potentially $D_1=D_2$) and let $\alpha$ be an arc in $W$ from $X\times\{0\}$ to $X\times\{1\}$ which intersects $D_1$ and $D_2$ exactly once at $p$ and is disjoint from all other discs in $\mathcal{W}$.  Let $q=(q_1,q_2)$ be a point in $S^2\times S^2$.  Then we form a new cobordism
	\[
	W':= \big(W\sm \nu \alpha\big)\cup_{\partial\nu \alpha = (\partial\nu q)\times I} \big((S^2\times S^2 \sm \nu q)\times I\big).
	\]
	We define a new set of Whitney discs $\mathcal{W}'$ to be the same as $\mathcal{W}$ except with $D_1$ and $D_2$ replaced with $D_1'$ and $D_2'$, respectively, defined as
	\begin{align*}
		D_1':=& \big(D_1\cap (W\sm\nu \alpha)\big) \cup \big( (S^2\times\{q_2\})\cap (S^2\times S^2\sm\nu q)\big), \\
		D_2':=& \big(D_2\cap (W\sm\nu \alpha)\big) \cup \big( (\{q_1\}\times S^2)\cap (S^2\times S^2\sm\nu q)\big).
	\end{align*}
	The Whitney discs in $\mathcal{W}'$ then pair the cancelling intersections of the $2$- and $3$-handles of $W'$.  The number of intersections between $D_1'$ and $D_2'$ is one fewer than the number of intersections between $D_1$ and $D_2$; we have effectively removed an intersection point.  This is known as the Norman trick.  Repeating this procedure for all intersections between Whitney discs, eventually we produce a smooth cobordism~$W''$ which is topologically a product $W''\approx X\# (\#^k S^2\times S^2)\times I$ for some non-negative integer~$k$.  The set of Whitney discs for the pairs of cancelling intersections of the $2$- and $3$-handles $\mathcal{W}''$ consists now of disjointly embedded Whitney discs, and hence (as noted above) $W''$ is diffeomorphic to a product cobordism relative to one end of the cobordism.  We now make this conclusion precise.
	
	Let $X_\#:=X\# (\#^k S^2\times S^2)$, $X'_{\#}:= X'\# (\#^k S^2\times S^2)$ and let $\mathscr{S}$, $\mathscr{S}'$ denote the smooth structures on $X_\#$ and $X'_{\#}$, respectively.  Now $W''$ is the mapping cylinder of $f_\#$, i.e.\ $W'' \cong X_\# \times I$, with the smooth structure $\mathscr{S}_{W''}$ where $\mathscr{S}_{W''}\vert_{X_\#\times\{0\}}=\mathscr{S}$ and $\mathscr{S}_{W''}\vert_{X_\#\times\{1\}}=f_\#^*(\mathscr{S}')$.  In the previous paragraph, we obtained a homeomorphism \[F\colon X_\#\times I \to X_\#\times I\] such that $F^*(\mathscr{S}'\times I)=\mathscr{S}_W''$ and such that $F\vert_{X_{\#}\times\{0\}\cup \partial X_{\#}\times I}$ is the identity.  By \zcref{dfn:diffeomorphism} this means we have produced a pseudo-isotopy between the smooth structures $\mathscr{S}$ and $f_\#^*(\mathscr{S}')$ and, by the same argument as in the third paragraph of the proof of \zcref{prop:iso_iff_smoothable}, this means that $f_\#$ is pseudo-isotopic to a diffeomorphism.
\end{proof}

To summarise, \zcref{prop:cs_stable_1} and \zcref{prop:cs_stable_2} together tell us that the Casson-Sullivan invariant is the stable obstruction to pseudo-smoothing homeomorphisms of $4$-manifolds, much as the Kirby-Siebenmann invariant is the stable obstruction to smoothing $4$-manifolds.

\subsubsection{Non-compact $4$-manifolds}

Non-compact $4$-manifolds have the property that they are easier to smooth than their compact counterparts, in the sense that non-compact $4$-manifolds always admit smooth structures (and hence compact $4$-manifolds can always be smoothed away from a point) \cite[Chapter 8.2]{freedman_quinn_1990}.  In light of this, one might wonder whether a stronger result than \zcref{prop:cs_stable_2} holds if we assume our manifold is non-compact.  We briefly explain what happens in this case.  We begin with some definitions (cf.\ \zcref{sec:isotopy_smooth_structures}).

\begin{dfn}
	Let $M$ be a topological manifold.  We say that smooth structures $\mathscr{S}$ and $\mathscr{S}'$ on $M$ are \emph{concordant} if there exists a smooth structure $\mathscr{T}$ on $M\times I$ such that $\mathscr{T}\vert_{M\times \{0\}}=\mathscr{S}$ and $\mathscr{T}\vert_{M\times \{1\}}=\mathscr{S}'$.
	
	We say that $\mathscr{S}$ and $\mathscr{S}'$ are \emph{sliced concordant} if they are concordant, as above, such that the projection $M\times I \to I$ is a submersion with respect to the smooth structure $\mathscr{T}$ on $M\times I$.
\end{dfn}

\begin{prop}\label{prop:sliced_concordance}
	Let $X$ be a topological $4$-manifold with $\mathscr{S}$ and $\mathscr{S}'$ smooth structures on $X$, and let $f\colon X\to X$ be a homeomorphism restricting to a diffeomorphism on $\partial X$ with respect to the smooth structures $\mathscr{S}\vert_{\partial X}$ and $\mathscr{S}'\vert_{\partial X}$.  If $\cs(f)=0$, then $\mathscr{S}$ and $f^*(\mathscr{S}')$ are concordant.  If $X$ is non-compact and $\cs(f)=0$, then  $\mathscr{S}$ and $f^*(\mathscr{S}')$ are sliced concordant.
\end{prop}

This proposition essentially follows (in the non-compact case) from the following theorem.

\begin{thm}[{\cite[Theorem 4.4]{siebenmann_1971},\cite[Theorem 8.7B]{freedman_quinn_1990}}]\label{thm:sliced_concordance}
	There is a one-to-one correspondence between sliced concordance classes of smooth structures on a non-compact $4$-manifold with homotopy classes of liftings of the stable tangent microbundle to $\BO$.
\end{thm}

\begin{rmk}
	Note that although in \cite{freedman_quinn_1990} they state the above theorem for concordance classes rather than sliced concordance classes, the references they refer to for the proof give it for sliced concordance classes.
\end{rmk}

\begin{proof}[Proof of \zcref{prop:sliced_concordance}]
	The fact that $\mathscr{S}$ and $f^*(\mathscr{S}')$ are concordant is clear in the context of the proof of \zcref{prop:cs_stable_2}.  In fact, this is exactly what the the first paragraph of the proof establishes.  If $\cs(f)=0$, then by \zcref{lem:lifts_action} and \zcref{prop:delta=cs} the lifts of the stable tangent microbundle corresponding to $\mathscr{S}$ and $f^*(\mathscr{S})$ are homotopic, so by \zcref{thm:sliced_concordance} these smooth structures are sliced concordant.
\end{proof}

\begin{rmk}\label{rmk:quinn_fix}
	Note that an important part of the proof of \zcref{thm:sliced_concordance} is that the map \[\TOP(4)/\tO(4) \to  \TOP/\tO\] is $5$-connected \cite[Theorem 9.7A]{freedman_quinn_1990}, and this theorem rests on the result of Quinn \cite{quinn_1986} that $\pi_4(\TOP(4)/\tO(4))=0$, the proof of which was found to contain a gap.  The proof, however, was recently corrected in \cite{gabai_gay_hartman_krushkal_powell_2026}.
\end{rmk}

A concordance, sliced or otherwise, between the smooth structures $\mathscr{S}$ and $f^*\mathscr{S}'$ does not obviously give any nice statement about the properties of $f$ itself.  Hence, the author does not know of an interpretation of \zcref{prop:sliced_concordance} in terms of the smoothability of the homeomorphism (cf.\ \zcref{sec:isotopy_smooth_structures}).

\subsection{A connected-sum formula over a circle for the Casson-Sullivan invariant}\label{sbs:casson_sullivan_connect_sum}

This subsection is devoted to proving a connected-sum along a circle formula for the Casson-Sullivan invariant.  We shall start by giving the necessary definitions.

\begin{dfn}
	Let $X_1$ and $X_2$ be a pair of smooth $4$-manifolds and let $\gamma_i \subset X_i$ be a pair of framed, embedded circles.  Then we define the \emph{connected-sum over $\gamma_1,\gamma_2$} to be the smooth manifold
	\[
	X_1\#_{\gamma_1=\gamma_2} X_2:=(X_1\sm\nu\gamma_1)\cup_\varphi(X_2\sm\nu\gamma_2)
	\]
	where the gluing is performed using the orientation reversing map $\varphi\colon S^1\times S^2\to S^1\times S^2$ which sends $(x,y)\to (x,a(y))$ for $a\colon S^2\to S^2$ the antipodal map.  For a precise description of how this gives a well-defined smooth manifold, see \cite[\S IV.4]{kosinski_1993}.
\end{dfn}

Let $X_1$, $X_2$, $X'_1$ and $X'_2$ be two pairs of smooth, orientable, compact $4$-manifolds and let $\gamma_i \subset X_i$ be a pair of embedded circles for $i=1$, $2$.  Furthermore, let $f_i\colon X_i\to X'_i$ be a pair of homeomorphisms.  Since our manifolds are orientable, these circles admit framings which we will use implicitly (the choice will have no affect from the perspective of the Casson-Sullivan invariant).  Let $\nu \gamma_i$ denote the two tubular neighbourhoods.  If we use these to perform the connected-sum over a circle $X_1\#_{\gamma_1=\gamma_2}X_2$, and use $f(\nu\gamma_i)$ perform the connected-sum over a circle $X'_1\#_{f_1(\gamma_1)=f_2(\gamma_2)}X'_2$, this gives a connected-sum over a circle homeomorphism
\[
f_\# \colon X_1\#_{\gamma_1=\gamma_2} X_2 \to X'_1\#_{f_1(\gamma_1)=f_2(\gamma_2)} X'_2
\] 
defined as $f_i$ on $(X_i\sm\nu\gamma_i)$.  A priori this homeomorphism depends on the choice of framings, but we will ignore these choices and consider any such result the connected-sum over a circle homeomorphism.  Like with the case of connected-sum homeomorphisms, there is an additional complication when one wants to consider self-homeomorphisms.  In this case, we must first assume that the $f_i$ map $\gamma_i$ to $\gamma_i$ as free homotopy classes.  Then, up to isotopy, we may assume by homotopy implies isotopy, isotopy extension \cite{edwards_kirby_1971}, uniqueness of normal bundles (\cite[Chapter 9.3]{freedman_quinn_1990}), and the calculation of the mapping class group of $S^1\times S^2$ \cite{gluck_1962} that the $f_i$ are either the identity map on a tubular neighbourhood of the curve $\gamma_i$ or they are the \emph{Gluck twist} map 
\begin{align}\label{eq:gluck_twist}
	T\colon S^1 \times D^3 \to& S^1 \times D^3 \\
	(t,x) \mapsto& (t, \Phi_t(x)) \nonumber
\end{align} 
where $\Phi_t$ denotes the (positive) rotation map of $D^3$ around the (oriented) straight line from the south pole to the north pole by an angle of $t$ (we have used the identification $S^1\cong [0,2\pi]/0\sim2\pi$).  If the homeomorphisms $f_i$ are both of the same type as above (i.e.\ both are identity maps or both are twist maps on tubular neighbourhoods of $\gamma_i$) we may then define the connected-sum of these homeomorphisms over the curves $\gamma_i$ to be 
\[
f_\# \colon X_1\#_{\gamma_1=\gamma_2} X_2 \to X_1\#_{\gamma_1=\gamma_2} X_2
\] 
as $f_i$ on $(X_i\sm\nu\gamma_i)$.  However, there was a choice in how to isotope the $f_i$ such that the homeomorphisms fix the curves $\gamma_i$, and so it is unlikely that the resulting homeomorphism is well-defined, even up to isotopy.  We will ignore this ambiguity, and whenever we write the connected-sum over a circle for self-homeomorphisms, we shall simply mean one such choice of a homeomorphism.  The above also generalises to general case, i.e.\ for $X_1$, $X_2$, $X'_1$, $X'_2$ and $\gamma_i\subset X_i$, $\gamma'_i\subset X'_i$ being two pairs of embedded circles, but one has to assume that the homeomorphisms $f_i$ send the free homotopy classes of $\gamma_i$ to $\gamma'_i$.

We now define a map that will be very important to us.

\begin{dfn}\label{dfn:Q}
	Let $X_{\#}:=X_1\#_{\gamma_1=\gamma_2}X_2$.  Consider the long exact sequence for the triple \[(X_{\#},\partial X_{\#}\sqcup \partial \ol{\nu}\gamma_1,\partial X_{\#})\] and the Mayer-Vietoris sequence for the decomposition \[\left(X_\#,\partial X_\# \sqcup\partial\ol{\nu}\gamma_1\right)=\left((X_1\sm\nu\gamma_1) \cup (X_2\sm\nu\gamma_2),\partial(X_1\sm\nu\gamma_1)\cup\partial(X_2\sm\nu\gamma_2)\right)\]
	with $\Z/2$-coefficients (which we will now suppress).  The Mayer-Vietoris sequence yields a isomorphism \[\mathfrak{a}\colon H^3(X_{\#},\partial \sqcup \partial\nu\gamma_1)\to H^3(X_1\sm\nu\gamma_1,\partial)\oplus H^3(X_2\sm\nu\gamma_2,\partial)\] and the sequence of the triple yields a surjection\footnote{That this is a surjection can be readily verified by considering the next two terms in the sequence of the triple.} \[\mathfrak{b}\colon H^3(X_{\#},\partial\sqcup\partial\nu\gamma_1)\to H^3(X_{\#},\partial).\]  The inclusion $X_i\sm\nu\gamma_i\subset X_i$ induces an isomorphism $\mathfrak{c}\colon H_1(X_1\sm\nu\gamma_1)\oplus H_1(X_2\sm\nu\gamma_2)\to H_1(X_1)\oplus H_1(X_2)$.  We then define a map
	\[Q\colon H^3(X_1,\partial)\oplus H^3(X_2,\partial)\to H^3(X_{\#},\partial)\]
	as $Q:=\mathfrak{b}\circ \mathfrak{a}^{-1}\circ (\PD^{-1},\PD^{-1})\circ \mathfrak{c}^{-1}\circ (\PD,\PD)$.
\end{dfn}

\begin{lem}\label{lem:mv_collapse}
	The map $Q\colon H^3(X_1,\partial X_1;\Z/2)\oplus H^3(X_2,\partial X_2;\Z/2)\to H^3(X_1\#_{\gamma_1=\gamma_2}X_2,\partial;\Z/2)$ is surjective and has kernel $\Z/2\langle (\PD^{-1}[\gamma_1],\PD^{-1}[\gamma_2])\rangle$.
\end{lem}

\begin{proof}
	All of the maps used in the definition of $Q$ are isomorphisms except for $\mathfrak{b}$, which is a surjection.  Hence $Q$ is surjective.  Furthermore, we see that \[\ker Q = (\PD^{-1},\PD^{-1})\circ \mathfrak{c}\circ (\PD,\PD)\circ \mathfrak{a}(\ker\mathfrak{b}).\]
	Now $\ker\mathfrak{b}$ is generated by $\PD^{-1}[\gamma_1]$, and hence (by a standard naturality of cap product argument) we see that $(\PD,\PD)\circ \mathfrak{a}(\ker\mathfrak{b})=([\gamma_1],[\gamma_2])$.  The result follows immediately.
\end{proof}

The rest of this subsection will be devoted to proving the following theorem.

\begin{thm}\label{thm:connect_sum}
	Let $X_1$, $X_2$, $X'_1$ and $X'_2$ be two pairs of compact, connected, smooth, orientable $4$-manifolds and let $\gamma_i \subset X_i$ and $\gamma'_i\subset X'_i$ be two pairs of embedded circles, with $f_i\colon X_i\to X'_i$ a pair of homeomorphisms such that $(f_i)_*[\gamma_i]=[\gamma'_i]$ as free homotopy classes, and such that the connected-sum homeomorphism 
	\[
	f_{\#}:=f_1\#_{\gamma_1=\gamma_2}f_2\colon X_1 \#_{\gamma_1=\gamma_2}X_2 \to X'_1\#_{\gamma'_1=\gamma'_2}X'_2
	\] 
	is defined. Let $Q$ be the map in \zcref{dfn:Q}.  Then
	\[
	\cs(f_\#)=Q(\cs(f_1),\cs(f_2)).
	\]
\end{thm}

The proof of this theorem relies on a sequence of diagram chases, using the relevant Mayer-Vietoris exact sequences and long exact sequences of the triple.  Although we have already used these, the setup here will be more complicated and so we recall them.  The long exact sequence for the triple of $CW$-complexes $W\supset B \supset A$ in cohomology (with $\Z/2$-coefficients suppressed) is 
\begin{equation}\label{eq:les_triple}
	\cdots\to H^k(W,B) \to H^k(W,A)\to H^k(B,A) \to H^{k+1}(W,B) \to\cdots
\end{equation}and the fully relative Mayer-Vietoris sequence for the pair of $CW$-complexes $(W,Y)=(A\cup B, C\cup D)$ (with $\Z/2$-coefficients suppressed) is\begin{equation}\label{eq:les_mv}
	\cdots\to H^k(W,Y)\to \begin{matrix} H^k(A,C)\\ \oplus \\ H^k(B,D)\end{matrix} \to H^k(A\cap B,C\cap D) \to H^{k+1}(W,Y) \to\cdots
\end{equation}See \cite[p.200/204]{hatcher_2002} for these standard exact sequences.

First we give the notation for the setup.

\begin{stp}\label{stp:connect_sum}
	Let $X_1$, $X_2$, $X'_1$ and $X'_2$ be two pairs of compact, connected, smooth, orientable $4$-manifolds, let $\gamma_i \subset X_i$  be a pair of embedded circles, and let $f_i\colon X_i\to X'_i$ be a pair of homeomorphisms such that $(f_i)_*[\gamma_i]=[\gamma'_i]\in\pi_1(X_i)$ and such that the $f_i$ restrict to diffeomorphisms $(f_i)_0\colon \partial X_i\to \partial X'_i$.  Denote by $X_\#$ and $X'_\#$ the connected-sum $X_1\#_{\gamma_1=\gamma_2} X_2$ and $X'_1\#_{\gamma'_1=\gamma'_2} X'_2$, respectively, and by $f_\#\colon X_\#\to X'_\#$ the connected-sum homeomorphism.  We then set up the following notation:
	\begin{enumerate}[(i)]
		\item Let $W:= M_{f_\#}$ the mapping cylinder of $f_{\#}$.
		\item Let $E:= \partial (\nu \gamma_1) \times I \subset W$.
		\item Let $F:= \partial W = (X_\# \times \{0\}) \cup (\partial X_\# \times I) \cup (X_\# \times \{1\})\subset W$.
		\item Let $A:= (X_1\sm(\nu\gamma_1)) \times I \subset W$ and let $B:= (X_2\sm(\nu\gamma_2)) \times I \subset W$.
		\item Let $F_A:=F\cap A$ and let $F_B:=F\cap B$.
		\item Let $C:= F_A\cup E\subset A$ and let $D:= F_B\cup E\subset B$.
		\item Let $Y:= C\cup D$.
	\end{enumerate}

\end{stp}
Note that $W=A\cup B$ and $A\cap B=C\cap D=E$.  See \zcref{fig:connect_sum_diagram}.

\begin{figure}[h]
	\includegraphics{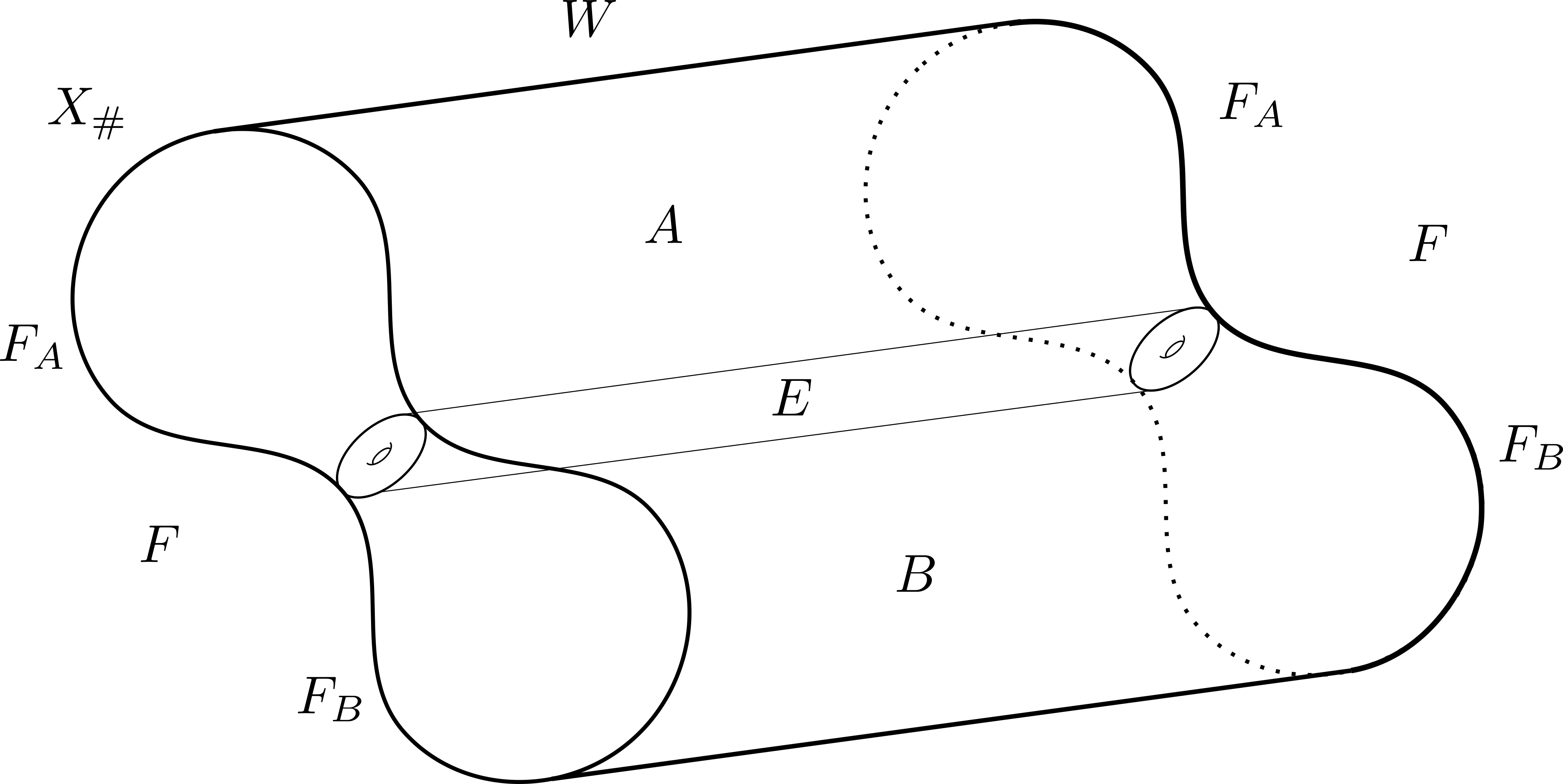}
	\centering
		\caption{A dimension-reduced picture of \zcref{stp:connect_sum} where the pictured tori denote the connected-sum $S^1\times S^2$.  Note that the labels for $C$ and $D$ have been omitted.}\label{fig:connect_sum_diagram}
\end{figure}

Combining the sequences (\ref{eq:les_triple}) and (\ref{eq:les_mv}) for the triples $(W,Y,F)$, $(A,C,F_A)$, and $(B,D,F_B)$ and the pairs $(W,Y)=(A\cup B,C\cup D)$, $(W,F)=(A\cup B, F_A\cup F_B)$, and $(Y,F)=(C\cup D, F_A\cup F_B)$ we obtain the following commutative diagram which we will use extensively.  In what follows, for all of our cohomology groups we are using $\Z/2$-coefficients but we will suppress these in the diagrams.

\begin{equation*}
\begin{adjustbox}{width=\textwidth,center}

	\begin{tikzcd}
		& H^3(Y,F) \arrow[d] \arrow[r] & H^3(C,F_A)\oplus H^3(D,F_B) \arrow[d, shift left=5ex] \arrow[d, shift right=5ex]  & \\
		H^3(A\cap B,C\cap D) \arrow[r] & H^4(W,Y) \arrow[d] \arrow[r] & H^4(A,C)\oplus H^4(B,D) \arrow[d, shift left=5ex] \arrow[d, shift right=5ex] \arrow[r] & H^4(A\cap B, C\cap D)\\
		H^3(A\cap B, F_A\cap F_B) \arrow[r] & H^4(W,F) \arrow[d] \arrow[r]  & H^4(A,F_A)\oplus H^4(B,F_B) \arrow[d, shift left=5ex] \arrow[d, shift right=5ex] \arrow[r] & H^4(A\cap B, F_A\cap F_B)\\
		&  H^4(Y,F) & H^4(C,F_A)\oplus H^4(D,F_B) & 
	\end{tikzcd}
\end{adjustbox}
\end{equation*}

It will be useful to simplify this diagram.  First, note that the leftmost and rightmost groups on the second line vanish because $A\cap B=C\cap D$ (further observe that this means the remaining non-trivial horizontal map on the second line must be an isomorphism).  The topmost groups and the leftmost group on the third line are all isomorphic to $H^3(E,\partial E)\cong \Z/2$.  We can replace all of the remaining outer groups with zeroes since these maps must be zero maps (one can explicitly see this by continuing the sequences and using commutativity along with Poincar\'{e} duality).  To illustrate this, we will show that the bottom vertical maps are zero.  Showing that the map to $H^4(A\cap B, F_A\cap F_B)$ is zero is analogous.  We start by continuing the sequences at the bottom of the diagram to obtain the following diagram.

\begin{equation*}
	\begin{tikzcd}
		& H^4(W,F) \arrow[d,"a"] \arrow[r] & H^4(A,F_A)\oplus H^4(B,F_B)  \arrow[d,"b"] & \\
		H^3(C\cap D, F_A\cap F_B) \arrow[r,"f"] & H^4(Y,F) \arrow[r,"c"] \arrow[d,"d"]  &H^4(C,F_A)\oplus H^4(D,F_B) \arrow[d,"e"] & \\
		0 \arrow[r]&H^5(W,Y) \arrow[r,"\cong"] & H^5(A,C)\oplus H^5(B,D) \arrow[r] & 0 \\
	\end{tikzcd}
\end{equation*}

We want to show that the maps $a$ and $b$ are both the zero maps.  It suffices to show that $d$ and $e$ are injections.  By Poincar\'{e} duality, $e$ is dual to direct sum of inclusion induced maps\[H_0(E)\oplus H_0(E) \to H_0(A)\oplus H_0(B)\] which is clearly an isomorphism since both $A$, $B$ and $E$ are connected.  Hence $e$ is injective.  Now we claim that $f$ is the zero map, since the previous map in the Mayer-Vietoris sequence is the diagonal map and hence is injective.  This means that $c$ is injective, and so by commutativity $d$ must be injective also, and hence $a$ and $b$ are both the zero maps.  The preceding simplification yields the following diagram.

\begin{equation}\label{eq:connect_sum_ks_diagram}
	\begin{tikzcd}
		& \Z/2 \arrow[d, "r"] \arrow[r, "\Delta"] & \Z/2 \ \ \oplus \ \ \Z/2 \arrow[d, shift left=5ex, "i_B"] \arrow[d, shift right=5ex, "i_A"]  & \\
		0 \arrow[r] & H^4(W,Y) \arrow[d, "q"] \arrow[r,"\cong"] & H^4(A,C)\oplus H^4(B,D) \arrow[d, shift left=5ex, "j_B"'] \arrow[d, shift right=5ex, "j_A"] \arrow[r] & 0\\
		\Z/2 \arrow[r] & H^4(W,F) \arrow[d] \arrow[r]  & H^4(A,F_A)\oplus H^4(B,F_B) \arrow[d, start anchor={[xshift=-5ex]}, end anchor=north west] \arrow[d, start anchor={[xshift=5ex]}, end anchor=north east] \arrow[r] & 0\\
		&  0 & 0 &
	\end{tikzcd}
\end{equation}

We note that the topmost horizontal map is the diagonal map $\Delta\colon x\mapsto (x,x)$.  Before giving the proof of \zcref{thm:connect_sum} we consider a special case in the form of a lemma.

\begin{proof}[Proof of \zcref{thm:connect_sum}]
	We will use the notation from \zcref{stp:connect_sum} throughout.  Consider the pair $(W,Y)=(A\cup B, C\cup D)$.  First, we will consider a special case, namely when $X_2=S^1\times S^3$, $\gamma_2 = S^1\times\{\pt\}$ and $f_2=\Id_{X_2}$.  We will then prove the general case of the theorem via our consideration of the special case.
	
	Let $X_2=S^1\times S^3=X'_2$, let $\gamma_2=S^1\times\{\pt\}$ and let $f_2=\Id_{X_2}$.  Then $X_2\sm (\nu\gamma_2)\cong S^1\times D^3$, so hence $X_\# \cong X_1$ and $f_\# = f_1$.  The map $i_B$ is then Poincar\'{e} dual to the inclusion induced map $H_1(S^1\times S^2\times I) \to H_1(S^1\times D^3 \times I)$ and hence is injective.  It follows by commutativity that $r$ is injective.  Consider any element $x\in H^4(W,F)$.  From the diagram and injectivity of $r$, we can see that there are two possible lifts of $x$ in $H^4(W,Y)$ that differ by $r(z)$ where $\Z/2=\Z/2\langle z\rangle$.  We will define a preferred lift $t(x)$ by specifying that the element $t(x)= (t_1(x),0)\in H^4(A,C)\oplus H^4(B,D)$, which uniquely determines $t$ (the fact that the second coordinate can be made to be zero follows from $i_B$ being surjective in this special case).  Now consider $\varpi\cs(f_1)\in H^4(W,F)$.  By naturality of the Kirby-Siebenmann invariant (\zcref{rmk:ks_naturality}), $q$ maps \[\ks(W,Y)=(\ks(A,C),\ks(B,D))\in H^4(W,Y)\] to $\varpi\cs(f_1)$, and, by the definition of $f_2$, we have that $\ks(B,D)=0$.  This means that~$t(\varpi\cs(f_1))=\ks(W,Y)$.
	
	We have produced a map $\mathfrak{A}\colon H^3(X_1,\partial X_1)\to H^4(A,C)$ as the composition
	\[
	H^3(X_1,\partial X_1)\xrightarrow{\varpi} H^4(W,F)\xrightarrow{t} H^4(A,C)\oplus H^4(B,D)\xrightarrow{\pr_1} H^4(A,C).
	\]
	Similarly, by considering the reverse special case where $X_1=S^1\times S^3=X'_1$, $f_1=\Id_{X_1}$ and where $f_2$ is any homeomorphism $f_2\colon X_2 \to X'_2$, we obtain in the same way a map $\mathfrak{B}\colon H^3(X_2,\partial X_2)\to H^4(B,D)$.
	
	Now consider general $X_1,X_2,X'_1,X'_2,f_1,f_2$.  Write $\alpha:= \mathfrak{A}(\cs(f_1))$ and $\beta:= \mathfrak{B}(\cs(f_2)$.  Consider the element $(\alpha,\beta)\in H^4(A,C)\oplus H^4(B,D)\cong H^4(W,Y)$ and map it down using $q$ to $q(\alpha,\beta)$.  We have that $q(\alpha,\beta)=\varpi\cs(f_\#)$, as by construction $(\alpha,\beta)=(\ks(A,C),\ks(B,D))$ and by naturality of the Kirby-Siebenmann invariant, this must map to $\ks(W,F)=\varpi\cs(f_\#)$ via $q$, since $q$ is inclusion induced.

	So we have constructed a map $P\colon H^3(X_1,\partial X_1)\oplus H^3(X_2,\partial X_2)\to H^3(X_{\#},\partial X_{\#})$ as the composition
	\[
	H^3(X_1,\partial X_1)\oplus H^3(X_2,\partial X_2) \xrightarrow{(\mathfrak{A},\mathfrak{B})} H^4(A,C)\oplus H^4(B,D) \xrightarrow{q} H^4(W,F)\xrightarrow{\varpi^{-1}} H^3(X_{\#},\partial X_{\#})
	\] which sends $(\cs(f_1),\cs(f_2))\mapsto \cs(f_\#)$.
	
	It remains to show that this map is equal to the map $Q$ from \zcref{dfn:Q}).  We show this now.  First, analogously to in \zcref{stp:connect_sum}, consider the long exact sequences for the triples
	\begin{enumerate}[(i)]
			\item $(X_\#,\partial X_\#\sqcup\partial\ol{\nu}\gamma_1,\partial X_\#)$,
			\item $(X_i\sm\nu\gamma_i,\partial(X_i\sm\nu\gamma_i),\partial X_i)$,
	\end{enumerate} and the relative Mayer-Vietoris sequences for the pairs
	\begin{enumerate}[(i)]
			\item $\left(X_\#,\partial X_\#\right)=\left((X_1\sm\nu\gamma_1) \cup (X_2\sm\nu\gamma_2),\partial X_1 \cup \partial X_2\right)$
			\item $\left(X_\#,\partial X_\# \sqcup\partial\ol{\nu}\gamma_1\right)=\left((X_1\sm\nu\gamma_1) \cup (X_2\sm\nu\gamma_2),\partial(X_1\sm\nu\gamma_1)\cup\partial(X_2\sm\nu\gamma_2)\right)$.
	\end{enumerate}  These sequences give the following commutative diagram analogous to (\ref{eq:connect_sum_ks_diagram}) (note that we have simplified the notation by writing $\partial$ on its own to refer to the boundary of a manifold when it is clear by context which manifold is being referred to).
		
	\begin{equation}\label{eq:connect_sum_cs_diagram}
			\begin{adjustbox}{width=\textwidth,center}
					\begin{tikzcd}
							& \Z/2 \arrow[d] \arrow[r, "\Delta"] & \Z/2 \ \ \oplus \ \ \Z/2 \arrow[d, shift left=5ex] \arrow[d, shift right=5ex]  & \\
							0 \arrow[r] & H^3(X_\#,\partial X_\# \sqcup \partial\ol{\nu}\gamma_1) \arrow[d] \arrow[r,"\cong"] & H^3(X_1\sm\nu\gamma_1,\partial)\oplus H^3(X_2\sm\nu\gamma_2,\partial)) \arrow[d, shift left=5ex] \arrow[d, shift right=5ex] \arrow[r] & 0\\
							\Z/2 \arrow[r] & H^3(X_\#,\partial) \arrow[d] \arrow[r]  & H^3(X_1\sm\nu\gamma_1,\partial X_1)\oplus H^3(X_2\sm\nu\gamma_2,\partial X_2) \arrow[d, start anchor={[xshift=-5ex]}, end anchor=north west] \arrow[d, start anchor={[xshift=5ex]}, end anchor=north east] \arrow[r] & 0\\
							&  0 & 0 &
						\end{tikzcd}
				\end{adjustbox}
		\end{equation}
	
	One can see that this diagram (\ref{eq:connect_sum_cs_diagram}) is isomorphic to (\ref{eq:connect_sum_ks_diagram}) using the isomorphism given in \zcref{dfn:casson_sullivan_invariant} (or the analogous isomorphisms).  Since those isomorphisms come from Poincar\'{e} duality and inclusion maps, the two parallel diagrams must commute.  This allows us to reinterpret the construction of our map $P$ on the level of the $4$-manifolds themselves, rather than the mapping cylinders, which then allows us to relate this map to $Q$.
	
	Consider the following diagram which, aside from the lowermost map and group, splits as a direct sum of diagrams.
	\begin{equation}\label{eq:connect_sum_cs_diagram_big}
	\begin{adjustbox}{width=\textwidth,center}
	\begin{tikzcd}
		& & \\
		\begin{matrix} H^3(X_1,\partial X_1 \sqcup \partial \ol{\nu} \gamma_1) \\ \oplus \\H^3(X_2,\partial X_2 \sqcup \partial \ol{\nu} \gamma_2)\end{matrix} \arrow[r,two heads] \arrow[d,"\cong"]& H^3(X_1,\partial)\oplus H^3(X_2,\partial) \arrow[dl,"t",dotted] \arrow[dr,"t'"',dotted] \arrow[r,"\PD", "\cong"'] & \begin{matrix} H_1(X_1) \\ \oplus \\H_1(X_2)\end{matrix}\\
		\begin{matrix} H^3(X_1\sm\nu\gamma_1,\partial) \oplus H^3(\ol{\nu}\gamma_1,\partial) \\ \oplus \\H^3(X_2\sm\nu\gamma_2,\partial)\oplus H^3(\ol{\nu}\gamma_2,\partial)\end{matrix} \arrow[rr,"\cong"', "{\left((\PD,\PD),(\PD,\PD)\right)}"] \arrow[dr,"{(\pr_1,\pr_1)}"']& & \begin{matrix} H_1(X_1\sm\nu\gamma_1)\oplus H_1(\ol{\nu}\gamma_1) \\ \oplus \\H_1(X_2\sm\nu\gamma_2)\oplus H_1(\ol{\nu}\gamma_2)\end{matrix} \arrow[u, two heads] \arrow[dl, "{(\PD^{-1}\circ\pr_1,\PD^{-1}\circ\pr_1)}"] \\
		& H^3(X_1\sm\nu\gamma_1,\partial)\oplus H^3(X_2\sm\nu\gamma_2,\partial) \arrow[d,"\mathfrak{b}\circ\mathfrak{a}^{-1}"] & \\
		& H^3(X_{\#},\partial) & 
	\end{tikzcd}\end{adjustbox}
	\end{equation}
	The two vertical maps are direct sums of the respective Mayer-Vietoris sequences.  The top-left horizontal map comes from the direct sum of the respective long exact sequences of triples.  The lowest vertical map is defined using maps from \zcref{dfn:Q}.  Ignoring the notated lifts $t$ and $t'$ for now, we show that this diagram commutes.  The commuting of the lower triangle is trivial.  We now show that the top rectangle commutes.  Since the rectangle is a direct sum of a two diagrams, it suffices to show that one of the diagram summands commutes.  We write this below, where we will drop the indices (i.e.\ $X_1=X$, etc.).
	
	\[
	\begin{tikzcd}
		H^3(X,\partial) \arrow[r,"\PD","\cong"'] & H_1(X) \\
		H^3(X,\partial X \sqcup \partial\ol{\nu}\gamma) \arrow[u,two heads] \arrow[d,"\cong"]& \\
		H^3(X\sm\nu\gamma,\partial)\oplus H^3(\ol{\nu}\gamma,\partial) \arrow[r,"{(\PD,\PD)}","\cong"']  & H_1(X\sm\nu\gamma)\oplus H_1(\ol{\nu}\gamma) \arrow[uu,two heads]
	\end{tikzcd}
	\]
	
	Let $x\in H^3(X,\partial X\sqcup \partial\ol{\nu}\gamma)$ and let $y$ denote its image in $H^3(X,\partial)$.  By the vertical isomorphism, we see that $x$ maps to an element $(\PD^{-1}\alpha_1,\PD^{-1}\alpha_2)\in H^3(X\sm\nu,\partial)\oplus H^3(\ol{\nu}\gamma,\partial)$, where $\alpha_1$ is the homology class of some curve in $X$ and $\alpha_2$ is trivial or is the homology class of $\gamma$.  Mapping this pair to the right and then up gives the class $\alpha_1+\alpha_2$, which is the homology class of some curve in $X$.  We now need to show that $y$ is Poincar\'{e} dual to $\alpha_1+\alpha_2$.  Since we are using $\Z/2$-coefficients, we will view these cohomology groups as the hom-duals of the respective homology groups.  The element $y$ is defined such that it evaluates on the relative cycle $s$ in $X$ first by considering $s$ as a further relative cycle $s'\in H_3(X,\partial X\sqcup\partial\ol{\nu}\gamma)$ and then evaluating using $x$.  By the isomorphism in the corresponding Mayer-Vietoris sequence in homology, any such cycle splits as a pair $(s_1,s_2)$ of relative cycles in $H_3(X\sm\nu\gamma,\partial)$ and $H_3(\ol{\nu}\gamma,\partial)$, respectively.  Hence, evaluating $x(s')$ is the same as evaluating $\alpha_1(s_1)+\alpha_2(s_2)$, and so $y$ is Poincar\'{e} dual to $\alpha_1+\alpha_2$.  This completes the proof that the diagram commutes. 
	
	Consider again diagram \ref{eq:connect_sum_cs_diagram_big} and let $(x_1,x_2)\in H^3(X_1,\partial)\oplus H^3(X_2,\partial)$.  This pair is mapped by $P$ by sending it down the left of the diagram, first by using the preferred lift $t$, then using the $(\pr_1,\pr_1)$ and then finally using the lowest vertical map.  The lift $t$ is such that $t(x_1,x_2)$ is of the form $\left((t(x_1),0),(t(x_2),0)\right)$.  
	
	The lift $t'$ is defined as in the beginning of \zcref{dfn:Q}: $t'$ is defined as $\mathfrak{c}^{-1}\circ (\PD,\PD)$ post-composed with the inclusion as a direct sum).  This is the unique lift of $(\PD,\PD)(x_1,x_2)$ which is of the form $(\ast,0,\ast,0)$.  This means that the pair $(x_1,x_2)$ is mapped by $Q$ by sending it down the right of the diagram, first by using the lift $t'$ and then $t'(x_1,x_2)$ is mapped down again to $H^3(X_{\#},\partial)$ via $(\PD^{-1}\circ \pr_1,\PD^{-1}\circ\pr_1))$ and the lowest vertical map.
	
	To finish the proof, we need to know that the middle triangle formed by the two lifts, $t$ and~$t'$, commutes.  The fact that the diagram (without the lifts) commutes shows that \[y:=\left((\PD,\PD),(\PD,\PD)\right)(t(x_1,x_2))\] is a lift of $(\PD,\PD)(x_1,x_2)$.  Since $y$ is also of the form $(\ast,0,\ast,0)$ (as $t(x_1,x_2)$ is) the observation in the above paragraph means that $y=t'(x_1,x_2)$, and hence the triangle commutes. 
\end{proof}

\subsection{A connected-sum formula for the Casson-Sullivan invariant}

Before finishing this section we will note the following. Similarly to \zcref{thm:connect_sum}, we also have a formula for the Casson-Sullivan invariant under the actual connected-sum operation.  Since all of the arguments are the same or simpler than for connected-summing over a circle, we will not give many details.  Let $X_1$, $X_2$, $X'_1$ and $X'_2$ be two pairs of compact, connected topological $4$-manifolds and let $f_i\colon X_i\to X'_i$ be a pair of homeomorphisms.  As mentioned previously (see \zcref{rmk:connected-sum} and the preceding paragraph), we can form a connected-sum homeomorphism  $f_\#\colon X_1\# X_2 \to X'_1\# X'_2$.  If we let $q_i\colon X_1\# X_2\to X_i$ denote the collapse maps onto the $i$th connected-summand, then we have the following formula.

\begin{thm}\label{thm:actual_connect_sum}
	Let $X_1$, $X_2$, $X'_1$ and $X'_2$ be two pairs of compact, connected, smooth, orientable $4$-manifolds with $f_i\colon X_i\to X'_i$ a pair of homeomorphisms such that the connected-sum homeomorphism \[f_{\#}:=f_1\# f_2\colon X_1\# X_2\to X'_1\# X'_2\] is defined.  Let $q_i$ be the pair of collapse maps defined above.  Then
	\[\cs(f_\#)=q_1^*\cs(f_1)+q_2^*\cs(f_2).\]
\end{thm}

We will not give the proof here, as it is exactly the same method as for proving \zcref{thm:connect_sum} but with simpler arguments.  In particular, having the degree one collapse is very useful (such collapse maps do not always exist for connected-sums over circles).  This result will only be used for proving \zcref{cor:swap_not_smoothable} in \zcref{sbs:interesting_example}.

\section{Stable realisation of the Casson-Sullivan invariant}\label{sec:stable_realisation}
In this section we will prove \zcref{thm:stable_realisation_theorem} which shows that the Casson-Sullivan invariant can always be realised stably.  The idea is to use \zcref{thm:connect_sum} along with the following proposition, which relies on the work of Cappell-Shaneson and Lee \cite{cappell_shaneson_1971,lee_1970}.  In what follows we will always write $\theta$ for the submanifold $S^1\times \{\pt\}\subset S^1\times S^3$.

\begin{prop}\label{prop_cs_existence_s1xs^3}
	There exists a homeomorphism 
	\[
	\sigma\colon (S^1\times S^3) \# (S^2\times S^2)\to (S^1\times S^3) \# (S^2\times S^2)
	\] such that $\cs(\sigma)$ is the generator of $H^3(S^1\times S^3;\Z/2)$.  Furthermore, $\sigma\vert_{\nu\theta}=\Id_{\nu\theta}$.
\end{prop}

We postpone the proof of this proposition until \zcref{sbs:prop_3.1}.  Now we use \zcref{thm:connect_sum} and \zcref{prop_cs_existence_s1xs^3} to prove the following theorem, from which \zcref{thm:stable_realisation_theorem} will immediately follow.  To parse this theorem, note that we will implicitly use the canonical diffeomorphisms
\[ 
	X_0\#_{\gamma=\theta}((S^1\times S^3)\#(S^2\times S^2))\cong X_0\#(S^2\times S^2)
\] 
and 
\[ 
	X'_0\#_{(f_0)_*(\gamma)=\theta}((S^1\times S^3)\#(S^2\times S^2))\cong X'_0\#(S^2\times S^2).
\]
\begin{thm}\label{thm:stable_realisation_techincal}
	Let $f_0\colon X_0\to X'_0$ be a homeomorphism of compact, connected, smooth, orientable $4$-manifolds and define $\eta_0:=\cs(f_0)$.  Let $\eta\in H^3(X_0,\partial X_0;\Z/2)$, and let $\gamma\subset X$ be a (framed) embedded curve dual to $\eta - \eta_0\in H^3(X_0,\partial X_0;\Z/2)$.  Then, after fixing a framing of $f_0(\theta)$, we have the following cases.
	\begin{enumerate}[(i)]
		\item If $f_0\vert_{\nu\gamma}$ is the identity map, then the connected-sum homeomorphism
		\[f_0\#_{\gamma=\theta}\sigma \colon X_0\# (S^2\times S^2) \to X'_0\# (S^2\times S^2)\]
		is defined and $\cs(f\#_{\gamma=\theta}\sigma)=\eta$.
		\item If $f_0\vert_{\nu \gamma} $ is the Gluck twist map, then the connected-sum homeomorphism
		\[f_0\#_{\gamma=\theta}(\sigma\circ t) \colon X_0\# (S^2\times S^2) \to X'_0\# (S^2\times S^2),\]
		where $t$ denotes the Gluck twist map extended onto the $(S^2\times S^2)$-summand, is defined and $\cs(f_0\#_{\gamma=\theta}\sigma)=\eta$.
	\end{enumerate}
\end{thm}

\begin{proof}[Proof (assuming \zcref{prop_cs_existence_s1xs^3})]
	That we only have to consider the two cases above comes from the exposition at the beginning of \zcref{sbs:casson_sullivan_connect_sum}.
	In both cases the connected-sum homeomorphism (as above) is defined by the last part of \zcref{prop_cs_existence_s1xs^3}. Let $Q$ be the map in the statement of \zcref{thm:connect_sum} for the above decomposition.  In case (i), by \zcref{thm:connect_sum} and \zcref{prop_cs_existence_s1xs^3} we have \[
	\cs(f_0\#_{\gamma=\theta}\sigma) = Q(\cs(f_0),(\cs(\sigma))) = \eta_0 + (\eta -\eta_0)=\eta.
	\]
	Similarly, in case (ii), by \zcref{prop:cs_homomorphism}, \zcref{thm:connect_sum} and \zcref{prop_cs_existence_s1xs^3} we have \begin{align*}
		\cs(f_0\#_{\gamma=\theta}(\sigma\circ t)) &= Q(\cs(f_0),\cs(\sigma\circ t))  \\ &= Q(\cs(f_0),\cs(\sigma)+\cs(t)) \\ &= Q(\cs(f_0),\cs(\sigma)) \\ &= \eta_0 + (\eta - \eta_0) = \eta.
	\end{align*}
	 In the above formulae we have used the isomorphism \[H^3(X_0,\partial X_0;\Z/2)\cong H^3(X_0\#(S^2\times S^2),\partial X_0;\Z/2)\] induced by the map collapsing the $S^2\times S^2$ summand.
\end{proof}

\begin{proof}[Proof of \zcref{thm:stable_realisation_theorem}]
	Using the isomorphism $H^3(X_0,\partial X_0;\Z/2)\cong H^3(X,\partial X;\Z/2)$ induced by collapsing the $S^2\times S^2$ summand, we can consider the given class $\eta\in H^3(X,\partial X;\Z/2)$ as an element $\eta\in H^3(X_0,\partial X_0;\Z/2)$.  By assumption there exists a homeomorphism $f_0\colon X_0\to X'_0$, and applying \zcref{thm:stable_realisation_techincal} to the class $\eta$ immediately gives the result.
\end{proof}

\subsection{Surgery}

Now we fundamentally make use of the following theorem, due to Ronnie Lee.  This subsection will contain Wall's surgery obstruction groups, the $L$-groups, but we will postpone recalling their definition (see \zcref{dfn:l_groups}) until \zcref{sec:unstable_realisation} where the surgery exact sequence will be recalled.  The only fact we need here is that $L_5(\Z[\Z])\cong \Z$ \cite[Theorem 13A.8]{wall_1970}.

\begin{thm}[\cite{lee_1970}]\label{thm:lee}
	The generator of $L_5(\Z[\Z])$ can be realised by a $2\times 2$ matrix.
\end{thm}

The proof of the above theorem has only been available in a handwritten note, which is not easily accessible.\footnote{Note the difference in notation between this paper and the handwritten note for the $L$-groups i.e.\ what we write as $L_n(\Z[\pi])$ is instead written as $L_n(\pi)$.}  For a modernised exposition of this proof, see \cite{galvin_2024} (available on the author's website).

Using this, we receive the following corollary.

\begin{cor}\label{cor:surgery_obstruction_generator_realisation}
	There exists a homeomorphism 
	\[
	\sigma\colon (S^1\times S^3) \# (S^2\times S^2)\to (S^1\times S^3) \# (S^2\times S^2)
	\]
	realising the generator of $L_5(\Z[\Z])$.  More precisely, let $N$ be the standard cobordism between $S^1\times S^3$ and $(S^1\times S^3)\#(S^2\times S^2)$.  Then the surgery problem
	\[
	W = (N \cup_\sigma -N) \to (S^1\times S^3) \times I
	\]
	has surgery obstruction the generator of $L_5(\Z[\Z])$.  Furthermore, $\sigma\vert_{\nu\theta}=\Id_{\nu\theta}$, where $\theta=S^1\times\{\pt\}\subset S^1\times S^3$.
\end{cor}

\begin{proof}
	By \zcref{thm:lee} we have a matrix $M$ with entries in $\Z[\Z]$ which represents the generator of $L_5(\Z[\Z])$.  By \cite[Theorem 2]{stong_wang_2000} there exists two pseudo-isotopy classes of self-homeomorphisms of $(S^1\times S^3)\# (S^2\times S^2)$ which induce the map $M$ on $H_2((S^1\times S^3)\# (S^2\times S^2);\Z[\Z])$, one which preserves the spin structures on $S^1\times S^3$ and one which swaps them.  Let $\sigma$ be a representative self-homeomorphism that fixes the spin structures. As noted in the beginning of \zcref{sec:stable_realisation}, $\sigma$ can be isotoped such that either~$\sigma\vert_{\nu\theta}=\Id_{\nu\theta}$ or $\sigma$ restricts to the Gluck twist map along $\nu\theta$. Since the latter induces the non-trivial map on spin structures and $\sigma$ fixes the spin structures, it must be that the former holds.  It follows from \cite[Theorem 3.1]{cappell_shaneson_1971} that $W$ has surgery obstruction the generator of $L_5(\Z[\Z])$.
\end{proof}

\subsection{Proof of \zcref{prop_cs_existence_s1xs^3}}\label{sbs:prop_3.1}

Throughout this subsection, let $\sigma\colon (S^1\times S^3) \# (S^2\times S^2)\to (S^1\times S^3) \# (S^2\times S^2)$ and $W$ be as in \zcref{cor:surgery_obstruction_generator_realisation}.  The aim is to prove \zcref{prop_cs_existence_s1xs^3} by showing that that this $\sigma$ has $\cs(\sigma)$ the generator of $H^3((S^1\times S^3)\# (S^2\times S^2);\Z/2)$.  This will be done by performing operations on $W$ given in \zcref{cor:surgery_obstruction_generator_realisation} and keeping track of what happens to the relative Kirby-Siebenmann invariant of the cobordism along the way.  We begin with some lemmas.

In what follows, we will also need the following lemma.

\begin{lem}\label{lem:sw_s1xs3}
	Every homeomorphism $f\colon S^1\times S^3 \to S^1\times S^3$ is pseudo-smoothable and hence has $\cs(f)=0$.
\end{lem}

\begin{proof}
	Stong-Wang classified homeomorphisms of $4$-manifolds $M$ with $\pi_1(M)\cong \Z$ up to pseudo-isotopy \cite[Theorem 2]{stong_wang_2000}.  This directly gives us that there are four homeomorphisms on $S^1\times S^3$ up to pseudo-isotopy represented by $(1)$ the identity map; $(2)$ conjugation on the $S^1$-factor composed with the reflection map on the $S^3$-factor; $(3)$ the corresponding Gluck twist map $S^1 \times S^3 \to S^1 \times S^3$ (\ref{eq:gluck_twist}); and $(4)$ the composition of the two previously stated non-trivial maps.  All of these maps are clearly smooth, hence every self-homeomorphism of $S^1\times S^3$ is pseudo-smoothable.
\end{proof}

\begin{lem}\label{lem:ks=cs}
	Under the standard identification \[
	H^4(W,\partial W;\Z/2) \cong H^4(M_\sigma,\partial M_{\sigma};\Z/2)\xrightarrow{\varpi^{-1}} H^3((S^1\times S^3)\# (S^2\times S^2);\Z/2)
	\] $($see \zcref{dfn:casson_sullivan_invariant}$)$ we have that $\ks(W,\partial W)=\varpi\cs(\sigma)$ where $M_\sigma$ denotes the mapping cylinder.
\end{lem}

\begin{proof}
	Note that $N$ is smoothable relative to the standard smooth structure on the boundary.  It is clear that we have a homeomorphism relative to the boundary 
	\[
	W\approx N \cup_{(S^1\times S^3)\# (S^2\times S^2)} \left(((S^1\times S^3)\#(S^2\times S^2))\times I\right) \cup_{\sigma}-N.
	\] i.e.\ inserting a product $(S^1\times S^3\#S^2\times S^2)\times I$ to the right end of $N$ does not change the relative homeomorphism type of $W$.  We now see that the identification \[H^4(W,\partial W;\Z/2) \xrightarrow{\cong} H^4(M_\sigma,\partial M_\sigma;\Z/2)\] gives $\ks(W,\partial W)=\ks(M_\sigma)=\varpi\cs(\sigma)$ where $M_\sigma$ is the mapping cylinder of $\sigma$ (note that here we write ``$=$" as there is only one isomorphism between groups isomorphic to $\Z/2$.).
\end{proof}

\begin{lem}[{\cite[Proof of 11.6A]{freedman_quinn_1990}}]\label{lem:surgery_obstruction_connect_sum}
	Let $S^1\times E_8 \to S^1\times S^4$ be the surgery problem which has surgery obstruction the inverse of the generator of $L_5(\Z[\Z])$, let $\gamma\subset W$ be an embedded curve representing the generator of $H_1(W)$, and let $\gamma'$ be the embedded curve $S^1\times\{\pt\}\subset S^1\times E_8$.  Then the connected-sum surgery problem
	\[
	W\#_{\gamma=\gamma'}(S^1\times E_8)\to (S^1\times S^3)\times I
	\] 
	has vanishing surgery obstruction.
\end{lem}

\begin{rmk}\label{rmk:shaneson_splitting}
	Note that the generator of $L_5(\Z[\Z])$ being representable by a surgery problem $S^1\times E_8 \to S^1\times S^4$ (as in \zcref{lem:surgery_obstruction_connect_sum}) follows from Shaneson splitting \cite[Theorem 5.1]{shaneson_1969}.
\end{rmk}

We can now prove the proposition.

\begin{proof}[Proof of \zcref{prop_cs_existence_s1xs^3}]
	Start with the cobordism $W$ and modify the surgery obstruction to be trivial using \zcref{lem:surgery_obstruction_connect_sum}.  Since the result has vanishing surgery obstruction, we can surger it relative to the boundary to an $s$-cobordism $W'$.  Since $\pi_1(W')\cong \pi_1(S^1\times S^3)\cong\Z$ is good we can use the $s$-cobordism theorem \cite[Theorem 7.1A]{freedman_quinn_1990}, which gives that $W'$ is homeomorphic to the mapping cylinder $M_f$ for some homeomorphism $f\colon S^1\times S^3\to S^1\times S^3$.  By \zcref{lem:sw_s1xs3}, $\cs(f)=0$ and hence $\ks(W',\partial W')=0$.
	
	We now need to keep track of how we modified the Kirby-Siebenmann invariant throughout this process.  The main tool for doing so will be the long exact sequence of the triple (see (\ref{eq:les_triple})) with $\Z/2$-coefficients suppressed throughout.  When we used \zcref{lem:surgery_obstruction_connect_sum} to kill the surgery obstruction, the connected-sum over a circle was induced by a relative cobordism $C$ between $W\sqcup (E_8\times S^1)$ and $W\#_{\gamma=\gamma'}(S^1\times E_8)$.  Let $W_\#:= W\#_{\gamma=\gamma'}(S^1\times E_8)$.  This relative cobordism $C$ is formed by attaching a single $1$-handle (at one point on $\gamma$ and at one point on $\gamma'$) and then a single $2$-handle which attaches by going along $\gamma$, then the $1$-handle, then $\gamma'$, and then back along the $1$-handle.  From the long exact sequence of the triple $(C,W\sqcup S^1\times E_8, \partial W)$ we have the sequence
	\[
	H^4(C,W\sqcup S^1\times E_8) \to H^4(C,\partial W) \to H^4(W\sqcup S^1\times E_8,\partial W) \to H^5(C,W\sqcup S^1\times E_8).
	\]
	The outer groups $H^4(C,W\sqcup S^1\times E_8)$ and $H^5(C,W\sqcup S^1\times E_8)$ must both vanish since the cobordism $C$ was made by only attaching $1$- and $2$-handles.  Hence we get an isomorphism 
	\begin{equation}\label{eq:cobordism_les_1}
		H^4(C,\partial W) \xrightarrow{\cong} H^4(W\sqcup S^1\times E_8,\partial W)\cong H^4(W,\partial W)\oplus H^4(S^1\times E_8)\cong \Z/2\oplus \Z/2.
	\end{equation}
	 By considering the long exact sequence of the triple $(C,W_{\#},\partial W_{\#})$ we have the following commutative diagram
	\[\begin{tikzcd}
		H^4(C,W_{\#}) \arrow[r] \arrow[d,"\cong"] & H^4(C,\partial W_{\#}) \arrow[r] \arrow[d,"\cong"] & H^4(W_{\#},\partial W_{\#}) \arrow[d,"\cong"] \\
		\Z/2 \arrow[r,"{1\mapsto(1,1)}"] & \Z/2\oplus \Z/2 \arrow[r,"{(a,b)\mapsto a+b}"] & \Z/2.
	\end{tikzcd}
	\]
	The leftmost vertical isomorphism again comes from the relative handle decomposition of $C$, and the rightmost vertical isomorphism is clear by the definition of $W_\#$.  The middle vertical isomorphism comes from (\ref{eq:cobordism_les_1}) and the fact that $H^*(C,\partial W_{\#})\cong H^*(C,\partial W)$.  The leftmost horizontal map can be seen to be the diagonal map since the generator of $H^4(C,W_{\#})$ is Poincar\'{e}-Lefschetz dual to the annulus with boundaries homologous to the generators of $H_1(W)$ and $H_1(S^1\times E_8)$.  The rightmost horizontal map is then given by exactness.  Using this and naturality of the Kirby-Siebenmann invariant, we can deduce that \[\ks(W_{\#},\partial W_{\#})=\ks(W,\partial W)+\ks(S^1\times E_8),\]
	where we can naturally identify the groups as there is only one isomorphism between groups isomorphic to $\Z/2$.
	
	Using a similar argument, one can see that the surgeries used to surger $W_{\#}$ to the $s$-cobordism $W'$ do not alter the Kirby-Siebenmann invariant (the relative handle decomposition for the relative cobordism given by the trace of any given surgery consists of only a single handle, so the computation is greatly simplified).  Hence, since by naturality $\ks(S^1\times E_8)=1$, \[\ks(W',\partial W')=\ks(W_{\#},\partial W_{\#})=\ks(W,\partial W) +\ks(S^1\times E_8)= \ks(W,\partial W)+1\]  As we already concluded that $\ks(W',\partial W')=0\in \Z/2$, this implies that $\ks(W,\partial W)=1$, and hence, by \zcref{lem:ks=cs}, $\cs(\sigma)$ is the generator of $H^3((S^1\times S^3)\# (S^2\times S^2);\Z/2)$, as claimed.
\end{proof}

\subsection{An interesting example}\label{sbs:interesting_example}

We can use the objects and tools that we have developed so far to illustrate an interesting example involving the Casson-Sullivan invariant that demonstrates its dependence on smooth structures, as is expected by \zcref{sbs:cs_dependence_smooth_structures}.  This will also prove \zcref{cor:swap_not_smoothable}.

\begin{exa}\label{ex:interesting}
	Let $X= (S^1\times S^3)\# (S^1\times S^3) \# (S^2\times S^2)$.  Let $f:=\Id_{S^1\times S^3}\#\sigma\colon X\to X$ be a homeomorphism where $\sigma$ denotes the homeomorphism constructed by \zcref{cor:surgery_obstruction_generator_realisation}, and let~$g\colon X\to X$ be the homeomorphism which swaps the $S^1\times S^3$ summands and leaves the $S^2\times S^2$ summand fixed.
	
	Let $\mathscr{S}$ denote the standard smooth structure on $X$ and (in an abuse of notation) also denote the induced formal smooth structure (see \zcref{rmk:formal_vs_informal}).  Then $g\colon X_\mathscr{S} \to X_\mathscr{S}$ is (isotopic to) a diffeomorphism, and hence $\cs(g)=0$ with respect to the smooth structure $\mathscr{S}$.  However, let $f^*(\mathscr{S})$ be the smooth structure on $X$ induced by $f$.  Using \zcref{prop:cs_dependence_ss} we can compute the Casson-Sullivan invariant of $g$ with respect to the smooth structure $f^*(\mathscr{S})$.  As in \zcref{prop:cs_dependence_ss}, let~$a\in [X,\TOP\!/\!\tO]$ be the unique element such that $a\cdot \mathscr{S} = f^*(\mathscr{S})$.  We then have that
	\begin{align*}
		\cs(g)&= \delta(g,f^*(\mathscr{S})) \\
		&= a + g^*(a)+\delta(g,\mathscr{S}) \\
		&= a+ g^*(a) \\
		&= \delta(f,\mathscr{S})+g^*(\delta(f,\mathscr{S})) \\
		&= \cs(f)+g^*(\cs(f)),
	\end{align*}
	which is equal to the element $(1,1)\in \Z/2\oplus \Z/2\cong H^3(X;\Z/2)$ by \zcref{thm:actual_connect_sum} and \zcref{prop_cs_existence_s1xs^3}.  So $g$ is no longer smoothable with respect to the smooth structure $f^*(\mathscr{S})$.
\end{exa}

\zcref{cor:swap_not_smoothable} follows immediately from this example.

\begin{proof}[Proof of \zcref{cor:swap_not_smoothable}]
	Take $\mathscr{S}':=f^*(\mathscr{S})$.  Let $X_{\mathscr{S}}$ denote $X$ with the standard smooth structure, and let $X_{\mathscr{S}'}$ denote $X$ with the smooth structure $\mathscr{S}'$.  By definition, $f\colon X_{\mathscr{S}}\to X_{\mathscr{S}'}$ is a diffeomorphism, so the two smooth structures are diffeomorphic (see \zcref{dfn:diffeomorphism}).  By the calculation in \zcref{ex:interesting}, $g$ is not stably pseudo-smoothable with respect to the smooth structure~$\mathscr{S}'$, since its Casson-Sullivan invariant is non-trivial.
\end{proof}

\section{Unstable realisation of the Casson-Sullivan invariant}\label{sec:unstable_realisation}

We now aim to realise the Casson-Sullivan invariant unstably in some cases.  In \zcref{sbs:ses} we will introduce the background necessary to give the proof.  In \zcref{sbs:forming_mapping_cylinders} and \zcref{sbs:realisation_proofs} we will prove \zcref{thm:unstable_realisation_theorem} and \zcref{thm:non_pseudo_smoothable_but_hom_smoothable}.  In \zcref{sbs:realisation_condition_examples} we will describe when \zcref{thm:unstable_realisation_theorem} and \zcref{thm:non_pseudo_smoothable_but_hom_smoothable} apply, and in \zcref{sbs:partial_realisation} we will describe some techniques for partial realisation of the Casson-Sullivan invariant.

\subsection{The surgery exact sequence}\label{sbs:ses}

For a reference, see \cite[Chapter 13]{ranicki_2002} in high-dimensions and for the $4$-dimensional surgery exact sequence see \cite[Chapter 11]{freedman_quinn_1990} or \cite[Chapter 22]{behrens_kalmar_kim_powell_ray_2021}.  We will use the surgery sequence in both the smooth and topological categories, so to simplify the notation we will use $\CAT$ to stand in for both $\DIFF$ and $\TOP$.

The surgery exact sequence consists of three types of objects, which we now recall: the structure set, the normal invariants, and the $L$-groups.

\begin{dfn}\label{dfn:structure_set}
	Let $(M,\partial M)$ be a connected $\CAT$ $n$-manifold with (potentially empty) boundary.  Then the $\CAT$-\emph{structure set} of $M$, denoted as $\mathcal{S}_{\CAT}(M,\partial M)$ is the set of all equivalence classes of pairs $(N,\varphi)$ where $N$ is a $\CAT$-$n$-manifold and $\varphi$ a homotopy equivalence $\varphi\colon N\xrightarrow{\simeq} M$ such that the restriction $\varphi\vert_{\partial N}\colon \partial N \to \partial M$ is a $\CAT$-isomorphism.  The equivalence relation is that $(N,\varphi)\sim (N',\varphi')$ if there exists a (relative) $\CAT$-$h$-cobordism $(W;N,N')$ with a homotopy equivalence
	\[
	\Phi \colon W \to M\times [0,1]
	\]
	such that $\Phi\vert_{N}=\varphi\colon N\to M\times\{0\}$ and $\Phi\vert_{N'}=\varphi'\colon N' \to M\times\{1\}$ and such that \[\Phi\vert_{\partial N\times [0,1]}=(\varphi\vert_{\partial N}\times \Id_{[0,1]})\colon \partial N\times[0,1] \to \partial M\times [0,1].\]
	
	Similarly, there is a \emph{simple} $\CAT$-\emph{structure set}, denoted $\mathcal{S}_{\CAT}^s$, which is defined analogously to the regular structure set but with all homotopy equivalences replaced with simple homotopy equivalences (and hence all $h$-cobordisms replaced with $s$-cobordisms).  Sometimes will we write $\mathcal{S}_{\CAT}^h$ to specify that we mean the regular structure set.
\end{dfn}

We now define the normal invariants.

\begin{prop}
	Let $G(k)$ be the monoid of homotopy equivalences $S^{k-1}\to S^{k-1}$, let $G$ denote the direct limit of the inclusions $G(k)\hookrightarrow G(k+1)$.  Then there are fibration sequences
	\[
	G/\!\tO \to \BO \to \BG \to \mathcal{B}(G/\!\tO)
	\]
	and
	\[
	G/\!\TOP \to \BTOP \to \BG \to \mathcal{B}(G/\!\TOP).
	\]
\end{prop}

\begin{proof}
	See \cite[Chapter 9]{ranicki_2002} and \cite[Chapter 11]{freedman_quinn_1990}.  Note that again $G/\!\tO$ and $G/\!\TOP$ are defined as the homotopy fibres of $\tO\to G$ and $\TOP\to G$, respectively, and that $\mathcal{B}(G/\!\tO)$, $\mathcal{B}(G/\!\TOP)$ and the rightmost fibrations exist by Boardman-Vogt \cite{boardman_vogt_1968}.
\end{proof}

\begin{dfn}\label{dfn:normal_invariants}
	Let $(M,\partial M)$ be a connected $n$-manifold with (potentially empty) boundary $\partial M$ which already has a $\CAT$-structure.  Then the \emph{normal invariants} of $M$, denoted as $\mathcal{N}_{\CAT}(M,\partial M)$ is the set of homotopy classes of maps $[(M,\partial M), G/\tO]$ if $\CAT=\DIFF$ or $[(M,\partial M), G/\TOP]$ if $\CAT=\TOP$.
\end{dfn}

\begin{rmk}\label{rmk:normal_invariants}
	The above definition is one way of defining the normal invariants for a given manifold.  An equivalent formulation is that a $\CAT$ normal invariant for a $\CAT$-manifold $(M,\partial M)$ is a $\CAT$-manifold $(N,\partial N)$ together with a so-called \emph{degree one normal map} $f\colon N \to M$ which restricts to a $\CAT$-isomorphism on $\partial N$.  For more information, see \cite[Chapter 9]{ranicki_2002}.
\end{rmk}

We now define Wall's surgery obstruction groups.

\begin{dfn}\label{dfn:l_groups}
	Let $n\in\Z$ and $\pi$ a finitely-presented group.  Then the quadratic $L$\emph{-group} $L_n(\Z[\pi])$ is defined differently depending on the residue of $n$ modulo $4$.
	
	Even case (if $n\equiv 0,2\pmod{4}$): $L_n(\Z[\pi])$ is defined as the set of stable equivalence classes of $(-1)^{n/2}$-quadratic forms over stably free $\Z[\pi]$-modules.
	
	Odd case (if $n\equiv 1,3\pmod{4}$): $L_n(\Z[\pi])$ is defined as the set of stable equivalence classes of $(-1)^{(n-1)/2}$-quadratic formations over stably free $\Z[\pi]$-modules.
\end{dfn}

For a reference on (quadratic) forms and formations see \cite[\S 1.6]{ranicki_1981}.  There are also \emph{simple} $L$-groups, denoted by $L_n^s(\Z[\pi])$.  We will not give the definition here, but for a reference see \cite[\S 9.10]{luck_2023}.

\begin{rmk}\label{rmk:l_group_notation}
	The notation used here for the $L$-groups matches up with the notation used for $K$-theory.  In principle one could consider $L$-groups of arbitrary rings with involution, but for our purposes we will only consider the $L$-theory of group rings with the standard involution and trivial orientation character.  One should be careful when reading other sources, as $L_n(\Z[\pi])$ is often written instead as $L_n(\pi)$.
\end{rmk}

We can now state the surgery exact sequence, due to Browder, Novikov, Sullivan and Wall (and Freedman-Quinn in dimension $4$).

\begin{thm}[{\cite[Theorem 10.8]{wall_1970}, \cite[Theorem 11.3A]{freedman_quinn_1990}}]\label{thm:surgery_exact_sequence}
	Let $M$ be a compact, connected, oriented $n$-dimensional $\CAT$-manifold.  Then for $\CAT=\DIFF$ and $n\geq 5$ we have the following exact sequence of pointed sets, which continues to the left in the obvious manner:
	\[
	\begin{tikzcd}
		\cdots \arrow[r] & L_{n+2}(\Z[\pi]) \arrow[r]  & \mathcal{S}_{\DIFF}(M\times I,\partial) \arrow[r] & \left[(M\times I,\partial), G/\tO\right] \arrow[r] & L_{n+1}(\Z[\pi]) \\
		& \arrow[r] & \mathcal{S}_{\DIFF}(M,\partial) \arrow[r] & \left[(M,\partial), G/\tO\right] \arrow[r] & L_{n}(\Z[\pi])		.
	\end{tikzcd}
	\]
	And for $\CAT=\TOP$ and $n\geq 5$ (and $n=4$ provided that $\pi$ is a good group; see \zcref{rmk:good}) we have the following analogous exact sequence of abelian groups:
	\[
	\begin{tikzcd}
		\cdots \arrow[r] & L_{n+2}(\Z[\pi]) \arrow[r]  & \mathcal{S}_{\TOP}(M\times I,\partial) \arrow[r] & \left[(M\times I,\partial), G/\TOP\right] \arrow[r] & L_{n+1}(\Z[\pi]) \\
		& \arrow[r] & \mathcal{S}_{\TOP}(M,\partial) \arrow[r] & \left[(M,\partial), G/\TOP\right] \arrow[r] & L_{n}(\Z[\pi]).		
	\end{tikzcd}
	\]
	Furthermore, there are simple versions of both of these exact sequences where the $L$-groups are replaced by the simple $L$-groups and the structure sets are replaced by the simple structure sets.
\end{thm}

\subsection{Forming mapping cylinders from the surgery exact sequence}\label{sbs:forming_mapping_cylinders}

Now we specialise to the case of interest.  Let $X$ be a compact, connected, smooth $4$-manifold with boundary $\partial X$.  The obstruction to lifting an element of $\mathcal{N}_{\TOP}$ to $\mathcal{N}_{\DIFF}$ is given by the map $\xi_*$ induced by the fibration 
\[
G/\tO \to G/\TOP \xrightarrow{\xi} \mathcal{B}(\TOP\!/\!\tO).
\]

We will consider the following augmented part of the $\TOP$ and $\DIFF$ surgery exact sequences for $X$:

\[
\begin{tikzcd}
	& \left[(X\times I,\partial), G/\tO\right] \arrow[d] & \\
	\mathcal{S}^s_{\TOP}(X\times I, \partial) \arrow[r] & \left[(X\times I,\partial),G/\TOP\right] \arrow[d, "\xi_*"] \arrow[r,"\theta"] & L^s_5(\Z[\pi]) \\
	& \left[(X\times I,\partial), \mathcal{B}(\TOP\!/\!\tO)\right] &
\end{tikzcd}
\]

The idea is to construct a mapping cylinder for a homeomorphism with non-trivial Casson-Sullivan invariant from this sequence.  The way we will do this is by finding an element $N\in \mathcal{N}_{\TOP}(X\times I,\partial)$ which has vanishing surgery obstruction $\theta(N)$, but has $\xi_*(N)\neq 0$.  First, we need to understand the map $\xi_*$ more, which we do via the following two lemmas.

\begin{lem}\label{lem:normal_invariants_isomorphism}
	We have an isomorphism \[[(X\times I,\partial), G/\TOP] \cong H^2(X\times I,\partial;\Z/2)\oplus H^4(X\times I,\partial;\Z).\]
\end{lem}

\begin{proof}
	This follows from the work of Sullivan \cite{sullivan_1967}, which can be found in \cite{ranicki_2002}.  Instead though, we will refer to \cite{madsen_milgram_1979} for its exposition on this topic.  In particular, it follows from \cite[Remark 4.36]{madsen_milgram_1979} that 
	\[
	(G/\TOP[2])_5 \simeq K(\Z_{(2)},4) \times K(\Z/2,2) 
	\]
	where $G/\TOP[2]$ denotes the $2$-localisation of $G/\TOP$ and $(G/\TOP)_5$ denotes its $5$th Postnikov stage.
	From this, one can use \cite[4.35]{madsen_milgram_1979} to see that 
	\[
	(G/\TOP)_5 \simeq K(\Z,4) \times K(\Z/2,2) 
	\]
	and, together with the fact that $X\times I$ is a $5$-dimensional CW-complex, a standard obstruction theoretic argument completes the proof of the lemma.
\end{proof}

\begin{lem}\label{lem:exact_sequence_commutative_diagram}
	We have the following commutative diagram.
	\[
	\begin{tikzcd}
		\left[(X\times I,\partial),G/\TOP\right] \arrow[r,"\cong"] \arrow[d, "\xi_*"] &  H^2(X\times I,\partial;\Z/2)\oplus H^4(X\times I,\partial;\Z) \arrow[d, "m"] \\
		\left[(X\times I,\partial), \mathcal{B}(\TOP\!/\!\tO)\right] \arrow[r,"\cong"] & H^4(X\times I,\partial;\Z/2)
	\end{tikzcd}
	\]
	where $m$ is the map sending $(x,y) \mapsto (\red_2(y))$.
\end{lem}

\begin{proof}
	First, note that the bottom horizontal isomorphism is given by there being a $7$-connected map from $\mathcal{B}(\TOP\!/\!\tO)\to K(\Z/2,4)$ (see \zcref{sbs:casson_sullivan_definition}), and the top horizontal isomorphism is given by \zcref{lem:normal_invariants_isomorphism}. 
	
	Morita \cite[Proposition 3]{morita_72} gives us that $m$ is the map sending $(x,y)\mapsto \Sq^2(x)+\red_2(y)$.  Now it suffices to show that $\Sq^2(x)=0$.  We have the following commutative diagram (by the naturality of Steenrod squares):
	\[\begin{tikzcd}
		\widetilde{H}^2(X\times I,\partial; \Z/2) \arrow[r,"\Sq^2"] & \widetilde{H}^4(X\times I,\partial; \Z/2) \\
		\widetilde{H}^1(X;\Z/2) \arrow[r,"\Sq^2"] \arrow[u,"\cong"] & \widetilde{H}^3(X;\Z/2) \arrow[u,"\cong"']
	\end{tikzcd}\]
and the lower horizontal map must vanish since $\Sq^i\colon H^j(-;\Z/2)\to H^{j+i}(-;\Z/2)$ is the zero map for $i<j$.
\end{proof}

Now we construct the mapping cylinder.

\begin{prop}\label{prop:mapping_cylinder_construction}
	Let $X$ be a compact, connected, smooth $4$-manifold with good fundamental group.  Let $N\in\mathcal{N}_{\TOP}(X\times I,\partial)$ be an element of the normal invariants which has $\theta(N)=0$.  Then $N$ can be lifted to an element of the structure set which is homeomorphic to a mapping cylinder $M_{f}=(X\times I)\cup_{f} X$ for some homeomorphism $f$.
\end{prop}

\begin{proof}
	Assume that we have an element $N\in\mathcal{N}_{\TOP}(X\times I,\partial)$.  By \zcref{rmk:normal_invariants} this means we can consider $N$ to be a manifold $(N,\partial N)$ together with $\partial N$ homeomorphic to $\partial (X\times I)$, and hence $\partial N$ has an induced smooth structure given by $X$. Further we assume that $\theta(N)=0$. This means that we can lift $N$ to an element (also denoted by $N$) in the structure set $\mathcal{S}^s_{\TOP}(X\times I, \partial)$.  The homeomorphism $\partial N \approx \partial (X\times I)\cong X\times \{0\}\cup \partial X\times I \cup X\times\{1\}$ induces a decomposition $\partial N = \partial_+ N \cup \partial_0 N \cup \partial_- N$.  If we define $D:= \partial (\partial_+ N)$, then we may assume that $\partial_0 N \cong D\times I$.  By the relative $s$-cobordism theorem \cite[Theorem 7.1A]{freedman_quinn_1990} (which applies since $X$ has good fundamental group by assumption) there exists a homeomorphism relative to $\partial_+N \cup (D\times I)$ 
	\[
	(N,\partial_+ N) \xrightarrow{\approx} (\partial_+ N \times I,\partial_+ N\times\{0\})
	\]
	such that this homeomorphism restricts to the identity on $\partial_+N \cup (D\times I)$ .
	Let $\widetilde{f}$ denote the restriction of this homeomorphism to $\partial_- N\to \partial_+N\times\{1\}$.  Since $N\in\mathcal{S}^s_{\TOP}(X\times I, \partial)$, we also have a (simple) homotopy equivalence, restricting to a homeomorphism on the boundary
	\[
	(N;\partial_+ N,\partial_0 N,\partial_- N) \xrightarrow{\simeq} (X\times I; X\times \{0\},\partial X\times I,X\times\{1\})
	\]
	Putting this together, we have the commutative diagram
	\begin{equation*}\label{eq:structure_set_commutative_diagram}
		\begin{tikzcd}
			\partial_- N \arrow[r,"\widetilde{f}"] \arrow[d,"\approx"] & \partial_+ N \arrow[d,"\approx"] \\
			X \arrow[r,"f"] & X 
		\end{tikzcd}
	\end{equation*}
	where $f$ is defined such that the diagram commutes.
	
	It follows that our constructed element $N\in\mathcal{S}^s_{\TOP}(X\times I, \partial)$ is homeomorphic to the mapping cylinder $M_{f}:=(X\times I) \cup_{f} X$, restricting to a diffeomorphism on the boundary by construction.
\end{proof}

\subsection{Proofs of \zcref{thm:unstable_realisation_theorem} and \zcref{thm:non_pseudo_smoothable_but_hom_smoothable}}\label{sbs:realisation_proofs}

We now define the technical ``realisability condition" which was mentioned in the introduction.

\begin{dfn}\label{dfn:cs_realisation_condition}
	Let $X$ be a closed, connected, smooth, orientable $4$-manifold with $\pi_1(X)\cong\pi$ where $\pi$ is a good group.  We say that $X$ satisfies the \emph{Casson-Sullivan realisability condition} if the surgery obstruction map (after reidentifying the normal invariants using \zcref{lem:normal_invariants_isomorphism})
	\[
	\theta\colon H^2(X\times I,\partial;\Z/2)\oplus H^4(X\times I,\partial;\Z) \to L_5(\Z[\pi])
	\]
	is such that for every $y\in H^4(X\times I,\partial;\Z)$ there exists an $x\in H^2(X\times I,\partial;\Z/2)$ such that $\theta(x,y)=0$.
\end{dfn}

We begin with a simple but essential observation.

\begin{lem}\label{lem:mapping_cylinder_casson_sullivan}
	Let $N$ be an element as in \zcref{prop:mapping_cylinder_construction} and let $M_f$ be the associated mapping cylinder.  Then $\cs(f)=\varpi^{-1}\xi_*(N)$.
\end{lem}

\begin{proof}
	Note that $\xi_*(N)$ is the obstruction to lifting $N$ to an element of the smooth structure set $\mathcal{S}^s_{\DIFF}(X\times I, \partial)$ and hence $\xi_*(N)=\ks(M_f,\partial M_f)$.  The lemma follows by the definition of the Casson-Sullivan invariant (\zcref{dfn:casson_sullivan_invariant}).
\end{proof}

Now it only remains to prove the main theorem, but most of the work has already been done.

\begin{proof}[Proof of \zcref{thm:unstable_realisation_theorem}]
	Let $X$ be as in the statement of the theorem and let $\eta\in H^3(X;\Z/2)$.  
	
	Recall the diagram from \zcref{lem:exact_sequence_commutative_diagram}. We claim that for any $z\in H^4(X\times I,\partial;\Z/2)$ there exists a $y\in H^4(X\times I,\partial; \Z)$ such that $m(0,y)=\red_2(y)=z$.  This is readily seen by considering the Bockstein exact sequence
	\[
	H^4(X\times I,\partial;\Z)\xrightarrow{\red_2} H^4(X\times I,\partial;\Z/2)\to H^5(X\times I,\partial;\Z)\xrightarrow{\cdot 2} H^5(X\times I,\partial;\Z).
	\]  The first map is surjective since the last map is injective.  This completes the proof of the claim.\footnote{We thank the anonymous referee for suggesting this simple proof of the claim.}

So, given any $\eta\in H^3(X,\partial X;\Z/2)$, we can find an element $y\in H^4(X\times I,\partial;\Z)$ such that \[\varpi^{-1}m(0,y)=\eta.\]  By the realisability condition, there exists an $x\in H^2(X\times I,\partial ;\Z/2)$ such that $\theta(x,y)=0$.  Define~$N_\eta$ such that $N_\eta$ maps to $(x,y)$ under the isomorphism given in \zcref{lem:normal_invariants_isomorphism}.  By \zcref{lem:exact_sequence_commutative_diagram}, it follows that $\xi_*(N_\eta)=\varpi\eta$.  Hence, by \zcref{prop:mapping_cylinder_construction} and \zcref{lem:mapping_cylinder_casson_sullivan}, we have that there exists a homeomorphism $f\colon X\to X$ with $\cs(f)=\eta$.
\end{proof}

We finish this section by proving \zcref{thm:non_pseudo_smoothable_but_hom_smoothable}, which follows easily from the construction of the non-pseudo-smoothable homeomorphisms produced by \zcref{thm:unstable_realisation_theorem}.

\begin{proof}[Proof of \zcref{thm:non_pseudo_smoothable_but_hom_smoothable}]
	Let $X$ be as in the statement of the theorem.  Then, as in the above proof of \zcref{thm:unstable_realisation_theorem}, for every non-zero $\eta\in H^3(X;\Z/2)$ there exists a mapping cylinder $M_f=(X\times I)\cup_f X$ such that $\cs(f)=\eta\neq0$.  Hence $f$ is not stably pseudo-isotopic to a diffeomorphism by \zcref{prop:cs_pseudo-isotopy_inv} and \zcref{prop:cs_stable_1}.  However, now also note that $M_f\in\mathcal{S}^s_{\TOP}(X\times I, \partial)$ and hence we have a homotopy equivalence
	\[
	\left((X\times I)\cup_f X;X\times \{0\}, X\times\{1\}\right) \simeq (X\times I;X\times\{0\},X\times\{1\})
	\]
	which by construction restricts to the identity $\Id_X$ on $X\times\{0\}$ and $f$ on $X\times \{1\}$.  By post-composing this homotopy equivalence with the projection to $X$, this produces a homotopy between~$f$ and~$\Id_X$.
\end{proof}

\subsection{Applications}\label{sbs:realisation_condition_examples}

In light of the previous subsection, it is natural to ask for which $4$-manifolds do the results \zcref{thm:unstable_realisation_theorem} and \zcref{thm:non_pseudo_smoothable_but_hom_smoothable} apply.  This subsection is devoted to giving such examples.  In all that follows, let $X$ be a closed, connected, smooth, orientable $4$-manifold with $\pi_1(X)\cong\pi$.

Let $\Gamma$ be a finite cyclic group.  Then it follows from Bak \cite{bak_1976} that $L^s_5(\Z[\Gamma])=0$.  Hence, for $\pi\cong \Gamma$ we have that the surgery obstruction vanishes trivially and hence \zcref{thm:unstable_realisation_theorem} and \zcref{thm:non_pseudo_smoothable_but_hom_smoothable} apply.  We can say more, however.

\begin{prop}\label{prop:sk_1=0_realisation}
	Let $X$ be a compact, connected, smooth, orientable $4$-manifold with $\pi_1(X)\cong \Gamma$ with $\Gamma$ a finite group such that $SK_1(\Z[\Gamma])=0$ or such that $SK_1(\Z[\rho])=0$ where $\rho$ denotes the Sylow 2-subgroup of $\Gamma$.  Then \zcref{thm:unstable_realisation_theorem} and \zcref{thm:non_pseudo_smoothable_but_hom_smoothable} both apply to $X$.
\end{prop}

\begin{rmk}\label{rmk:sk_1}
	For a group ring $\Z[\Gamma]$ we define $SK_1(\Z[\Gamma])$ to be a certain subgroup of the algebraic $K$-theory group $K_1(\Z[\Gamma])$.  More specifically, it is defined as the kernel of the inclusion induced map $K_1(\Z[\Gamma])\to K_1(\Q[\Gamma])$.  If $\Gamma$ is abelian, then this is equivalently defined as the kernel of the determinant map $\det\colon K_1(\Z[\Gamma])\to \Z[\Gamma]^\times$ (see \cite[Page 2]{oliver_1988}.
\end{rmk}

\zcref{prop:sk_1=0_realisation} follows from a result of Hambleton-Milgram-Taylor-Williams, which we now restate the relevant part of (adapted to our situation).

\begin{thm}[{\cite[Theorem A]{hambleton_milgram_taylor_williams_1988}}]\label{thm:hambleton_et_al}
	Let $M$ be a $5$-dimensional manifold with boundary $\partial M$ such that $\pi_1(M)\cong \pi$ is finite and $\Image(SK_1(\Z[\rho])\to SK_1(\Z[\pi]))=0$, where $\rho$ is the Sylow 2-subgroup of $\pi$.  Then the surgery obstruction map 
	\[
	\theta\colon [(M,\partial M),G/\TOP] \to L_5^s(\Z[\pi])
	\]
	is given by
	\[
	x \mapsto \kappa^s_3(c_*(\Arf_3(x))).
	\]
\end{thm}
In the above theorem, $\kappa^s_3$ denotes the map (constructed in \cite[$\S1$]{hambleton_milgram_taylor_williams_1988})
\[
\kappa^s_3\colon H_3(\mathcal{B}\pi;\Z/2)\to L^s_5(\Z[\pi]),
\]
where $c\colon M\to \mathcal{B}\pi$ is defined as the classifying map for the universal cover of $M$, and $\Arf_3$ denotes the (suitable) Arf invariant.

Before we prove \zcref{prop:sk_1=0_realisation}, we state a few groups $\Gamma$ for which $SK_1(\Gamma)$ vanishes.  This is the work of many mathematicians, and we direct the reader to \cite[p.3-4]{oliver_1988} and the citations within for references.
\begin{enumerate}[(i)]
	\item $\Gamma = \Z/2n$
	\item $\Gamma = \Z/2^n \times \Z/2$
	\item $\Gamma = (\Z/2)^n$.
	\item $\Gamma = D_{2n}$, the dihedral group of order $4n$.
\end{enumerate}
So \zcref{prop:sk_1=0_realisation} applies to any of the above groups.  It also follows that \zcref{thm:unstable_realisation_theorem} and \zcref{thm:non_pseudo_smoothable_but_hom_smoothable} also apply to any $\Gamma$ that has any of the above groups as its Sylow 2-subgroup.  For example, let $\rho$ be isomorphic to one of the above groups and let $O$ be an odd-order group.  Then $\Gamma:=\rho\times O$ has $\rho$ as its Sylow 2-subgroup and hence \zcref{prop:sk_1=0_realisation} also applies to $\Gamma$.

\begin{proof}[Proof of \zcref{prop:sk_1=0_realisation}]
	Let $X$ satisfy the hypotheses of \zcref{prop:sk_1=0_realisation} and let~$b\in H^4(X\times I;\Z)$.  We need to show that there exists an $a\in H^2(X\times I;\Z/2)$ such that $\theta(a,b)=0$.  We claim that we can simply take $a=0$ regardless of $b$.
	
	Note that $X\times I$ is a $5$-dimensional manifold with boundary and by assumption we have that $\Image(SK_1(\Z[\rho])\to SK_1(\Z[\Gamma]))=0$, hence \zcref{thm:hambleton_et_al} applies to $X\times I$.  \zcref{thm:hambleton_et_al} then tells us that the surgery obstruction map $\theta$ factors as
	\[
	[(X\times I,\partial),G/\TOP] \to H^2(X\times I,\partial; \Z/2)\cong H_3(X;\Z/2)\to H_3(\mathcal{B}\pi;\Z/2) \xrightarrow{\kappa_3^s} L_5^s(\Z[\pi]),
	\]
	where the first map is the projection, using \zcref{lem:normal_invariants_isomorphism}.  Hence, $\theta(a,b)$ does not depend on $b$, and so for all~$b\in H^4(X\times I,\partial;\Z/2)$ we have that $\theta(0,b)=\theta(0,0)=0$.\end{proof}

\subsection{Partial unstable realisation of the Casson-Sullivan invariant}\label{sbs:partial_realisation}

The purpose of this subsection is to find examples where we can partially realise the Casson-Sullivan invariant using the method described in \zcref{sbs:forming_mapping_cylinders}.  The way we will do this is by considering the assembly maps for the surgery obstruction map $\theta$ and comparing them using spectral sequences to surgery obstruction maps that we have full realisation for.  In all that follows, let $X$ be a closed, connected, smooth, orientable $4$-manifold with $\pi_1(X)\cong\pi$.

First we give the relevant notation.  Let $\varepsilon\in\{-\infty\}\cup\{\dots,-1, 0, 1, 2\}$ be the decoration, and recall that $\varepsilon=2$ refers to $\varepsilon=s$ and $\varepsilon=1$ refers to $\varepsilon=h$.  For more information about decorations, see \cite{luck_2023}.  Let $\Lg^\varepsilon_\bullet(R)$ denote the quadratic $L$-theory spectrum of a ring $R$ with decoration $\varepsilon$ (this is a spectrum that has homotopy groups $\pi_k(\Lg^\varepsilon_\bullet(R))=L^\varepsilon_k(R)$) and let $\Lg_\bullet:=\Lg_\bullet(\Z)$ (note that, as suggested by the notation, this is independent of our choice of decoration). We will use $\Lg_\bullet\langle1\rangle$ to denote the $1$-connective cover of $\Lg_\bullet$ (a spectrum which has $\pi_k(\Lg_\bullet\langle1\rangle)=0$ for all $k<1$ and has $\pi_k(\Lg_\bullet\langle1\rangle)=\pi_k(\Lg_\bullet)$ for $k\geq 1$).  Then these spectra determine a generalised (co)homology theory and we have the following factorisation of the surgery obstruction map $\theta^\varepsilon$ (note that the notation used will be explained below).
\[
\begin{tikzcd}
	\mathcal{N}_{\TOP}(X\times I,\partial) \arrow[r,"\theta^\varepsilon"] \arrow[d,"\cong", "a"'] & L_5^\varepsilon(\Z[\pi]) \\
	H^0(X\times I,\partial;\Lg^\varepsilon_\bullet\langle1\rangle) \arrow[d,"\cong", "b"'] & H_5^\pi(\ast_{\mathcal{ALL}};\Lg^\varepsilon_\bullet) \arrow[u,"\cong", "g"'] \\
	H_5(X\times I;\Lg^\varepsilon_\bullet\langle1\rangle) \arrow[d, "c"'] & H_5^\pi(\ast_{\mathcal{TRIV}};\Lg^\varepsilon_\bullet) \arrow[u, "\sigma^\varepsilon"'] \\
	H_5(\mathcal{B}\pi;\Lg^\varepsilon_\bullet\langle1\rangle) \arrow[r, "d"'] & H_5(\mathcal{B}\pi;\Lg^\varepsilon_\bullet) \arrow[u,"\cong", "e"']
\end{tikzcd}
\]
The fact that $\theta^{\varepsilon}$ factors first as maps $a$ then $b$ is due to Ranicki and Quinn \cite{quinn_1970, ranicki_1992}.  Further factoring through the homology of $\mathcal{B}\pi$ is contained in \cite[Appendix]{hambleton_milgram_taylor_williams_1988}, though they attribute this to \cite{ranicki_1992}.  The denoted even further factorisation of the assembly map is described in \cite{davis_luck_1998}.  We now describe the maps.

The isomorphism $a$ arises via the definition of cohomology with coefficients in a spectrum: since we have that $(\Lg_\bullet\langle 1\rangle)_0\simeq G/\TOP$ we have that \[H^0(X\times I,\partial; \Lg_\bullet)\cong [X\times I, \partial; G/\TOP]\cong \mathcal{N}_{\TOP}(X\times I,\partial).\]  The isomorphism $b$ is given by Ranicki-Sullivan duality.  To define the map $c$, we factor it as the composition of maps given in the diagram below where $\widetilde{X\times I}$ denotes the universal cover of $X\times I$.\[
\begin{tikzcd}
	H_5(X\times I;\Lg^\varepsilon_\bullet\langle1\rangle) \arrow[d,"c"] \arrow[r,"\cong"] & H^\pi_5(\widetilde{X\times I};\Lg^\varepsilon_\bullet\langle1\rangle) \arrow[r,"\cong"] & H_5(\mathcal{E}\pi\times_\pi (\widetilde{X\times I});\Lg^\varepsilon_\bullet\langle1\rangle) \arrow[ld, bend left, end anchor={east}, "c'"'] \\ H_5(\mathcal{B}\pi;\Lg^\varepsilon_\bullet\langle1\rangle) & H_5(\mathcal{E}\pi\times_\pi \{\pt\};\Lg^\varepsilon_\bullet\langle1\rangle) \arrow[l,"\cong"] & 
\end{tikzcd}
\]To define the map $d$, recall that we have a map $\Lg^\varepsilon_\bullet\langle1\rangle\to \Lg^\varepsilon_\bullet$ by the definition of a $1$-connective cover.  Then $d$ is the induced map on homology via this map.  The definition of the isomorphism $e$ follows from the definition of the $\pi$-equivariant homology of the $\Orb(\pi)$-space $\ast_{\mathcal{TRIV}}$ with coefficients in a spectrum.  For this details on this, see \cite[Ex 5.5]{davis_luck_1998}.  The isomorphism $g$ is again given by the definition of the $\pi$-equivariant homology.  The map $\sigma^\varepsilon$ is what we will call the $\varepsilon$\emph{-assembly map}.  Since we need to use the $s$-cobordism theorem to construct our mapping cylinders, we will need to consider the case $\varepsilon=s$, for which this assembly map is certainly not an isomorphism in general.  Composing all of these maps together we get a factorisation of $\theta^s$.  

Now we want to use this formalism to try to understand $\theta^s$ for some groups for which we cannot use \zcref{prop:sk_1=0_realisation}.  Let $\Gamma$ be a finite cyclic group such that we have a non-trivial homomorphism $h\colon\Gamma\to \pi$.  In what follows we will condense the `assembly map' by setting~$\ol{\sigma}:=g\circ \sigma^s\circ e$.  

Consider the following diagram.\[
\begin{tikzcd}
	\mathcal{N}_{\TOP}(X\times I,\partial) \arrow[r,"\theta^s"] \arrow[d,"\cong", "a"'] & L_5^s(\Z[\pi])  & L_5^s(\Z[\Gamma])=0 \arrow[l]\\
	H_5(X\times I; \Lg^s_\bullet\langle1\rangle) \arrow[r] & H_5(\mathcal{B}\pi;\Lg^s_\bullet) \arrow[u,"\ol{\sigma}"] & H_5(\mathcal{B}\Gamma;\Lg^s_\bullet) \arrow[l,"h_*"] \arrow[u,"\ol{\sigma}"]
\end{tikzcd}\]By naturality of the assembly maps, if a normal invariant is hit by the map $h_*$ after being mapped down to $H_5(\mathcal{B}\pi;\Lg^s_\bullet)$, then its surgery obstruction must be zero.  It follows that if the map $h_*$ is surjective then $\theta^s$ is the zero map.  

Everything we have presented so far could in principle be used to study any case where we have a non-trivial homomorphism $\Gamma\to \pi$, but we now specialise to the case where $\pi=\Z\times\Gamma$, where we will be able to make concrete findings.

To try to understand this map $h_*$ better we will use the Atiyah-Hirzebruch spectral sequence (AHSS).  Recall that the AHSS is a homology spectral sequence that computes the generalised homology of a space in terms of the regular homology of that space with coefficients in the generalised homologies of a point.  In our particular case it computes, for a space~$K$,
\[
E^2_{p,q} = H_p(K;\pi_q(\Lg_\bullet\langle1\rangle)) \implies E^\infty_{p,q}=H_{p+q}(K;\Lg_\bullet\langle1\rangle).
\]
Recall that $\pi_q(\Lg_\bullet\langle1\rangle)=L_q(\Z)$ for $q>0$ and is zero otherwise.  We first wish to compute $H_5(\mathcal{B}G;\Lg_\bullet\langle1\rangle)$ for the cases $G=\pi$ and $G=\Gamma$.  We show the relevant terms below of the $E^3$-page, along with the relevant third-page differentials, in \zcref{fig:spectral_sequence}.

\begin{figure}
	\centering
	\caption{The $E^3$-page of the spectral sequence for computing $H_*(\mathcal{B}G;\Lg_\bullet\langle1\rangle)$ with the relevant terms and differentials shown for computing $H_5(\mathcal{B}G;\Lg_\bullet\langle1\rangle)$.}\label{fig:spectral_sequence}
	\begin{tikzpicture}[scale=0.78,on top/.style={preaction={draw=white,-,line width=#1}},
		on top/.default=4pt]
		
		\draw[step=3,black] (0,0) grid (18.2, 9.2); 
		
		\draw (0,1.5) -- (18.2,1.5); 
		\draw (0,4.5) -- (18.2,4.5);
		\draw (0,7.5) -- (18.2,7.5);
		
		\node at (1.5,-0.5) {$0$}; 
		\node at (4.5,-0.5) {$1$};
		\node at (7.5,-0.5) {$2$};
		\node at (10.5,-0.5) {$3$};
		\node at (13.5,-0.5) {$4$};
		\node at (16.5,-0.5) {$5$};
		\node at (18,-0.5) {$p$};
	
		\node at (-0.5,0.75) {$0$}; 
		\node at (-0.5,2.25) {$1$};
		\node at (-0.5,3.75) {$2$};
		\node at (-0.5,5.25) {$3$};
		\node at (-0.5,6.75) {$4$};
		\node at (-0.5,8.25) {$5$};
		\node at (-0.5,9) {$q$};

		\node at (1.5,0.75) {$0$}; 
		\node at (4.5,0.75) {$0$};
		\node at (7.5,0.75) {$0$};
		\node at (10.5,0.75) {$0$};
		\node at (13.5,0.75) {$0$};
		\node at (16.5,0.75) {$0$};
		
		\node at (1.5,2.25) {$0$}; 
		\node at (4.5,2.25) {$0$};
		\node at (7.5,2.25) {$0$};
		\node at (10.5,2.25) {$0$};
		\node at (13.5,2.25) {$0$};
		\node at (16.5,2.25) {$0$};
		
		\node at (10.5,3.75) {$H_3(\mathcal{B}G;\Z/2)$}; 
		\node at (13.5,3.75) {$H_4(\mathcal{B}G;\Z/2)$}; 
		
		\node at (1.5,5.25) {$0$}; 
		\node at (4.5,5.25) {$0$};
		\node at (7.5,5.25) {$0$};
		\node at (10.5,5.25) {$0$};
		\node at (13.5,5.25) {$0$};
		\node at (16.5,5.25) {$0$};
		
		\node at (1.5,6.75) {$H_0(\mathcal{B}G;\Z)$}; 
		\node at (4.5,6.75) {$H_1(\mathcal{B}G;\Z)$};
		
		\node at (1.5,8.25) {$0$}; 
		\node at (4.5,8.25) {$0$};
		\node at (7.5,8.25) {$0$};
		\node at (10.5,8.25) {$0$};
		\node at (13.5,8.25) {$0$};
		\node at (16.5,8.25) {$0$};
				
		\draw [thick, -latex, red](9.15,4.05) [on top]-- (2.6,6.5);
		\draw [thick, -latex, red](12.15,4.05) [on top] -- (5.6,6.5);
		\node at (5.2,4.9) {$\boldsymbol{d_3^{3,2}}$};
		\node at (9.9,5.5) {$\boldsymbol{d_3^{4,3}}$};
		\draw [thick, -latex] (18.5,-0.5) -- (19,-0.5);
		\draw [thick, -latex] (-0.5, 9.5) -- (-0.5,10);
	\end{tikzpicture}
\end{figure}
The conclusions we can draw from this spectral sequence are contained in the following lemma.

\begin{lem}\label{lem:spectral_sequence_commutative_diagram}
	Let $\Gamma$ be a finite cyclic group of order $2n$, let $\pi\cong \Z\times\Gamma$ and let $d_G$ denote the differential $d^{4,3}_{3}$ denoted in \zcref{fig:spectral_sequence}.  Then we have the following commutative diagram where the rows are exact.
	\begin{equation*}
		\begin{tikzcd}
			0 \arrow[r] & \coker d_\Gamma \arrow[r] \arrow[d] & H_5(\mathcal{B}\Gamma;\Lg_\bullet\langle1\rangle) \arrow[r] \arrow[d] & H_3(\mathcal{B}\Gamma;\Z/2) \arrow[r] \arrow[d] & 0 \\
			0 \arrow[r] & \coker d_\pi \arrow[r] & H_5(\mathcal{B}\pi;\Lg_\bullet\langle1\rangle) \arrow[r] & H_3(\mathcal{B}\pi;\Z/2) \arrow[r] & 0 
		\end{tikzcd}
	\end{equation*}
	and $H_3(\mathcal{B}\Gamma;\Z/2)\cong \Z/2$,  $\coker d_\Gamma \cong \Z/2n$ or $\Z/n$, $H_3(\mathcal{B}\pi;\Z/2)\cong \Z/2\oplus \Z/2$, and $\coker d_\pi \cong \Z\oplus \Z/2n$ or $\Z\oplus\Z/n$.
\end{lem}

\begin{proof}
	First note that for any group $G$ the differential $d^{3,2}_3$ vanishes since it is map from a torsion group to a torsion-free group.  Then recall that a model for $\mathcal{B}\Gamma$ is given by the infinite lens space $L(2n)$, and that a model for $\mathcal{B}\pi$ is given by the space $S^1\times L(2n)$.  We have that $d_\Gamma\colon \Z/2\to \Z/2n$ is either the zero map or is injective.  Similarly, we get that $d_\pi\colon \Z/2\oplus \Z/2 \to \Z\oplus \Z/2n$ is either the zero map or the map $(a,b)\mapsto (0,nb)$ (the other possibilities may be ruled out by further comparing with the spectral sequence for $G=\Z$ and using naturality).  From this we can conclude that the isomorphism types of the cokernels are as described in the lemma.
	
	Finally, the diagonal $p+q=5$ computes the associated graded for $H_5(\mathcal{B}G;\Lg_\bullet\langle1\rangle)$ and hence we have the short exact sequences which form the rows of the stated commutative diagram.  The vertical maps are then induced by the obvious inclusion map $\Gamma\to \pi$ and the diagram commutes by naturality of spectral sequences.  This completes the proof.
\end{proof}

We have that $H^4(X\times I,\partial;\Z/2)\cong H^3(X;\Z/2)\cong \Z/2\oplus \Z/2$ (where the first summand is generated by the Poincar\'{e} dual to a curve representing the generator of $\Z\subset\Z\times\Gamma$ and the second is generated by the Poincar\'{e} dual to a curve representing the generator of $\Z\subset\Z\times\Gamma$) and the aim is now to realise the element $(0,1)$ as the Casson-Sullivan invariant of a homeomorphism $f\colon X\to X$.  Following the realisation procedure laid out in \zcref{sbs:forming_mapping_cylinders} and \zcref{sbs:realisation_proofs}, by \zcref{prop:mapping_cylinder_construction} and \zcref{lem:mapping_cylinder_casson_sullivan} it suffices to show that there exists an element $y\in \mathcal{N}_{\TOP}(X\times I,\partial)\cong [(X\times I,\partial),G/\TOP]$ such that the map $m$ from \zcref{lem:exact_sequence_commutative_diagram} sends $y$ to $(0,1)\in H^4(X\times I,\partial;\Z/2)$ and such that $\theta(y)=0$.

\begin{lem}\label{lem:spectral_sequence_vertical_map}
	Let $\Gamma$ be a finite cyclic group of order $2n$, $\pi=\Z\times\Gamma$ and let $\iota\colon \Gamma\to\pi$ be the obvious inclusion map.  The leftmost vertical map in the commutative diagram from \zcref{lem:spectral_sequence_commutative_diagram} is then $\iota_*$ and is given by
	\begin{align*}
		\iota_*\colon \begin{cases} \Z/2n \\ \Z/n\end{cases} \to& \begin{cases} \Z\oplus \Z/2n \\ \Z\oplus\Z/n \end{cases} \\
		a \mapsto& (0,a).
	\end{align*}
\end{lem}

\begin{proof}
	We see this from the following commutative diagram, where we have identified $H_1(\mathcal{B}G;\Z)\cong \ab(G)$ for any group $G$, where $\ab(G)$ denotes the abelianisation of $G$.
	\[\begin{tikzcd}
		\Gamma \arrow[d,"\iota"] \arrow[r,"\cong"] & \ab(\Gamma) \arrow[d] \arrow[r] & \coker d_\Gamma \arrow[d, "\iota_*"] \\
		\pi  \arrow[r,"\cong"] & \ab(\pi)  \arrow[r] & \coker d_\pi 
	\end{tikzcd}\]
	Since $\iota(a)=(0,a)$, the commutativity of the above diagram gives the result regardless of the isomorphism types of the cokernels given in \zcref{lem:spectral_sequence_commutative_diagram}.
\end{proof}

Before the final proposition of this section, we need a technical lemma concerning how Poincar\'{e} duality interacts with homology in $L$-theory coefficients.

\begin{lem}\label{lem:duality}
	Let $X$ be a compact 4-manifold and set $\pi:=\pi_1(X)$.  Then the following diagram commutes.
	\[
	\begin{tikzcd}
		H^4(X\times I,\partial;\Z) \oplus H^2(X\times I,\partial; \Z/2) \arrow[r,"\cong"] \arrow[d,"\PD\oplus \PD","\cong"'] & H^0(X\times I,\partial; \Lg_\bullet\langle1\rangle) \arrow[d, "\RS","\cong"'] \\
		H_1(X;\Z)\oplus H_3(X; \Z/2) \arrow[r,"\cong"] & H_5(X;\Lg_\bullet\langle1\rangle)
	\end{tikzcd}
	\]
	where $\RS$ denotes the Ranicki-Sullivan duality map.
\end{lem}

\begin{proof}
	This is essentially deducible from \cite{taylor_williams_1979}.  We follow \cite{orson_powell_randal-williams_2025} where the details are spelled out more clearly, but this result is also still implicit there.  We give the necessary observations needed to deduce the above result.  For higher dimensional manifolds the above square generally does not need to commute.  The correct commutative square is given in \cite[Lemma 5.6]{orson_powell_randal-williams_2025} (note the difference in notation used to denote (co)homology with coefficients in $L$-theory, and also that we do not need to take 2-local coefficients due to our being in low dimensions; see \zcref{lem:normal_invariants_isomorphism}).  The failure for the above square to commute is measured in the square in \cite[Lemma 5.6]{orson_powell_randal-williams_2025} by an isomorphism
	\[
	\Phi \colon H^4(X\times I,\partial;\Z)\oplus H^2(X\times I,\partial;\Z/2)\xrightarrow{\cong} H^4(X\times I,\partial;\Z)\oplus H^2(X\times I,\partial;\Z/2),
	\]
	where the above square commutes if and only if $\Phi$ is the identity map.  This is true in our situation because one can verify that all of the extra terms in the definition of $\Phi$ vanish for degree reasons.
\end{proof}

\begin{prop}\label{prop:partial_realisation}
	Let $\Gamma$ be a cyclic group of order $2n$ and let $X$ be a closed, connected, smooth, orientable $4$-manifold with $\pi_1(X)\cong \pi = \Z\times \Gamma$.  Then $H^3(X;\Z/2)\cong \Z/2\oplus \Z/2$ (under the above identification) and there exists a self-homeomorphism $f\colon X\to X$ such that $\cs(f)=(0,1)\in H^3(X;\Z/2)$.
\end{prop}

\begin{proof}
	First, note that we can similarly use the AHSS to compute $H_5(X\times I;\Lg_\bullet\langle1\rangle)$.  Of course, we already know the answer since $H_5(X\times I;\Lg_\bullet\langle1\rangle)\cong H_1(X;\Z)\oplus H_3(X;\Z/2)$ by \zcref{lem:normal_invariants_isomorphism}, but this means that we can fit this data into a larger commutative diagram with the one from \zcref{lem:spectral_sequence_commutative_diagram}, which we do now below.
	\begin{equation*}
		\begin{tikzcd}
			 \coker d_\Gamma \arrow[rr, tail] \arrow[dd,"\iota_*"] && H_5(\mathcal{B}\Gamma;\Lg_\bullet\langle1\rangle) \arrow[rr,twoheadrightarrow] \arrow[dr,"\theta^s"] \arrow[dd, "\iota_*"] && H_3(\mathcal{B}\Gamma;\Z/2) \arrow[dd,"\iota_*"]  \\
			&&& L_5^s(\Z[\Gamma]) \arrow[dd] & \\
			 \coker d_\pi \arrow[rr,tail] && H_5(\mathcal{B}\pi;\Lg_\bullet\langle1\rangle) \arrow[dr,"\theta^s"]\arrow[rr,twoheadrightarrow, crossing over] && H_3(\mathcal{B}\pi;\Z/2)  \\
			&&& L_5^s(\Z[\pi]) &\\
			 H_1(X;\Z) \arrow[rr,tail] \arrow[uu] && H_5(X\times I;\Lg_\bullet\langle1\rangle) \arrow[uu, "a"] \arrow[rr,twoheadrightarrow] && H_3(X;\Z/2) \arrow[uu] 
		\end{tikzcd}
	\end{equation*}

	Note that we have omitted the zeroes from the ends of the rows, but the rows are still split exact sequences.  Let $x\in H_5(X\times I;\Lg_\bullet\langle1\rangle)$ be the element corresponding to \[((0,1),0)\in H^4(X\times I, \partial;\Z)\oplus H^2(X\times I, \partial; \Z/2)\] via the isomorphism from \zcref{lem:normal_invariants_isomorphism}.  If we can show that the surgery obstruction $\theta^s(x)=0$, then by \zcref{lem:exact_sequence_commutative_diagram}, \zcref{prop:mapping_cylinder_construction} and \zcref{lem:mapping_cylinder_casson_sullivan} we can construct a homeomorphism $f\colon X\to X$ which has $\cs(f)=m((0,1),0)=(0,1)$.
	
	From \zcref{lem:duality} it follows that $x$ is hit by an element from $H_1(X;\Z)$.  Using this fact and the description of the leftmost map denoted by $\iota_*$ from \zcref{lem:spectral_sequence_vertical_map}, a simple diagram chase tells us that $a(x)$ is in the image of~$\iota_*$.  Hence by naturality of assembly maps, and the fact that $L_5^s(\Z[\Gamma])=0$, we have that $\theta(x)=0$, completing the proof. 
\end{proof}

\zcref{prop:partial_realisation} tells us that we can also find examples of homeomorphisms that are not stably (pseudo-)smoothable but are homotopic to the identity for $4$-manifolds with fundamental group~$\Z\times \Z/2n$.  Note also that the above arguments would have worked just as well with $\Z$ replaced with $\Z^n$, so we can also find non-smoothable homeomorphisms in those cases as well.

\section{An application to stable isotopy of surfaces}\label{sec:stable_isotopy}

This section is devoted to proving \zcref{thm:stable_isotopy}.  We begin with a definition.

\begin{dfn}\label{dfn:isotopy_surfaces}
	Let $X$ be a connected, compact, orientable smooth $4$-manifold and let $\Sigma_1,\Sigma_2\subset X$ be a pair of smoothly embedded surfaces, such that $\partial \Sigma_1 = \partial \Sigma_2 = L\subset \partial X$ a fixed link in $\partial X$ (which may be disconnected).  We say that $\Sigma_1$ and $\Sigma_2$ are \emph{topologically isotopic} \{smoothly isotopic\} if there exists a homeomorphism \{diffeomorphism\} of pairs $F\colon (X,\Sigma_1)\to (X,\Sigma_2)$ such that $F$ is isotopic \{smoothly isotopic\} to the identity.  We say that $\Sigma_1$ and $\Sigma_2$ are \emph{externally stably smoothly isotopic} if there exists $n\geq 0$ such that $\Sigma_1$ and $\Sigma_2$ become smoothly isotopic in $X\#(\#^n S^2\times S^2)$, where we perform the connected-sums in the complement of $\Sigma_1\cup\Sigma_2$.
\end{dfn}

The idea of proving \zcref{thm:stable_isotopy} is as follows: first, we find a stable diffeomorphism between the exteriors of the surfaces.  Then, we show that we can modify this diffeomorphism to find one such that when we glue back in the tubular neighbourhoods of the surfaces, we have a diffeomorphism which is the identity on $H_2(X;\Z)$.  Then we use the following result of Saeki and Orson-Powell which says when a diffeomorphism of a simply-connected $4$-manifold is stably isotopic to the identity.

\begin{thm}[\cite{saeki_2006,orson_powell_2025}]\label{thm:saeki_stable_isotopy}
	Let $X$ be a connected, compact, simply-connected smooth $4$-manifold with $($potentially disconnected$)$ boundary $\partial X$ and let $f\colon X\to X$ be a diffeomorphism (restricting to the identity on $\partial X$) such that the variation $($see \cite[Definition 2.11]{orson_powell_2025}$)$ induced by $f$ is trivial, and such that the induced map on relative spin structures $($see \cite[Definition 2.5]{orson_powell_2025}$)$ is trivial.  Then $f$ is stably smoothly isotopic to the identity.
\end{thm}

The above theorem crucially relies on the theorem of Quinn \cite{quinn_1986} that smooth pseudo-isotopy implies stable smooth isotopy (see also \cite[Section 2]{gabai_2022}, and \cite{gabai_gay_hartman_krushkal_powell_2026}).

Since in \zcref{thm:stable_isotopy} we assume that our surfaces are topologically isotopic, this means that we have a homeomorphism of the surface exteriors, and we will take this as our starting point.  First, however, we need to arrange that our homeomorphism is smooth near the surfaces.

\begin{lem}\label{lem:smooth_near_surface}
	Let $X$, $\Sigma_1$ and $\Sigma_2$ be as in the statement of \zcref{thm:stable_isotopy} and let $\widehat{G}\colon X\to X$ be a homeomorphism which sends $\Sigma_1$ to $\Sigma_2$.  Then $\widehat{G}$ is isotopic relative $\partial X \cup \Sigma_1$ to a homeomorphism which sends $\nu\Sigma_1$ to $\nu\Sigma_2$ by a diffeomorphism.
\end{lem}

\begin{proof}
	Let $U:=\widehat{G}(\nu\Sigma_1)$.  We begin by isotoping $\widehat{G}\vert_{\Sigma_1}$ to a diffeomorphism, which is always possible since homeomorphisms of surfaces are isotopic to diffeomorphisms \cite{epstein_1966} (see also \cite{hatcher_2014}).  This isotopy extends to $\widehat{G}$ by extending it first to a tubular neighbourhood (perform the isotopy less and less as you extend radially from $\Sigma_1$) and then as the constant isotopy on the complement of the tubular neighbourhood of $\Sigma_1$ (of course, it was clear, by the isotopy extension theorem \cite{edwards_kirby_1971}, that this extended, but we can see the extension explicitly and easily in this case).  Denote the result of this isotopy still by $\widehat{G}$.  By the uniqueness of topological tubular neighbourhoods \cite[Chapter 9.3]{freedman_quinn_1990}, $U$ and $\nu\Sigma_2$ are isotopic by an isotopy that fixes $\Sigma_2$, and hence we can isotope $\widehat{G}$ relative to $\partial X\cup\Sigma_1$ such that $\widehat{G}(\nu\Sigma_1)=\nu\Sigma_2$ as bundles.  We now smooth the map on the fibres of the normal bundles relative to the $0$-section $\Sigma_1$, which we can do since any map $\Sigma_1 \to \tO(2)$ can be isotoped to a smooth map.  The result of this isotopy now sends $\nu\Sigma_1$ to $\nu\Sigma_2$ via a diffeomorphism.
\end{proof}

\begin{rmk}\label{rmk:little_hauptvermutung}
	It should be noted that one cannot always smooth homeomorphisms of $4$-manifolds near arbitrary codimension $2$ submanifolds.  In general, one can only smooth a homeomorphism near a codimension $2$ submanifold after a small topological isotopy (see \cite[Theorem 8.1A]{freedman_quinn_1990}).  \zcref{lem:smooth_near_surface} does not contradict this since we have a much stronger hypothesis (it is a self-homemorphism of a smooth manifold and our homeomorphism already maps the submanifold in question to another smooth submanifold).
\end{rmk}

The following lemma will handle most of the technical aspects of the proof of \zcref{thm:stable_isotopy}.

\begin{lem}\label{lem:exteriors_stable_diffeo}
	Let $X$, $\Sigma_1$ and $\Sigma_2$ be as in the statement of \zcref{thm:stable_isotopy} and let $\widehat{G}\colon X\to X$ be a homeomorphism restricting to the identity on $\partial X$ which sends $\Sigma_1$ to $\Sigma_2$.  Further let $G\colon X\sm\nu\Sigma_1 \to X\sm\nu\Sigma_2$ be the restriction of $\widehat{G}$ to the surface exteriors.  Then for some $k\geq 0$ there exists a diffeomorphism \[F\colon (X\sm\nu\Sigma_1)\#(\#^k S^2\times S^2) \to (X\sm\nu\Sigma_2)\#(\#^k S^2\times S^2)\] which restricts to the identity on $\partial X$.  Furthermore, $F$ extends to a diffeomorphism \[\widehat{F}\colon X\#(\#^k S^2\times S^2)\to X\#(\#^k S^2\times S^2),\] and the induced maps on second homology fit into the following commutative diagram.
	\begin{equation}\label{eq:stable_isotopy}
		\begin{tikzcd}
			H_2(X\sm\nu\Sigma_1)\oplus \Z^2 \oplus \Z^{2k-2} \arrow[d,"\cong"] \arrow[r,"{(G_*,A,\Id)}"] & H_2(X\sm\nu\Sigma_2)\oplus \Z^2 \oplus \Z^{2k-2} \arrow[d,"\cong"] \\
			H_2((X\sm\nu\Sigma_1)\#(\#^k S^2\times S^2)) \arrow[d,"(i_1)_*"] \arrow[r,"F_*"] & H_2((X\sm\nu\Sigma_2)\#(\#^k S^2\times S^2)) \arrow[d,"(i_2)_*"] \\
			H_2(X\#(\#^k S^2\times S^2)) \arrow[d,"\cong"] \arrow[r,"(\widehat{F})_*"] & H_2(X\#(\#^k S^2\times S^2)) \arrow[d,"\cong"] \\
			H_2(X)\oplus \Z^2 \oplus \Z^{2k-2} \arrow[r, "{(\Id,A,\Id)}"] & H_2(X)\oplus \Z^2 \oplus \Z^{2k-2}
		\end{tikzcd}
	\end{equation}
	where $i_1$ and $i_2$ denote the inclusions of the exteriors, and $A\colon \Z^2\to \Z^2$ is either the identity map if $\cs(G)=0$, or is given by the map sending $(x,y)\mapsto (-y,-x)$ if $\cs(G)\neq 0$.  Furthermore, $\widehat{F}$ is topologically pseudo-isotopic to a map which differs from the stabilisation of $\widehat{G}$ only on a neighbourhood of a curve in $X\sm\nu\Sigma_1$ and a single $S^2\times S^2$ summand.
\end{lem}

\begin{proof}
	We assume, by \zcref{lem:smooth_near_surface}, that $\widehat{G}\colon X\to X$ already restricts to a diffeomorphism of the tubular neighbourhoods $\widehat{G}\vert_{\nu\Sigma_1}\colon\nu \Sigma_1\to \nu \Sigma_2$.
	
	We begin by showing the existence of a stable diffeomorphism of the exteriors, and that the top square in (\ref{eq:stable_isotopy}) commutes.  Assume that $\cs(G)\neq 0$, and let $\gamma\subset X\sm\nu\Sigma_1$ be a framed embedded curve such that $[\gamma]\in H_1(X\sm\nu \Sigma_1;\Z/2)$ is dual to $\cs(G)$.  By \zcref{thm:stable_realisation_techincal}, the connected-sum homeomorphism (using the notation therein) \[G':=G\#_{\gamma=\theta}\widehat{\sigma}\colon (X\sm\nu \Sigma_1)\#(S^2\times S^2) \to (X\sm\nu \Sigma_2)\#(S^2\times S^2)
	\] has $\cs(G')=\cs(G)+\cs(G)=0$, where $\widehat{\sigma}=\sigma$ or $\sigma\circ t$, depending on $G$.
	
	By \zcref{prop:cs_stable_2}, $G'$ is stably pseudo-isotopic rel.\ boundary to a diffeomorphism \[F\colon (X\sm\nu\Sigma_1)\#(\#^k S^2\times S^2) \to (X\sm\nu\Sigma_2)\#(\#^k S^2\times S^2)\] for some $k\geq 1$.  By the computation in \cite{lee_1970} (see \cite[Lemma 2.3]{galvin_2024}), $\widehat{\sigma}$ induces the map sending $(x,y)\mapsto (-y,-x)$ on $H_2((S^1\times S^3)\# (S^2\times S^2))$ and hence the top square in (\ref{eq:stable_isotopy}) commutes.
	
	If $\cs(G)=0$ then we may immediately apply \zcref{prop:cs_stable_2} and similarly obtain a diffeomorphism $F$, though potentially needing no stabilisations, and the top half of (\ref{eq:stable_isotopy}) commutes with $A$ given by the identity map.
	
	Now $F$ extends to a diffeomorphism on $X\#(\#^k S^2\times S^2)$ since we already assumed that $\widehat{G}$ restricted to a diffeomorphism $\nu\Sigma_1\to\nu\Sigma_2$ (note that $F=G$ on $\partial(X\sm\nu\Sigma_1$).  More specifically, we can fill back in the tubular neighbourhoods of the surfaces $\Sigma_1$ and $\Sigma_2$ to obtain a diffeomorphism $\widehat{F}$ which fits into the following diagram.
	\[\begin{tikzcd}
		(X\sm\nu \Sigma_1)\#(\#^k S^2\times S^2) \arrow[r,"F"] \arrow[d,"i_1"] & (X\sm\nu \Sigma_2)\#(\#^k S^2\times S^2) \arrow[d,"i_2"] \\
		X\#(\#^k S^2\times S^2) \arrow[r,"\widehat{F}"] & X\#(\#^k S^2\times S^2)
	\end{tikzcd}\]
	This means that the middle square in (\ref{eq:stable_isotopy}) must commute, and the commutativity of the bottom square then follows immediately.  The final statement is clear by the construction.
\end{proof}

\begin{proof}[Proof of \zcref{thm:stable_isotopy}]
	By \zcref{lem:exteriors_stable_diffeo}, we have that for some $k\geq 0$ there exists a diffeomorphism \[\widehat{F}\colon X\#(\#^k S^2\times S^2)\to X\#(\#^k S^2\times S^2)\] such that on second homology it induces the map \[(\Id,A,\Id)\colon H_2(X)\oplus \Z^2\oplus \Z^{2k-2} \to H_2(X)\oplus \Z^2\oplus \Z^{2k-2}\] where $A$ is either the identity map or the map sending $(x,y)\mapsto (-y,-x)$.  Since $\widehat{F}$ is an extension of a diffeomorphism of the exteriors, it must send $\Sigma_1$ to $\Sigma_2$.
	
	We want to apply \zcref{thm:saeki_stable_isotopy}, and so we need to show that the variation induced by $\widehat{F}$, as well as the induced map on relative spin structures, is trivial.  By the last statement of \zcref{lem:exteriors_stable_diffeo}, the variation induced by $\widehat{F}$ can only differ by the variation induced by $\widehat{G}$ by its action on relative homology classes which cannot be represented by a relative surface disjoint from the neighbourhood of a curve union a single ($S^2\times S^2$)-connected-summand.  Such a class $x\in H_2(X\#(\#^k S^2\times S^2),\partial)$ can be represented as $x= x'+x''$, where $x'$ is represented by a relative surface disjoint from a curve union a single ($S^2\times S^2$)-connected-summand, and $x''$ is a (potentially trivial) homology class on that ($S^2\times S^2$)-connected-summand.  Since $\widehat{G}$ is topologically isotopic to the identity, its induced variation is trivial, and hence the action of the variation induced by $\widehat{F}$ on all relative classes disjoint from a curve union a single ($S^2\times S^2$)-connected-summand is trivial.  Putting these facts together, this means that the variation induced by $\widehat{F}$ can only differ from the trivial variation if its induced map on homology is non-trivial.  
	
	If $A$ is the identity map, then $\widehat{F}$ acts trivially on homology, and hence (by the above) its induced variation is trivial.  If $A$ is not the identity map, then post-compose $\widehat{F}$ with the map \[
	a := \Id \# a' \# \Id \colon X\# (S^2\times S^2)\#(\#^k S^2\times S^2) \to X\# (S^2\times S^2)\#(\#^k S^2\times S^2),
	\] where $a'\colon S^2\times S^2\to S^2\times S^2$ is the map defined as the antipodal map on both $S^2$-factors, composed with the diffeomorphism that swaps the two $S^2$-factors, and obtain an orientation-preserving diffeomorphism $a\circ \widehat{F}$.  This new diffeomorphism still sends $\Sigma_1$ to $\Sigma_2$, since $a$ is supported away from $\Sigma_1$, but now induces the trivial map on homology.  In either case, this means the map we have created now induces the trivial variation.
	
	That the induced map on relative spin structures by $\widehat{F}$ is trivial follows again from the last statement in \zcref{lem:exteriors_stable_diffeo} and that every arc between two distinct boundary components of $X$ can be made disjoint from a curve union a single $(S^2\times S^2)$-summand.
	
	Now apply \zcref{thm:saeki_stable_isotopy} to obtain that $\widehat{F}$ is stably smoothly isotopic to the identity, and hence $\Sigma_1$ and $\Sigma_2$ are stably smoothly isotopic.
\end{proof}

As was mentioned in \zcref{rmk:stable_isotopy}, \zcref{thm:stable_isotopy} has the following consequence when paired with certain results concerning when embedded surfaces are (stably) topologically isotopic.

\begin{cor}\label{cor:stable_isotopy}
	Let $X$ be as in \zcref{thm:stable_isotopy} and assume $\partial X=\emptyset$.  Let $\Sigma_1$, $\Sigma_2$ be a pair of closed, homologous, smoothly embedded surfaces in $X$ with the same genus, such that $\pi_1(X\sm\nu\Sigma_1)\cong \pi_1(X\sm\nu\Sigma_2)$.  Then if any of the following conditions are satisfied, the surfaces are stably smoothly isotopic relative to their boundaries $($below $b_2(X)$ denotes the second Betti number of $X$ and $\sig(X)$ denotes the signature of $X$$)$.
	\begin{enumerate}[(i)]
		\item $\Sigma_1$ and $\Sigma_2$ are both spheres and $\pi_1(X\sm\nu\Sigma_1)\cong \Z/d$ for some $d \geq 0$.
		\item $\pi_1(X\sm\nu\Sigma_1)$ is trivial.
	\end{enumerate}
\end{cor}

\begin{proof}
	If case (i) is satisfied then \cite[Theorem 4.8]{hambleton_kreck_1993} or \cite[Corollary 1.3]{lee_willczynski_1990} and \cite[Addendum 2]{lee_wilczynski_1993} shows that the surfaces are stably topologically isotopic.  Then \zcref{thm:stable_isotopy} applied to the stabilisation shows that the surfaces are stably smoothly isotopic.  If case (ii) is satisfied then \cite[Theorem F]{boyer_1993} shows that the surfaces are topologically isotopic, and then applying \zcref{thm:stable_isotopy} shows that they are stably smoothly isotopic.
\end{proof}

\section{Application to isotopy classification of smooth structures}\label{sec:isotopy_smooth_structures}

This section is devoted to writing up the correspondence between non-isotopic but diffeomorphic smooth structures and non-smoothable homeomorphisms.  We will then interpret the results from the rest of the paper to show that there exist many examples of non-isotopic but diffeomorphic smooth structures on $4$-manifolds that remain non-isotopic after arbitrarily many stabilisations (in contrast to instances of this phenomenon detected via gauge theory).

\subsection{Non-isotopic but diffeomorphic smooth structures}\label{sbs:isotopy_smooth_structures}

We start with the basic definitions.

\begin{dfn}\label{dfn:diffeomorphism}
	Let $\mathscr{S}_1,\mathscr{S}_2$ be a pair of smooth structures on a topological manifold $X$, restricting to a fixed smooth structure $\mathscr{S}_\partial$ on $\partial X$.  We denote the resulting smooth manifolds from these structures by $X_{\mathscr{S}_i}$.  A \emph{diffeomorphism} $f\colon X_{\mathscr{S}_1}\to X_{\mathscr{S}_2}$ is a self-homeomorphism $f$ of $X$, restricting to the identity on $\partial X$, such that $f^*(\mathscr{S}_2)=~\mathscr{S}_1$.  We say that these two smooth structures are \emph{isotopic} \{\emph{pseudo-isotopic}\} if there exists an isotopy $f_t\colon X\to X$ \{pseudo-isotopy $F\colon X\times I\to X\times I$\} such that $f_0=\Id$, $f_1^*(\mathscr{S}_2)=\mathscr{S}_1$ and $(f_t\vert_{\partial X})^*(\mathscr{S}_\partial)=\mathscr{S}_\partial$ \{$F\vert_{X\times \{0\}}=\Id$, $(F\vert_{\partial X\times I})^*(\mathscr{S}_\partial \times I)=\mathscr{S}_\partial \times I$ and $(F\vert_{X\times\{1\}})^*(\mathscr{S}_2)=\mathscr{S}_1$\}.  Note that, in the case of isotopy, $f_t^*(\mathscr{S}_2)$ is then a $1$-parameter family of smooth structures which continuously deforms from $\mathscr{S}_2$ into $\mathscr{S}_1$. 
\end{dfn}

Throughout we will assume that $X$ already has a smooth structure on its boundary.  For all of our applications, $X$ will be a $4$-manifold, and hence its boundary admits a unique smooth structure up to isotopy.

\begin{dfn}\label{dfn:mapping_class_groups}
	Let $\Homeo(X,\partial X)$ be the group of self-homeomorphisms of $X$ restricting to the identity map on the boundary and let $\Diff(X_\mathscr{S})$ be the group of self-diffeomorphisms of $X_{\mathscr{S}}$ restricting to the identity near $\partial X$.  We topologise these using the compact-open topology and the $C^\infty$ topology, respectively.
	
	Let $\pi_0\Homeo(X,\partial X)$ \{$\widetilde{\pi}_0\Homeo(X,\partial X)$\} denote the \emph{topological mapping class group of} $X$ \{\emph{topological pseudo-mapping class group of} $X$\} which is defined as the quotient of $\Homeo(X,\partial X)$ via the equivalence relation: $f\sim g$ if $f$ is isotopic \{pseudo-isotopic\} to $g$.  We can analogously define the \emph{smooth mapping class group} \{\emph{smooth pseudo-mapping class group}\}, which we will denote by $\pi_0\Diff(X_\mathscr{S},\partial X)$ \{$\widetilde{\pi}_0\Diff(X_{\mathscr{S}},\partial X)$\}.
\end{dfn}

\begin{rmk}\label{rmk:blockeo}
	There is a separate definition for the pseudo-mapping-class group, where we instead first define groups $\widetilde{\Diff}(X_{\mathscr{S}},\partial X)$ and $\widetilde{\Homeo}(X,\partial X)$, called the \emph{block-diffeomorphism group} and \emph{block-homeomorphism group}, respectively.  These are defined as semi-simplicial complexes, such that the pseudo-mapping-class groups $\widetilde{\pi}_0\Homeo(X,\partial X)$ and $\widetilde{\pi}_0\Diff(X_{\mathscr{S}},\partial X)$ naturally occur as $\pi_0\widetilde{\Homeo}(X,\partial X)$ and $\pi_0\widetilde{\Diff}(X_{\mathscr{S}},\partial X)$, respectively.  We will not develop this viewpoint further in this paper.
\end{rmk}

There is then a continuous inclusion of sets $\Diff(X_\mathscr{S},\partial X)\hookrightarrow\Homeo(X,\partial X)$ given by forgetting the smooth structure.  This induces the following map on the level of mapping class groups \{pseudo-mapping class groups\}.
\begin{align}\label{eq:Phi}
	\begin{split}
		 \Phi\colon \pi_0\Diff(X_\mathscr{S},\partial X) &\to\pi_0\Homeo(X,\partial X), \\
		\widetilde{\Phi}\colon \widetilde{\pi}_0\Diff(X_\mathscr{S},\partial X) &\to\widetilde{\pi}_0\Homeo(X,\partial X).
	\end{split}
\end{align}
We are particularly interested in the cokernel of this map which we will write as the quotient \[
\coker\Phi=\frac{\pi_0\Homeo(X,\partial X)}{\pi_0\Diff(X_\mathscr{S},\partial X)},\ \coker\widetilde{\Phi}=\frac{\widetilde{\pi}_0\Homeo(X,\partial X)}{\widetilde{\pi}_0\Diff(X_\mathscr{S},\partial X)}
\] which corresponds to self-homeomorphisms of $X$ which are not topologically isotopic \{pseudo-isotopic\} to any self-diffeomorphism of $X_\mathscr{S}$.  We have that this is a well-defined group, by the following lemma.

\begin{lem}\label{lem:diff_homeo_normal_subgroup}
	Let $X_\mathscr{S}$ be a smooth manifold, $f\colon X\to X$ be a self-homeomorphism and let $\Phi\colon \pi_0\Diff(X_\mathscr{S},\partial X) \to\pi_0\Homeo(X,\partial X)$ \{$\widetilde{\Phi}\colon \widetilde{\pi}_0\Diff(X_\mathscr{S},\partial X) \to\widetilde{\pi}_0\Homeo(X,\partial X)$\} be the map from \ref{eq:Phi}.  Then the smooth structures $f^*(\mathscr{S})$ and $\mathscr{S}$ are isotopic \{pseudo-isotopic\} if and only if~$[f]\in\Image\Phi$\{$[f]\in\Image\widetilde{\Phi}$\}.
\end{lem}

\begin{proof}
	We show that the image of $\Diff(X_{\mathscr{S}})$ in $\Homeo(X,\partial X)$ is normal, and the lemma then follows immediately.  Let $f\colon X_\mathscr{S}\to X_\mathscr{S}$ be a diffeomorphism and let $g\colon X\to X$ be a homeomorphism.  Then we want to show that $(g^{-1}\circ f\circ g)^*(\mathscr{S})=\mathscr{S}$.  It suffices to show that $f^*(g^*(\mathscr{S}))=g^*(\mathscr{S})$, i.e.\ that $f$ is also a diffeomorphism with respect to the smooth structure induced by $g$.  This is true, since $f$ being smooth with respect to $g^*(\mathscr{S})$ is equivalent to the function $\psi\circ g\circ f\circ g^{-1}\circ\varphi^{-1}$ being a (classically) smooth map for any charts $\psi$,$\varphi$.  By maximality of smooth structures, $\psi\circ g$ and $\varphi\circ g$ are both charts for $\mathscr{S}$, and $f$ being a diffeomorphism with respect to $\mathscr{S}$ finishes the proof. 
\end{proof}

\begin{prop}\label{prop:iso_iff_smoothable}
	Let $X_\mathscr{S}$ be a smooth manifold, $f\colon X\to X$ be a self-homeomorphism and let $\Phi\colon \pi_0\Diff(X_\mathscr{S},\partial X) \to\pi_0\Homeo(X,\partial X)$ \{$\widetilde{\Phi}\colon \widetilde{\pi}_0\Diff(X_\mathscr{S},\partial X) \to\widetilde{\pi}_0\Homeo(X,\partial X)$\} be the map from \ref{eq:Phi}.  Then the smooth structures $f^*(\mathscr{S})$ and $\mathscr{S}$ are isotopic \{pseudo-isotopic\} if and only if~$[f]\in\Image\Phi$\{$[f]\in\Image\widetilde{\Phi}$\}. 
\end{prop}

\begin{proof}
	We prove this only for $\Phi$, with the proof being exactly the same for $\widetilde{\Phi}$.  
	
	The reverse implication is straightforward.  Let $f$ be a smoothable homeomorphism relative to $\mathscr{S}$, i.e.\ $f$ is topologically isotopic to a diffeomorphism $f'\colon X_{\mathscr{S}}\to X_{\mathscr{S}}$.  Then the smooth structure $f^*(\mathscr{S})$ is isotopic to $(f')^*(\mathscr{S})$, and we have $(f')^*(\mathscr{S})=\mathscr{S}$ since $f'$ is a diffeomorphism.
	
	Now for the forwards implication.  Let $f$ be such that $\mathscr{S}$ and $f^*(\mathscr{S})$ are isotopic.  Then by the definition we have that there exists a diffeomorphism $g\colon X_{\mathscr{S}}\to X_{f^*(\mathscr{S})}$ such that $g$ is topologically isotopic to the identity.  More specifically, there exists a continuous path of homeomorphisms $g_t\colon X\to X$ such that $g_0=\Id_X$ and $(g_{1}^{-1})^*(f^*(\mathscr{S}))=\mathscr{S}$.  Consider the composition homeomorphism $f_t:= f\circ g_{t}\colon X\to X$ and note that $f_0=f$ and \[
	f_1=f\circ g\colon X_{\mathscr{S}}\to X_\mathscr{S}
	\] is a diffeomorphism by construction.
\end{proof}

\begin{cor}\label{cor:iso_to_iso}
	Let $X_\mathscr{S}$ be a smooth manifold and let $\mathcal{S}(X_\mathscr{S},\partial X)$ \{$\widetilde{\mathcal{S}}(X_\mathscr{S},\partial X)$\} denote the set of~isotopy \{pseudo-isotopy\} classes of smooth structures on $X$ diffeomorphic to $\mathscr{S}$ restricting to the given smooth structure on $\partial X$.  Then there is a bijection as defined below
	\begin{align*}
		\frac{\pi_0\Homeo(X,\partial X)}{\pi_0\Diff(X_\mathscr{S},\partial X)} &\longrightarrow \mathcal{S}(X_\mathscr{S},\partial X)\\
		[f] &\longmapsto f^*(\mathscr{S}),
		\end{align*}
	\begin{align*}
		\frac{\widetilde{\pi}_0\Homeo(X,\partial X)}{\widetilde{\pi}_0\Diff(X_\mathscr{S},\partial X)} &\longrightarrow \widetilde{\mathcal{S}}(X_\mathscr{S},\partial X)\\
		[f] &\longmapsto f^*(\mathscr{S}).
	\end{align*}
\end{cor}

\begin{proof}
	The fact that the map written in the statement is well-defined and only maps the trivial element to $\mathscr{S}$ follows directly from \zcref{prop:iso_iff_smoothable}.  That the second condition is enough to ensure that the map is injective follows from \zcref{lem:diff_homeo_normal_subgroup},  That the map is surjective follows from the definition of $\mathcal{S}(X_\mathscr{S},\partial X)$ \{$\widetilde{\mathcal{S}}(X_\mathscr{S},\partial X)$\}.
\end{proof}

Note that since the quotient is a group (by \zcref{lem:diff_homeo_normal_subgroup}), this means that we get a group structure on $\mathcal{S}(X_{\mathscr{S}})$ \{$\widetilde{\mathcal{S}}(X_{\mathscr{S}})$\} by \zcref{cor:iso_to_iso}.

\begin{rmk}
	There is an alternative interpretation of this section given recently by Lin-Xie \cite{lin_xie_2023}, where they instead construct a space of smooth structures, and now interpret $\mathcal{S}(X_\mathscr{S},\partial X)$ as $\pi_0$ of this space.  We make this more precise.  We have a fibration 
	\[
	\mathcal{B}\!\Diff(X_{\mathscr{S}},\partial X) \to \mathcal{B}\!\Homeo(X,\partial X)
	\]
	induced by the inclusion and we denote the homotopy fibre of this map by $F(X_{\mathscr{S}})$.  From the long exact sequence of the fibration and the fact that $\pi_i(\mathcal{B}G)=\pi_{i-1}(G)$ for any group $G$, we get the exact sequence
	\[
	\pi_0\Diff(X_\mathscr{S},\partial X) \to \pi_0\Homeo(X,\partial X) \to \pi_0(F(X_{\mathscr{S}}))\to 0.
	\] 
	Since it is clear that the first map in this sequence is $\Phi$ we have that $\pi_0(F(X_{\mathscr{S}}))=\coker\Phi$.  We call $F(X_{\mathscr{S}})$ the \emph{space of smooth structures on} $X$ \emph{diffeomorphic to} $\mathscr{S}$, and $\coker\Phi$ corresponds to its path components.  Using block diffeomorphisms and block homeomorphisms (see \zcref{rmk:blockeo}), we can similarly define a space $\widetilde{F}(X_{\mathscr{S}})$ where $\pi_0(\widetilde{F}(X_{\mathscr{S}}))$ corresponds to $\coker\widetilde{\Phi}$.  We will not explore this interpretation further in this paper.
\end{rmk}

\subsection{Stably non-isotopic but diffeomorphic smooth structures}We now assume that $\partial X\neq \emptyset$.  In this case, all of \zcref{sbs:isotopy_smooth_structures} can be considered stably.  Let $(S^2\times S^2)_\mathscr{T}$ be $S^2\times S^2$ with smooth structure $\mathscr{T}$, the standard smooth structure on $S^2\times S^2$ given as the product of the standard smooth structures on $S^2$.  Further let $X_\mathscr{S}$ be a smooth $4$-manifold and fix a collar $C$ of $\partial X$.  Since $\partial X$ has an essentially unique smooth structure, we may assume that $\mathscr{S}$ is canonically a product on $C$.  Fix once and for all a 4-disc $D\subset C$ with $\mathscr{S}$ restricting to the standard structure on $D$.  Hence we may define a smooth structure denoted by $\mathscr{S}\#_{D}\mathscr{T}$ on the topological connected-sum $X\# S^2\times S^2$.  This motivates the following definitions.

\begin{dfn}
	Let $X_{\mathscr{S}_i}$ for $i=1,2$ be two smooth manifolds with smooth structures $\mathscr{S}_i$, respectively.  Then a \emph{stable diffeomorphism} from $X_{\mathscr{S}_1}$ to $X_{\mathscr{S}_2}$ is a diffeomorphism (in the sense of \zcref{dfn:diffeomorphism}) \[
	f\colon X_{\mathscr{S}_1}\#(\#^k (S^2\times S^2)_\mathscr{T})\to X_{\mathscr{S}_2}\#(\#^k (S^2\times S^2)_\mathscr{T})
	\] for some non-negative integer $k$.  In this case, we say that the smooth structures $\mathscr{S}_i$ are \emph{stably isotopic} \{\emph{stably pseudo-isotopic}\} if $\mathscr{S}_1\#(\#^k\mathscr{T})$ is isotopic \{pseudo-isotopic\} to $\mathscr{S}_2\#(\#^k\mathscr{T})$.
\end{dfn}

Analogously to the unstable case, we then have the stable mapping class groups.

\begin{dfn}
	Let $\pi^{\Stab(D)}_0\Homeo(X,\partial X)$ \{$\widetilde{\pi}^{\Stab(D)}_0\Homeo(X,\partial X)$\} denote the \emph{stable topological mapping class group of} $X$ \{\emph{stable topological pseudo-mapping class group of} $X$\} which is defined as the quotient of $\Homeo(X,\partial X)$ via the equivalence relation: $f\sim g$ if $f$ is stably isotopic \{pseudo-isotopic\} to $g$.  Here all stabilisations are performed using the fixed 4-disc $D$, which is assumed to be fixed by any homeomorphism/diffeomorphism.
\end{dfn}

Continuing the analogy with the unstable case, there is then the continuous inclusion of sets which induces a map of the stable mapping class groups \{pseudo-mapping class groups\}
\begin{align}\label{eq:Phi_stable}
	\begin{split}
		\Phi^{\Stab(D)}\colon \pi^{\Stab(D)}_0\Diff(X_\mathscr{S},\partial X) &\to\pi^{\Stab(D)}_0\Homeo(X,\partial X), \\
		\widetilde{\Phi}^{\Stab(D)}\colon \widetilde{\pi}^{\Stab(D)}_0\Diff(X_\mathscr{S},\partial X) &\to\widetilde{\pi}^{\Stab(D)}_0\Homeo(X,\partial X).
	\end{split}
\end{align}
Again we are interested in the cokernel of this map which we write as the quotient of the stable mapping class groups \{stable pseudo-mapping class groups\} in the obvious manner:
\[
\coker\Phi^{\Stab(D)}=\frac{\pi^{\Stab(D)}_0\Homeo(X,\partial X)}{\pi^{\Stab(D)}_0\Diff(X_\mathscr{S},\partial X)},\ \coker\widetilde{\Phi}^{\Stab(D)}=\frac{\widetilde{\pi}^{\Stab(D)}_0\Homeo(X,\partial X)}{\widetilde{\pi}^{\Stab(D)}_0\Diff(X_\mathscr{S},\partial X)}.
\]
This cokernel corresponds to self-homeomorphisms of $X$, up to stable isotopy \{pseudo-isotopy\}, which are not stably topologically isotopic \{pseudo-isotopic\} to any stable self-diffeomorphism of $X_\mathscr{S}$.  The proofs of \zcref{lem:diff_homeo_normal_subgroup}, \zcref{prop:iso_iff_smoothable} and \zcref{cor:iso_to_iso} follow through unchanged for $\Phi^{\Stab(D)}$, and hence we have the following corollary.

\begin{cor}\label{cor:stable_iso_to_iso}
	Let $X_\mathscr{S}$ be a smooth manifold and let $\mathcal{S}^{\Stab(D)}(X_\mathscr{S},\partial X)$ \{$\widetilde{\mathcal{S}}^{\Stab(D)}(X_\mathscr{S},\partial X)$\} be the set of stable isotopy classes of smooth structures on $X$ \{stable pseudo-isotopy classes of smooth structures on $X$\} stably diffeomorphic to $\mathscr{S}$, where stabilisations are taken using the disc $D$.  Then there are bijections as defined below
	\begin{align*}
		\frac{\pi^{\Stab(D)}_0\Homeo(X,\partial X)}{\pi^{\Stab(D)}_0\Diff(X_\mathscr{S},\partial X)} &\longrightarrow \mathcal{S}^{\Stab(D)}(X_\mathscr{S},\partial X)\\
		[f] &\longmapsto f^*(\mathscr{S}).
	\end{align*}
	\begin{align*}
	\frac{\widetilde{\pi}^{\Stab(D)}_0\Homeo(X,\partial X)}{\widetilde{\pi}^{\Stab(D)}_0\Diff(X_\mathscr{S},\partial X)} &\longrightarrow \widetilde{\mathcal{S}}^{\Stab(D)}(X_\mathscr{S},\partial X)\\
	[f] &\longmapsto f^*(\mathscr{S}).
	\end{align*}
\end{cor}

\subsection{Proofs of \zcref{thm:_stable_non_stably_pseudo_iso} and \zcref{thm:_unstable_non_stably_pseudo_iso}}

We now reinterpret the results from the rest of the paper in terms of smooth structures.  By \zcref{prop:cs_stable_1}, we know that the Casson-Sullivan invariant is an obstruction to a homeomorphism $f\colon X_\mathscr{S}\to X_\mathscr{S}$ being stably pseudo-isotopic to a diffeomorphism.  Hence, if $f\colon X\to X$ has $\cs(f)\neq 0$, then $f$ represents a non-trivial element in $\coker\widetilde{\Phi}^{\Stab(D)}$ and by \zcref{cor:stable_iso_to_iso} $f$ gives rise to a non-trivial element in $\mathcal{S}^{\Stab(D)}(X_\mathscr{S},\partial X)$.  We now give the formal proofs of the theorems.

\begin{proof}[Proof of \zcref{thm:_stable_non_stably_pseudo_iso}]
	We first recap the hypothesis of the theorem.  Let $X$ be a $4$-manifold such that $X=X'\#(S^2\times S^2)$ for some compact, connected, smooth, orientable $4$-manifold $X'$ and let $\mathscr{S}$ denote the smooth structure on $X$ induced from $X'$.  
	
	Let $\eta\in H^3(X,\partial X;\Z/2)$ be non-trivial.  Then by \zcref{thm:stable_realisation_theorem} there exists a homeomorphism $f\colon X\to X$ such that $\cs(f)=\eta$.  By \zcref{prop:cs_stable_1} we have that $f$ represents a non-trivial element in $\coker\widetilde{\Phi}^{\Stab(D)}$ and hence by \zcref{cor:stable_iso_to_iso} we see that $\mathscr{S}_\eta:=f^*(\mathscr{S})$ is a non-trivial element in $\widetilde{\mathcal{S}}^{\Stab(D)}(X_\mathscr{S},\partial X)$.  Repeating this for all non-trivial $\eta\in H^3(X,\partial X;\Z/2)$ generates the family of smooth structures $\{\mathscr{S}_\eta\}$ and these all must be distinct in $\widetilde{\mathcal{S}}^{\Stab(D)}(X_\mathscr{S},\partial X)$ since they correspond (via \zcref{cor:stable_iso_to_iso}) to homeomorphisms which all have different Casson-Sullivan invariants.
\end{proof}

The proof for \zcref{thm:_unstable_non_stably_pseudo_iso} is analogous.

\begin{proof}[Proof of \zcref{thm:_unstable_non_stably_pseudo_iso}]
	Again we recap the hypothesis of the theorem.  Let $X$ be a closed, connected, smooth, orientable $4$-manifold with $\pi_1(X)\cong\pi$ where $\pi$ which satisfies the Casson-Sullivan realisation condition (see \zcref{dfn:cs_realisation_condition}).  Let $X=X_{\mathscr{S}}$ for some smooth structure $\mathscr{S}$.  
	
	Let $\eta \in H^3(X,\partial X;\Z/2)$ be non-trivial.  Then by \zcref{thm:unstable_realisation_theorem} there exists a homeomorphism $f\colon X\to X$ such that $\cs(f)=\eta$.  The proof is then exactly the same as the proof of \zcref{thm:_stable_non_stably_pseudo_iso}.  
\end{proof}

\bibliographystyle{alpha}
\bibliography{bibliography.bib}

\end{document}